\documentclass[a4paper, 10pt, reqno]{amsart}

\usepackage{ amssymb, amsmath, amsthm}
\usepackage{graphicx, psfrag, setspace, subfigure, gensymb}
\usepackage{amsfonts}
\usepackage{bbm}
\usepackage{fix-cm}
\usepackage{comment}
\usepackage{enumitem}
\usepackage[mathcal]{eucal}
\usepackage{tikz,tikz-cd}
\usepackage{pgfplots}
\usetikzlibrary{pgfplots.fillbetween}
 \usetikzlibrary{tikzmark, arrows.meta}
\usetikzlibrary{matrix,arrows,decorations.pathmorphing,decorations.pathreplacing}
\tikzset{commutative diagrams/diagrams={baseline=-2.5pt},commutative diagrams/arrow style=tikz}
\let\amsamp=&

\usepackage{xcolor}
\usepackage[
  colorlinks=true,
  citecolor=blue!50!black,
  linkcolor=purple!50!black  
]{hyperref}

\usepackage{float}
\usepackage{setspace} 
\usepackage{enumitem}
\usepackage[
  backend=biber,
  style=alphabetic
]{biblatex}
\usepackage{float}
\usepackage{caption}
\usepackage{mathtools}
\usepackage{listings}
\usepackage{quiver}
\usetikzlibrary{nfold}
\usetikzlibrary{braids,positioning}
\usetikzlibrary{decorations.markings}
\usepackage{svg}

\usepackage{courier}
\usepackage{color}

\definecolor{mygreen}{rgb}{0,0.6,0}
\definecolor{mygray}{rgb}{0.5,0.5,0.5}
\definecolor{mymauve}{rgb}{0.58,0,0.82}

\graphicspath{ {./images/} }
\usepackage[export]{adjustbox}

\newcommand\R{\mathbb R}

\newcommand{\beq}[1]{\begin{equation}\label{#1} }
\newcommand{\eeq}{\end{equation}}

\theoremstyle{plain}
\newtheorem{prop}[equation]{Proposition}
\newtheorem*{prop*}{Proposition}
\newtheorem{thm}[equation]{Theorem}

\newtheorem*{lemma*}{Lemma}
\newtheorem{cor}[equation]{Corollary}

\newtheorem{obs}[equation]{Observation}

\theoremstyle{remark}
\newtheorem{rem}[equation]{Remark}

\theoremstyle{definition}
\newtheorem{defn}[equation]{Definition}

\newtheorem{eg}[equation]{Example}
\newtheorem{constr}[equation]{Construction}

\makeatletter \@addtoreset{equation}{section} \makeatother

\let\oldtocsection=\tocsection
\let\oldtocsubsection=\tocsubsection
\let\oldtocsubsubsection=\tocsubsubsection
\renewcommand{\tocsection}[3]{\hspace{0em}\oldtocsection{#1}{#2}{#3}}
\renewcommand{\tocsubsection}[3]{ \hspace{1em} \oldtocsubsection{#1}{\small{#2}}{\small{#3}} }
\renewcommand{\tocsubsubsection}[3]{\hspace{2em}\oldtocsubsubsection{#1}{\small{#2}}{\small{#3}}}

\let\oldmarginpar\marginpar
\renewcommand\marginpar[1]{\-\oldmarginpar[\framebox{\setstretch{\marginparstretch}\begin{minipage}{\marginparwidth}{\raggedleft\scriptsize #1}\end{minipage}}]{\framebox{\setstretch{\marginparstretch}\begin{minipage}{\marginparwidth}{\raggedright\scriptsize #1}\end{minipage}}}}

\let\oldmarginpar\marginpar
\renewcommand\marginpar[1]{\-\oldmarginpar[\raggedleft\footnotesize #1]%
	        {\raggedright\footnotesize #1}}

\newskip\stdskip                      
\title[Homological Mirror Symmetry for orbifold log Calabi-Yau surfaces]{Homological Mirror Symmetry for orbifold log Calabi-Yau surfaces}
\author{Bogdan Simeonov}
\date{}

\begin{document}
\maketitle
\begin{abstract}
    We prove homological mirror symmetry for orbifold log Calabi-Yau surfaces at the large complex structure limit by constructing an abstract Lefschetz fibration associated to each pair $(\mathcal{X},\mathcal{D})$ with $\mathcal{X}$ a projective rational surface with isolated cyclic quotient orbifold points and $\mathcal{D}$ a stacky anticanonical divisor. We describe a Lefschetz stabilization procedure which, on the mirror, corresponds to the special McKay correspondence of \cite{ishii_special_2013}. Moreover, we relate our abstract construction to an explicit Laurent polynomial mirror in an example consisting of a family of orbifold del Pezzo surfaces.
    \end{abstract}
\setcounter{tocdepth}{1}
\hypersetup{
    bookmarksopen=true,
    bookmarksdepth=2
}
\tableofcontents
\addtocounter{section}{-1}
\section{Introduction}\label{sec:intro}
\subsection{Main results}
  In this paper, we study homological mirror symmetry for log Calabi-Yau orbifold surfaces, building on earlier work of \cite{hacking_homological_2023} on smooth log Calabi-Yau surfaces. The main objects of study are a class of log Calabi-Yau orbifold surfaces which we call \textit{effective} (Definition \ref{def:effectivelogcy}): these are the projective rational surfaces with isolated cyclic orbifold points which admit an effective anticanonical nodal stacky cycle $\mathcal{D}$ whose singular set contains all of the orbifold points of $\mathcal{X}$. Equivalently, these are the surfaces which can be constructed by taking a smooth log Calabi-Yau pair $(Y,D)$ with $D\in |-K_Y|$ a reduced nodal curve, contracting disjoint chains $\mathcal{E}_i=\bigcup E_{i,j}$ of rational curve components $\mathbb{P}^1\simeq E_{i,j}\subset D$ satisfying $E_{i,j}^2\leq -2$ and considering the resulting surface with cyclic quotient singularities as an orbifold. This produces a different compactification $(\mathcal{X}, \mathcal{D})$ of the same open Calabi-Yau variety $U=Y\setminus D = \mathcal{X}\setminus \mathcal{D}$.
  
  Associated to such a log Calabi-Yau pair $(\mathcal{X}, \mathcal{D})$ we construct an abstract mirror Lefschetz fibration $w_\mathcal{X}:W'\rightarrow \mathbb{C}$ (Construction \ref{constr:Lefschetz}). 
  \begin{thm}[Homological Mirror Symmetry at the large complex structure limit, Theorems \ref{thm:destabilization}, \ref{prop:restrictandvarphi}]\label{thm:mainthmabstract}
    Suppose $(\mathcal{X},\mathcal{D})$ is an effective log Calabi-Yau surface. Assume that $Y$, the minimal resolution of the coarse space of $\mathcal{X}$, is equipped with the distinguished complex structure as in \cite[Section 2.1.1]{hacking_homological_2023}. Consider the exact Lefschetz fibration $w_\mathcal{X}:W'\rightarrow \mathbb{C}$ from construction \ref{constr:Lefschetz} and the Hacking-Keating Lefschetz fibration $w_Y:W\rightarrow \mathbb{C}$ from \cite{hacking_homological_2023} which is mirror to $(Y,D)$. Then:
    \begin{itemize}
      \item The exact Lefschetz fibration $w_\mathcal{X}:W'\rightarrow \mathbb{C}$ can be obtained from $w_Y:W\rightarrow \mathbb{C}$ by a successive sequence of Lefschetz stabilizations. Hence, $W\simeq W'$ as symplectic manifolds.
      \item There is an equivalence of $A_\infty$ categories $D^b(\mathcal{X})\simeq \mathcal{F}(W',w_\mathcal{X})$ and a commutative diagram \[\begin{tikzcd}
D^b(Y) \arrow[d, "\Phi"', hook] \arrow[r, "\simeq"] & {\mathcal{F}(W,w_Y)} \arrow[d, hook] \\
D^b(\mathcal{X}) \arrow[r, "\simeq"']               & {\mathcal{F}(W',w_\mathcal{X})}              
\end{tikzcd}\]
intertwining the fully faithful embedding $\Phi$ defined via a Fourier-Mukai kernel by Ishii-Ueda \cite{ishii_special_2013} and the fully faithful inclusion induced by Lefschetz stabilization.
    \end{itemize}
  
  \end{thm}
We will review the Lefschetz stabilization procedure in Section \ref{pointer:stabilization}. Notably, Lefschetz stabilization does not change the total space, up to Liouville deformation equivalence. This is consistent with the expectation that compactifying $U$ by adding a divisor is mirror to equipping the SYZ mirror $U^\lor$ with a function, with different divisors corresponding to different functions on the same space.
  \subsection{Mirror symmetry for smooth log Calabi-Yau surfaces}
  Mirror symmetry for log Calabi-Yau surfaces was studied by Gross-Hacking-Keel  \cite{gross_mirror_2015} in their proof of Looijenga's conjecture, which states that a cusp singularity is smoothable if and only if the fundamental cycle of the minimal resolution of the dual cusp appears as an anticanonical divisor in a smooth log Calabi-Yau surface. Those authors used a tropicalized version of SYZ, in line with the Gross-Siebert program, to construct an intrinsic mirror to $(Y,D)$ which forms a formal family over $\mathrm{Spec}\,\mathbb{C}[[t]]$ smoothing the cusp singularity.
Later on, Hacking and Keating in \cite{hacking_homological_2023} used a different approach to prove homological mirror symmetry for smooth log CY surfaces. Rather than using the intrinsic mirror SYZ construction, they instead built an abstract Lefschetz fibration associated to $(Y,D)$ and showed that the total space of this Lefschetz fibration is equivalent to the total space of the almost toric fibration associated to $(Y,D)$. This, in turn, is expected to be the same as the general fiber of the Gross-Hacking-Keel family.

We briefly sketch the construction from \cite{hacking_homological_2023}. A guiding heuristic in mirror symmetry is that the anticanonical compactifying divisor $D$ should be mirror to the fiber of the mirror Landau-Ginzburg model. By work of Lekili-Polishchuk \cite{lekili_arithmetic_2017}, the mirror to a cycle of rational curves $D$ is a punctured torus, with as many punctures as there are nodes on $D$. Given a full exceptional collection of line bundles $\mathcal{L}_i$ on $D^b(Y)$, one can restrict them to line bundles on $D$ and find the corresponding mirror Lagrangian $S^1$ in the punctured torus. In particular, in the Lekili-Polishchuk equivalence the mirror to the structure sheaf $\mathcal{O}$ is identified with a reference longitude $\mathcal{V}_0$, whereas the mirror to a skyscraper sheaf on $D_i$ is identified with a meridian $M_i$. The mirror to a line bundle $\mathcal{L}|_D$ is constructed by Dehn twisting the longitude along this collection of meridians: $\mathcal{V}:=\prod_s \tau_{M_s}^{\mathcal{L}\cdot D_s} \mathcal{V}_0$. 

Given the data of a Riemann surface $F$ and an ordered sequence of vanishing cycles on it, one can construct (as in \cite[Lemma 16.9]{seidel_fukaya_2008}) an exact Lefschetz fibration $w_Y:W\rightarrow \mathbb{C}$ with fiber $F$ by attaching Weinstein $2$-handles to $\mathbb{D}^2\times F$ whose attaching $S^1$s are the vanishing cycles. Such a Lefschetz fibration has a categorical invariant, namely its Fukaya-Seidel category: \begin{defn}\label{def:FukayaCategory}
  The category of Lagrangian vanishing cycles $\mathcal{F}(W,w|\boldsymbol{\gamma})$ is defined to be the $\mathbb{Z}$-graded $A_\infty$ category whose objects are the Lagrangian vanishing cycles $\Gamma_i$ of $w:W\rightarrow \mathbb{C}$ corresponding to an ordered collection of vanishing paths $\boldsymbol{\gamma}$, together with morphisms spaces and $A_\infty$ operations defined in the compact Fukaya category of a reference fiber $\Sigma^0$: $$\mathrm{Hom}^\bullet_{\mathcal{F}(W,w)}(\Gamma_i,\Gamma_j):=\begin{cases}
    CF^\bullet_{\mathcal{F}(\Sigma^0)}(\Gamma_i, \Gamma_j), \quad i<j\\
  \mathbb{C}\langle \mathrm{id}\rangle, \quad i=j\\
  0,\quad i>j
  \end{cases}$$ 

  Note that while $\mathcal{F}(W,w|\boldsymbol{\gamma})$ depends on the choice of vanishing paths $\boldsymbol{\gamma}$, the category of twisted complexes $\mathrm{Tw}\,\mathcal{F}(W,w|\boldsymbol{\gamma})$ does not, by \cite[Theorem 18.24]{seidel_fukaya_2008}. We will denote this triangulated category by $\mathcal{F}(W,w)$.

  \begin{rem}
    There are various alternative definitions of the Fukaya-Seidel category in the exact setting, including the partially wrapped Fukaya categories of Sylvan \cite{sylvan_partially_2019} and Ganatra-Pardon-Shende \cite{ganatra_covariantly_2019}. In the case of an exact Lefschetz fibration, these categories are known to be generated by the Lefschetz thimbles (see \cite[Corollary 1.17]{ganatra_sectorial_2023}).
  \end{rem}
  Hacking and Keating showed, in \cite[Theorem 1.1]{hacking_homological_2023}, that under a special complex structure assumption, there is an $A_\infty$ equivalence $$D^b(Y)\simeq \mathcal{F}(W,w_Y)$$Moreover, this is compatible with localization: quotienting $D^b(Y)$ by the image of $D^b(D)$ results in $D^b(U)$ which is equivalent to $\mathcal{W}(W)$, which in turn is the result of a stop removal localization as defined in \cite{ganatra_sectorial_2023}. The special complex structure assumption corresponds to the exactness of the mirror: non-exact symplectic deformations of $W$ should correspond to certain deformations of the complex structure on $Y$.

  \subsection{Mirror symmetry and orbifold del Pezzo surfaces}

  In a different context, mirror symmetry has been applied in an attempt to classify orbifold del Pezzo surfaces. While the smooth del Pezzo surfaces fall into exactly $10$ deformation families, an analogous classification for del Pezzo surfaces admitting at most cyclic quotient singularities remains an open problem. A promising classification approach was laid out in \cite{akhtar_mirror_2015}, which describes a conjectural correspondence between, on one hand, certain $\mathbb{Q}$-Gorenstein deformation classes of orbifold del Pezzo surfaces which admit $\mathbb{Q}$-Gorenstein toric degenerations and, on the other hand, mutation-equivalence classes of Fano polygons. Additionally, under this conjectural correspondence, the classical periods of maximally-mutable Laurent polynomials whose Newton polytope produces a Fano polygon should recover the regularized quantum period of the mirror orbifold del Pezzo surface.
  The orbifold del Pezzo surfaces admitting $\mathbb{Q}$-Gorenstein toric degenerations have an effective anticanonical cycle and hence two different mirror LG models can be associated to them: an abstract one, as developed in this paper, and an explicit Laurent polynomial one, as developed in \cite{akhtar_mirror_2015} (see also \cite{oneto_quantum_2018}). In sections \ref{sec:explicitcase} and \ref{sec:Lagvc} we compare these two constructions in an explicit example: the family of orbifold del Pezzo surfaces which are the hypersurfaces of degree $k+1$ in $\mathbb{P}(1,1,1,k)$.  We sketch this comparison in Section \ref{sec:introexample} below.
  
  However, it must be mentioned that there are notable examples of orbifold del Pezzo surfaces which do not admit qG toric degenerations. For example, the anticanonical quasismooth wellformed log del Pezzo surfaces in weighted projective 3-spaces were classified by Johnson and Kollár in \cite{johnson_kahler-einstein_2001}. A series of such anticanonical log del Pezzo surfaces were studied extensively from the viewpoint of mirror symmetry in \cite{corti_hyperelliptic_2021} and their derived categories were described in \cite{gugiatti_full_2023}. Proving homological mirror symmetry for the Johnson-Kollár surfaces remains an interesting open problem.

  \subsection{The derived category of an orbifold surface}
  When $\mathcal{X}$ is a surface with $N$ orbifold points $\{q_m\}_{m=1}^N$ admitting local charts of the form $[\mathbb{C}^2/G_m]$ with $G_m\subset GL_2(\mathbb{C})$ a small finite subgroup, it was proved by Ishii and Ueda \cite[Theorem 1.4]{ishii_special_2013} (see also \cite{gugiatti_full_2023},\cite{ishii_mckay_nodate},\cite{ito_special_2001}) that there is a semiorthogonal decomposition of the form $$D^b(\mathcal{X})=\langle \mathbf{e}_{q_1}, \dots, \mathbf{e}_{q_N}, \Phi D^b(Y)\rangle$$
  This decomposition has a component which is the fully faithful image under a Fourier-Mukai functor $\Phi$ of the derived category of a smooth surface $Y$, the minimal resolution of the coarse space $X$ of $\mathcal{X}$. Moreover, for each orbifold point $q_m\in\mathcal{X}$, there is a component $\mathbf{e}_{q_m}$ admitting a full exceptional collection $N^{q_m}_d$ of sheaves\footnote{We use the notation $N_d$ instead of the notation $E_d$ which appears in \cite[Proposition 8.1]{ishii_special_2013} so as to not confuse these with the exceptional curves in the resolution $Y$. } with support at the orbifold point which are indexed by the non-special representations $\rho_d$ of $G_m$ (the notion of special representations was introduced in \cite{wunram_reflexive_1988}, where it is shown that the nontrivial special representations correspond to the exceptional curves of the minimal resolution of the quotient singularity).

  \subsection{The construction of the mirror abstract Lefschetz fibration}
  As mentioned before, mirror symmetry heuristics suggest that the general fiber $\Sigma^0$ of $w_\mathcal{X}:W'\rightarrow \mathbb{C}$ should be mirror to the nodal stacky cycle $\mathcal{D}$. Homological mirror symmetry for such nodal stacky cycles is proved in \cite{lekili_auslander_2018} (see also \cite{habermann_homological_2023}) using the following construction: first, consider the disjoint union of cylinders $C_m$ indexed by the components of $\mathcal{D}$. Then, for each $\frac{1}{r}(1,l)$ singularity of $\mathcal{X}$ appearing as a node on $\mathcal{D}$ where two orbifold rational curves $\mathcal{D}_m, \mathcal{D}_{m+1}$ meet\footnote{These could be the same, in the case that $\mathcal{D}$ is a reduced nodal orbifold curve with one component.}, Lekili and Polishchuk attach $n$ handles to the two cylinders $C_m, C_{m+1}$. Identifying the circles $\partial_+C_m$ and $\partial_- C_{m+1}$  with $\mathbb{R}/\mathbb{Z}$, the handle attachment is constructed so as to join together the point $\frac{a}{n}\in \partial_+ C_m$ with the point $\frac{-\overline{l^{-1}}a}{n}\in \partial_- C_{m+1}$, where $\overline{l^{-1}}$ denotes the residue modulo $r$. The resulting Riemann surface $\Sigma^0$ (potentially of high genus) will be the fiber of our Lefschetz fibration. Lekili and Polishchuk show in \cite{lekili_auslander_2018} that $D^b(\mathcal{D})\simeq \mathcal{W}(\Sigma^0)$ (see Section \ref{sec:LCSL} for a detailed description of this equivalence).
  
  To construct a collection of vanishing cycles on $\Sigma^0$, one has to understand the restrictions of $\Phi \mathcal{L}_i \in D^b(\mathcal{X})$ and $N^{q_i}_d$ to $\mathcal{D}$: this is slightly challenging, as these are not line bundles and hence they are not determined simply by a sequence of degrees on each $\mathcal{D}_m$, as is the case for smooth log CY surfaces. 

  \subsubsection{The vanishing cycles corresponding to the $\Phi D^b(Y)$ component}

  Now consider the special I-series $\{i_1=l, i_2, \dots, i_{n+1}=0\}$ associated to a $\frac{1}{r}(1,l)$ singularity, as in \cite[Definition 2.5]{gugiatti_full_2023}. These form the set of weights of the special representations from \cite{wunram_reflexive_1988}. We make the following elementary but crucial observation:
  \begin{obs} \label{ob:observation} The map $\{i_1, i_2, \dots, i_{n+1}\} \rightarrow \{ 0,1,\dots, r-1\}$ defined by $i_j\mapsto -i_j l^{-1} \mod r$ is order-preserving.
  \end{obs}
  This is because multiplication by $l^{-1}$ identifies the descending $I$-series with the ascending $J$-series, by \cite[Lemma 2]{wunram_reflexive_1987}.

  Hence, during the construction of the Lekili-Polishchuk mirror to $\mathcal{D}$, if one only attaches handles associated to the special I-series instead of all of $\{0,1,\dots, r-1\}$ for each $\frac{1}{r}(1,l)$ point of $\mathcal{X}$, the handles will not overlap and the resulting Riemann surface will be (after completion) a punctured torus. If we denote this Riemann surface by $T^0$, there is a natural identification between $T^0$ and $F$, the fiber of the Hacking-Keating mirror $w_Y:W\rightarrow \mathbb{C}$ to $(Y,D)$. This can be visualized by rotating the punctures of $F$ to be distributed vertically instead of horizontally. Under this identification, the meridian which is mirror to a skyscraper sheaf on $E\subset \mathcal{E}\subset D$ passes through exactly one of the handles. This is the handle associated to the special representation which corresponds to the exceptional curve $E$ under Wunram's correspondence \cite[Theorem 1.2]{wunram_reflexive_1988}.

  \begin{prop}[Propositions \ref{prop:fullyfaithfulbdry} and \ref{prop:fullyfaithfulspecialtorus}]
        There is a fully faithful functor $\Phi^\partial: \mathrm{Perf}(D) \rightarrow \mathrm{Perf}(\mathcal{D})$ which commutes with $\Phi$ under restriction and, under the equivalences $\mathrm{Perf}(D)\simeq \mathcal{F}(T^0), \mathrm{Perf}(\mathcal{D})\simeq \mathcal{F}(\Sigma^0)$, corresponds to a functor on the Fukaya category which is induced (on the level of underlying Lagrangians) by the inclusion of the special torus $T^0$ into $\Sigma^0$. In other words, there is a commutative diagram of functors: \[\begin{tikzcd}
        D^b(Y) \arrow[r, "\Phi", hook] \arrow[d]                                        & D^b(\mathcal{X}) \arrow[d]                              \\
        \mathrm{Perf}(D) \arrow[r, "\Phi^\partial", hook] \arrow[d, "\simeq"', no head] & \mathrm{Perf}(\mathcal{D}) \arrow[d, "\simeq", no head] \\
        \mathcal{F}(T^0) \arrow[r, hook]                                               & \mathcal{F}(\Sigma^0)                                 
        \end{tikzcd}\]
              \end{prop}

    The upshot of this proposition is that one can take the vanishing cycles from Hacking and Keating's construction and push them forward from $T^0$ into $\Sigma^0$.
   \subsubsection{The vanishing cycles corresponding to the Ishii-Ueda exceptional collection}
  It remains to find the vanishing cycles $L_d\subset \Sigma^0$ which are mirror to the Ishii-Ueda exceptional collection $N_d$ for the orthogonal to $\Phi D^b(Y)$ in $D^b(\mathcal{X})$ (we will omit the superscript $q_m$ when it is not important).

  The sheaves $N_d$ have support at an orbifold point $q$ which is a nodal point of $\mathcal{D}$. In fact, there is a sheaf $B_d\in D^b(\mathcal{D})$ with $\iota_{\mathcal{D}*} B_d=N_d$. Restricting this back to the boundary results in a push-pull exact triangle (as in \cite[Lemma 2.8]{kuznetsov_serre_2021}) $$\underbrace{\iota_{\mathcal{D}}^*N_d}_{\iota_\mathcal{D}^*\iota_{\mathcal{D}*} B_d}\rightarrow B_d \rightarrow B_d \otimes \omega_\mathcal{X}|_{\mathcal{D}}[2]$$
  Using this exact triangle, one can replace $\iota_{\mathcal{D}}^*N_d$ by a two-term complex. On the A-side, we use the equivalence $D^b(\mathcal{D})\simeq \mathcal{W}(\Sigma^0)$ from \cite{lekili_auslander_2018} to show in Propositions \ref{prop:mirrortoBd}, \ref{prop:mirrortoLd} that $B_d$ is mirror to a Lagrangian arc and the two-step complex, after performing surgery, is a Lagrangian $S^1$ that we denote by $L_d$.
  For each orbifold point $q\in \mathcal{D}$, one thus gets an ordered collection of Lagrangian $S^1$'s which we denote by $\mathfrak{N}_{q}:=\{L_{d_{min}}, \dots, L_{d_{max}}\}$. Finally, the abstract Weinstein Lefschetz fibration $w_\mathcal{X}:W'\rightarrow \mathbb{C}$ is determined by the data $$\{ \Sigma^0, (\mathfrak{N}_{q_1}, \dots, \mathfrak{N}_{q_N}, \mathcal{V}_1, \mathcal{V}_2, \dots, \mathcal{V}_n)\}$$ where the $\mathcal{V}_i$ are Lagrangian vanishing cycles contained in $T^0\subset \Sigma^0$ determined by the Hacking-Keating construction.

  \begin{figure}[H]
    \usetikzlibrary{backgrounds}
  \pgfdeclarelayer{base}
            \pgfsetlayers{base,background,main}
  \begin{tikzpicture}[scale=1.3,
    gap/.style={draw, very thick},
    band/.style={fill=white, draw=black, line width=0.6pt},
    line cap=round
  ]

  \begin{pgfonlayer}{background}
    \draw[thick, dashed] (-1,0) to (1,-8/10);
    \draw[thick,dashed] (-1,-4/10) to (1,-12/10);
    
    \draw[thick,dashed] (-1,12/10) to (1,-4/10);
    \draw[thick,dashed] (-1,16/10) to (1,0);
  \draw[thick,red!50!black, dotted] (-1,13/10) to (1,-3/10);
  \draw[thick,blue!50!black, dotted] (-1,15/10) to (1,-1/10);
  \draw[thick,dotted, red!50!black] (-1,-2/10) to (1,-10/10);

  \end{pgfonlayer}

  \path[band, draw=none] (-1,-2+8/10) -- (1,4/10) -- (1,8/10) -- (-1,-2+12/10)  -- cycle;
    \draw[very thick] (-1,-2+8/10) to (1,4/10);
    \draw[very thick] (-1,-2+12/10) to (1,8/10);
  \path[band, draw=none] (-1,-2+24/10) -- (1,-2+32/10) -- (1,-2+36/10) -- (-1,-2+28/10)  -- cycle;
    \draw[very thick] (-1,-2+24/10) to (1,-2+32/10);
    \draw[very thick] (-1,-2+28/10) to (1,-2+36/10);

  \draw[very thick] (-2,-2)  to (2,-2);
  \draw[very thick] (-2,2.4)  to (2,2.4);
  \draw[very thick] (-1,-2+4/10)  to (1,-2+4/10);
  \draw[very thick] (-1,2)  to (1,2);
  \draw[thick, red!50!black] (-1,-2+26/10) to (1,14/10);
  \draw[thick, red!50!black] (-1,-2+43/10) to (1,-2+43/10);
  \draw[thick,red!50!black] (-1,-2+43/10) to[in=180, out=180,looseness=0.4] (-1,-2+33/10);
  \draw[thick,red!50!black] (1,-2+10/10) to[in=0, out=0,looseness=0.4] (1,-2+17/10);
  \draw[thick,red!50!black] (-1,-2+18/10) to[in=180, out=180,looseness=0.4] (-1,-2+26/10);
  \draw[thick,red!50!black] (1,-2+43/10) to[in=0, out=0,looseness=0.4] (1,-2+34/10);
  \draw[thick,blue!50!black] (-1,-2+10/10) to (1,6/10);
  \draw[thick,blue!50!black] (-1,-2+2/10) to (1,-2+2/10);
  \draw[thick,blue!50!black] (-1,-2+41/10) to (1,-2+41/10);
  \draw[thick,blue!50!black] (1,-2+26/10) to[in=0, out=0,looseness=0.3] (1,-2+41/10);
  \draw[thick,blue!50!black] (1,-2+2/10) to[in=0, out=0,looseness=0.3] (1,-2+19/10);
  \draw[thick,blue!50!black] (-1,-2+2/10) to[in=180, out=180,looseness=0.3] (-1,-2+10/10);
  \draw[thick,blue!50!black] (-1,-2+41/10) to[in=180, out=180,looseness=0.3] (-1,-2+35/10);


  \draw[very thick] (-1,-2+4/10)  to (-1,-2+8/10);
  \draw[very thick] (-1,-2+12/10)  to (-1,-2+16/10);
  \draw[very thick] (-1,-2+20/10)  to (-1,-2+24/10);
  \draw[very thick] (-1,-2+28/10)  to (-1,-2+32/10);
  \draw[very thick] (-1,-2+36/10)  to (-1,-2+40/10);
  \begin{scope}[xshift=2cm]
  \draw[very thick] (-1,-2+4/10)  to (-1,-2+8/10);
  \draw[very thick] (-1,-2+12/10)  to (-1,-2+16/10);
  \draw[very thick] (-1,-2+20/10)  to (-1,-2+24/10);
  \draw[very thick] (-1,-2+28/10)  to (-1,-2+32/10);
  \draw[very thick] (-1,-2+36/10)  to (-1,-2+40/10);
  \end{scope}

  \draw[very thick, dashed] (-1,-2+32/10)  to (-1,-2+36/10);
  \draw[very thick, dashed] (-1,-2+16/10)  to (-1,-2+20/10);
  \draw[very thick, dashed] (1,-2+8/10)  to (1,-2+12/10);
  \draw[very thick, dashed] (1,-2+16/10)  to (1,-2+20/10);

  \node[fill, circle, inner sep =1pt] at (1.1063,1.8104) {};
  \draw[ postaction={decorate},
            decoration={markings, mark=at position .5 with {\arrow{>}}}] (-2,-2) node (v1) {} -- (2,-2) node (v3) {};
              \draw[ postaction={decorate},
            decoration={markings, mark=at position .5 with {\arrow{>}}}] (-2,2.4) node (v2) {} -- (2,2.4) node (v4) {};

              \draw[dotted, postaction={decorate},
            decoration={markings, mark=at position .5 with {\arrow{>>}}}] (v1) -- (v2);
            
              \draw[dotted, postaction={decorate},
            decoration={markings, mark=at position .5 with {\arrow{>>}}}] (v3) -- (v4);
  \end{tikzpicture}\caption{\label{fig:1/5(1,3)}The general fiber of the Lefschetz fibration associated to a surface with a $\frac{1}{5}(1,3)$ orbifold point. The non-special handles are drawn with a dash. The blue curve describes the Lagrangian $L_2$ and the red one describes $L_4$.}
  \end{figure}
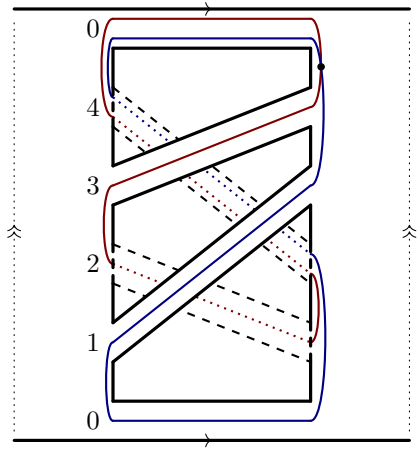

  \subsubsection{Example}
Consider $\mathbb{P}^2$ with its toric boundary. Then, blow up the point $[0:1:-1]$ three times in a row and the point $[1:0:-1]$ four times in a row. The result is a log Calabi-Yau surface $(Y,D)$ such that $D=D_1\cup D_2\cup D_3, D_1^2=1, D_2^2=-2, D_3^2=-3$. The chain $D_2\cup D_3$ can be contracted to a $\frac{1}{5}(1,3)$ singularity, producing an orbifold surface $(\mathcal{X},\mathcal{D})$.

Associated to this singularity are the special representations $\rho_0,\rho_1, \rho_3$ and the non-special representations $\rho_2, \rho_4$. There is a semiorthogonal decomposition $$D^b(\mathcal{X})=\langle \underbrace{N_2, N_4}_{\mathbf{e}_{q}}, \Phi D^b(Y)\rangle$$
The construction \ref{constr:Lefschetz} produces a Lefschetz fibration whose general fiber is a genus three Riemann surface $\Sigma^0$ with one boundary component, which can be thought of as a torus with three boundary components $T^0\simeq F$ to which two non-special handles have been attached, as in Figure \ref{fig:1/5(1,3)}.

The proof of homological mirror symmetry in Theorem \ref{thm:mainthmabstract} proceeds along in the same manner as in \cite{hacking_homological_2023}: one finds a model of $D^b(\mathcal{X})$ as a directed subcategory of $\mathrm{Perf}(\mathcal{D})$, as well as a model of $\mathcal{F}(W',w_\mathcal{X})$ inside $\mathcal{F}(\Sigma^0)$, which are identified under the Lekili-Polishchuk equivalence.
\subsection{Lefschetz stabilization}
 \label{pointer:stabilization}There is an operation on Lefschetz fibrations called Lefschetz stabilization (see \cite{giroux_existence_2016}, also \cite[Section 2]{keating_homological_2017}, \cite{koert_lecture_2017}, \cite{cieliebak_stein_2012}) which modifies the general fiber of a Lefschetz fibration and adds a new vanishing cycle, yet keeps the total space the same. This is a form of the Morse handle cancellation lemma: a $1$-handle (modifying the fiber) and $2$-handle (adding a Lefschetz thimble) cancel when the belt sphere of the $1$-handle intersects the attaching sphere of the $2$-handle at a single point transversely. In the language of Lefschetz fibrations, this takes the following shape: begin with an exact Lefschetz fibration on an exact symplectic $4$-manifold $(W, d\theta)$ with fiber $F$ and a set of vanishing paths determining a collection of vanishing cycles. Given an arc $\gamma \subset F$ with $\partial\gamma \subset \partial F, [\theta]=0\in H^1(\gamma, \partial \gamma)$, one can construct a new Lefschetz fibration $W'$ with fiber $F'$ given by attaching a Weinstein handle to $F$ along $\partial \gamma$. This process adds a new  critical point with associated vanishing cycle given by gluing the core of the Weinstein handle to $\partial \gamma$. The total space of $W'$ is Weinstein deformation equivalent to $W$.

 Conversely, if one finds an arc in the fiber $F$ which intersects a vanishing cycle $L$ in a single point (and no other vanishing cycles), then one can cut the arc from the general fiber and remove that vanishing cycle, resulting again in a Lefschetz fibration on the same total space. This procedure is known as \textit{destabilization}.

 Notice now that since the sheaves $N_d$ form an exceptional collection, $$\mathrm{Ext}_\mathcal{X}(N_d, N_{d'})=\mathrm{Ext}_\mathcal{D}(\iota^*\iota_*B_d, B_{d'})=\begin{cases}
 \mathbb{C}, d=d'\\
 0, d'<d
 \end{cases}$$
  To show that $W'$ and $W$ are related by Lefschetz stabilization (as is claimed in Theorem \ref{thm:mainthmabstract}), one uses the arcs $A_d$ mirror to $B_d$ as destabilizing arcs for the Lagrangian $S^1$'s $L_d$ mirror to $\iota^*N_d$. Moreover, the effect of cutting these arcs is the same as removing the non-special $1$-handles in the surface $\Sigma^0$ (note that the arcs $A_d$ need not be the cocores of the non-special handles!).

 \subsection{An explicit case study}\label{sec:introexample}
The main theorem of this paper concerns varieties at the large complex structure limit, to which an abstract Lefschetz fibration is associated. In practice, many interesting examples (e.g. orbifold del Pezzo surfaces) are not at the large complex structure limit (which tends to have lots of $-2$ curves), but rather come equipped with a deformation of the special complex structure. Deforming the complex structure on the B-side sometimes corresponds to partially compactifying and modifying the symplectic structure on the A-side. Moreover, rather than an abstract Lefschetz fibration, one can associate explicit Laurent polynomial mirrors to many orbifold del Pezzo surfaces by using toric degeneration techniques.

Consider for example the family of quasismooth hypersurfaces with a single $\frac{1}{k}(1,1)$ point $X_{k+1}\subset \mathbb{P}(1,1,1,k)$  (with $k$ assumed to be odd\footnote{This is a simplifying assumption: the mirror Laurent polynomial when $k$ is even has one more term, but we prefer to keep things uniform.}). These are orbifold del Pezzo surfaces which can be geometrically realized by taking $\mathbb{P}^2$, blowing up $k+1$ distinct points on a line (resulting in a smooth surface $Y$) and contracting the strict transform of the line. The mirror to $X_{k+1}$ is an explicit Landau-Ginzburg model arising from a toric degeneration of $X_{k+1}$. We verify in Proposition \ref{prop:abstractandexplicit} that the Lefschetz fibration induced by this Landau-Ginzburg model coincides, topologically, with the abstract one from Construction \ref{constr:Lefschetz} in the case where all the $k+1$ blowups happen at the same point. However, unlike in the assumption of Theorem \ref{thm:mainthm}, a generic $Y$ does not come equipped with the distinguished complex structure since the points blown up on $\mathbb{P}^2$ are distinct: there is a nontrivial moduli of complex structures depending on the positions of these points. An explicit mirror map is constructed, identifying an open subset of this moduli of complex structures with a space of non-exact symplectic structures on the total space of the mirror Landau-Ginzburg model. This mirror map is defined by computing intrinsic quantities in a non-exact Fukaya-Seidel category, in a similar vein as \cite{auroux_mirror_2006},\cite{auroux_mirror_2004}.
  \begin{rem}
  The orbifold del Pezzo surfaces with a single $\frac{1}{k}(1,1)$ point have been classified completely by \cite{cavey_pezzo_2020}: all but one exceptional case are given by blowing up a collection of points on $\mathbb{P}(1,1,k)$. In the notation of \cite{cavey_pezzo_2020}, the exceptional case consists of a family denoted $B_k^{(k)}$ which coincides with the family of hypersurfaces $X_{k+1}$.  \end{rem}
  \subsubsection{The derived category of $X_{k+1}$}

  By applying Ishii and Ueda's theorem \cite{ishii_special_2013} to a member $\mathcal{X}$ of the family of hypersurfaces $X_{k+1}$, we obtain a semiorthogonal decomposition of $D^b(\mathcal{X})$ into torsion sheaves $e_d:=\mathcal{O}_q\otimes \rho_d$ supported at the orbifold point, each of which corresponds to a non-special representation of the subgroup $\{\begin{pmatrix}
    \chi & 0 \\ 0 & \chi
    \end{pmatrix}, \chi^k=1\}\subset GL(2,\mathbb{C})$, as well as the fully faithful image under a Fourier-Mukai functor $\Phi$ of the derived category of the smooth surface $Y$ which is the minimal resolution of the coarse space of $\mathcal{X}$: $$D^b(\mathcal{X})=\langle \underbrace{e_2, \dots, e_{k-1}}_{\mathbf{e}_q}, \Phi D^b(Y)\rangle$$
\subsubsection{The mirror LG model}
  The task now is to produce an analogous semiorthogonal decomposition on the Fukaya-Seidel category of the mirror LG model.

  The mirror manifold $M^0$ is defined to be, topologically, the total space of the manifold defined by the algebraic equation $$\{ zx=P(y)\}\subset \mathbb{C}_x\times \mathbb{C}_z\times \mathbb{C}^\times_y$$
  where $P(y)=\prod_{i=1}^{k+1} (y+\mathbf{q}_i)$ is a polynomial of degree $k+1$. We will take $P=(1+y)^{k+1}+\epsilon$ for a small constant $\epsilon$. We will remark later on the significance of the parameters $\mathbf{q}_i$.
  In other words, $M^0$ is diffeomorphic to the complement of a cylinder in the $A_k$ Milnor fiber. As a symplectic manifold, $M^0$ will come equipped with a symplectic form of the following type:
    \begin{equation}\label{def:symplecticform}
    \omega^{\underline{\varepsilon}}:=\underbrace{dx\wedge d\overline{x}+d\log y \wedge d\log{\overline{y}}+dz\wedge d\overline{z}}_{_{\omega^{ex}}}+\sum_{i=1}^k \varepsilon_i \eta_i, \qquad \varepsilon_i\neq 0\end{equation} where $\eta_i$ are Thom forms supported near a chain of $-2$ spheres in $M^0$. The constants $\varepsilon$ are assumed to be all distinct and sufficiently small so that the form is non-degenerate.
    There is a function $\mathbf{f}:M^0\rightarrow \mathbb{C}$ which in the coordinates $x,y,z$ is written as $$\mathbf{f}=\frac{\prod_{i=1}^{k+1} (y+\mathbf{q}_i)}{xy}+x+(\tau_1y+\dots \tau_{\frac{k-1}{2}}y^{\frac{k-1}{2}})=\frac{z}{y}+x+(\tau_1y+\dots \tau_{\frac{k-1}{2}}y^{\frac{k-1}{2}})$$
  This function defines a fibration whose general fiber is a twice-punctured Riemann surface of genus $\frac{k+1}{2}$.

  \begin{figure}[H]
    \centering
      \begin{tikzpicture}[scale=0.33]
        \centering
        \def\n{13}

         \pgfdeclareradialshading{annulus1}{\pgfpoint{0cm}{0cm}}%
          {rgb(0cm)=(1,1,1);
          rgb(0.2cm)=(1,1,1);
          rgb(0.3cm)=(1,0.92,0.8);
          rgb(0.4cm)=(1,1,1)
          }
              
        \pgfdeclareradialshading{annulus2}{\pgfpoint{0cm}{0cm}}%
          {rgb(0cm)=(1,1,1);
          rgb(0.2cm)=(1,1,1);
          rgb(0.3cm)=(1,0.8,0.8);
          rgb(0.4cm)=(1,1,1)
          }
          
        \pgfdeclareradialshading{annulus3}{\pgfpoint{0cm}{0cm}}%
          {rgb(0cm)=(1,1,1);
          rgb(0.2cm)=(1,1,1);
          rgb(1cm)=(0.6,0.6,1);
          rgb(1.5cm)=(1,1,1)
          }

        \pgfdeclareradialshading{annulus4}{\pgfpoint{0cm}{0cm}}%
          {rgb(0cm)=(1,1,1);
          rgb(0.60cm)=(1,1,1);
          rgb(0.7cm)=(0.8,0.6,0.8);
          rgb(0.80cm)=(1,1,1)
          }
          
              
            
            
            \shade[shading=annulus4,even odd rule] (0,0) (0,0) circle (7.5);
            \shade[shading=annulus3,even odd rule] (0,0) (0,0) circle (0.43);
            \shade[shading=radial, outer color = orange!40, inner color = white] (-1,0) circle (0.2);


        \def\reference{4.5}

        \draw  plot[smooth, tension=.7] coordinates {(\reference,0) (6,0)};
        \draw  plot[smooth, tension=.7] coordinates {(\reference,0) (5.2972,-2.7948)};
        \draw  plot[smooth, tension=.7] coordinates {(\reference,0) (3.3979,-4.8988)};
        \draw (\reference,-0) .. controls (2.4205,-4.4426) and (3.0212,-5.5418) .. (0.7236,-5.9524);
        \draw (\reference,-0) .. controls (1.887,-4.8836) and (1.752,-5.2546) .. (-2.1022,-5.5734);
        \draw (\reference,-0) .. controls (1.1141,-5.6189) and (-0.4,-4.2) .. (-4.4807,-3.9649);
        \draw (\reference,-0) .. controls (0.7686,-5.6295) and (-1.7,-3.7) .. (-5.8026,-1.4526);
        \draw (\reference,0) .. controls (0.3345,-5.6209) and (-3.1,-3.3) .. (-5.8,1.4);
        \draw (\reference,0) .. controls (-0.113,-5.719) and (-4.5,-2.4) .. (-4.5,3.9);
        \draw (\reference,0) .. controls (-0.4651,-5.3938) and (-3.9859,-1.9715) .. (-2.1044,5.5953);
        \draw (\reference,-0) .. controls (-0.4395,-4.8153) and (-4.653,-1.5284) .. (0.6874,5.935);
        \draw (\reference,-0) .. controls (-0.6774,-4.6207) and (-5.5001,0.3428) .. (3.4171,4.9186);
        \draw (\reference,-0) .. controls (-0.9097,-4.461) and (-5.6,1.9) .. (5.3337,2.7407);

        \coordinate[circle,fill, inner sep = 1pt,orange] (v4) at (-1,0) {};
        \coordinate[circle,fill, inner sep = 1pt,blue] (v2) at (-0.15,0.25) {};
        \coordinate[circle,fill, inner sep = 1pt,blue] (v3) at (0.25,0) {};
        \coordinate[circle,fill, inner sep = 1pt,teal] (v1) at (\reference,0) {} {};

        \node[circle,fill, inner sep = 1pt,blue] (v7) at (-0.15,0.25) {};
        \node[circle,fill, inner sep = 1pt,orange] (v8) at (-1,0) {};
        \node[circle,fill, inner sep = 1pt,blue] (v5) at (-0.15,-0.25) {};
        \node[circle,fill, inner sep = 1pt,blue] (v6) at (0.25,0) {};
          
        \draw  (v1) edge (v5);

        \draw  (v1) edge (v6);

        \draw  (v1) edge (v7);

        \draw (\reference,0) .. controls (3.6605,1.8643) and (1.5446,1.8043) .. (-1,0);
        \node[circle,fill, inner sep = 0.5pt,orange] at (-1,0) {};
        \node[circle,fill, inner sep = 2pt,teal] at (\reference,0) {} {};
        \foreach \k in {0,...,\numexpr\n-1} {
            \node[circle,purple,opacity=1,fill,inner sep=1pt,
                  label=above:] 
              at ({360*\k/\n}:6) {};
          }
        \node[circle,fill, inner sep = 1pt,orange] at (-1,0) {};
        \end{tikzpicture}
        \caption{\label{fig:vanishingpaths} A reference fiber (green), $k-2$ critical values which go to infinity as $s\rightarrow 0$ (purple), three critical values close to $0$ in blue, as well as a critical value at $-1$ (orange) with multiplicity $k+1$, in the case $k=15, \mathbf{q}_i=1, \tau_1=1, \tau_j=0,j>1$.}
      \end{figure}
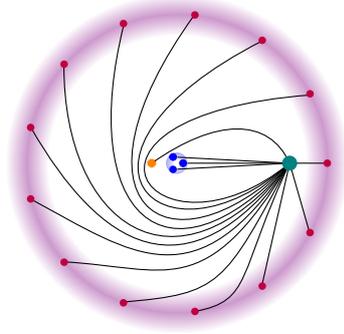

  The parameters $\mathbf{q}_i$ and $\tau_i$ have a precise interpretation in terms of holomorphic curve invariants of $\mathcal{X}$, namely the disk potential and the quantum orbifold cohomology: $\mathbf{q}_i$ is a Novikov parameter associated to a class in $H^2(\mathcal{X})$ , whereas the $\tau_i$ correspond to twisted sectors in the Chen-Ruan cohomology of $\mathcal{X}$. As such, they depend on the symplectic structure of $\mathcal{X}$ which in turn determines the complex structure of $M^0$. However, our main goal is to study instead the symplectic structure of $M^0$ so these values will not be of particular importance. In particular, there exists a deformation $\mathbf{f}_s:=\frac{\prod (y+\mathbf{q}_i)}{xy}+sx+(\tau_1y+\dots \tau_{\frac{k-1}{2}}y^{\frac{k-1}{2}})$ which, for $s$ very small, $\tau_1=1, \tau_j=0,j>1$ and $\mathbf{q}_i$ all equal to $1$, has critical values distributed as in Figure \ref{fig:vanishingpaths}. There are $k-2$ critical values which shoot off to infinity as $s\rightarrow 0$, three critical values near $0$ (in blue) and a degenerate $k+1$-fold critical value at $-1$. After perturbing $P$ from $(1+y)^{k+1}$ to $(1+y)^{k+1}+\epsilon$, the degenerate critical value splits off into $k+1$ non-degenerate ones. It is crucial to take $s$ very small, otherwise the distribution of critical values will be quite different than the one depicted in Figure \ref{fig:vanishingpaths}.

    The total space of the Landau-Ginzburg model is a non-exact symplectic manifold.  Following Auroux-Katzarkov-Orlov in \cite{auroux_mirror_2004}, we will define our non-exact Fukaya-Seidel category to be generated by the vanishing cycles as in \ref{def:FukayaCategory} with the various pseudoholomorphic polygons weighted by their symplectic area in $M^0$, as well as the holonomy associated to a B-field (this, as well as the grading data and spin structure, is explained in more detail in Section \ref{sec:Lagvc}). The structure constants of multiplication depend on a set of Novikov parameters\footnote{Not to be confused with the parameters $\mathbf{q}_i$ which depend on the symplectic structure of $\mathcal{X}$ rather than $M^0$.} $q_{1,j}:=\exp(2\pi i [B+i\omega](S_{1,j})), j=2,\dots, k+1$, where $S_{1,j}$ are spheres in $M^0$. Under the topological identification of $M^0$ as the complement of a cylinder in the $A_k$ Milnor fiber, these are the generators of $H_2$ of the Milnor fiber.
  \end{defn}

  \subsubsection{Homological mirror symmetry for $X_{k+1}$}
    A natural collection of vanishing paths (like the ones depicted in Figure \ref{fig:vanishingpaths}, after perturbing $P$) produces a collection of Lagrangian vanishing cycles that we denote by:
  $$\color{purple}\tilde{L}_{k-1}, \tilde{L}_{k-2}, \dots, \tilde{L}_2, \color{blue} P_{-1},P_0,P_1, \color{orange} B_1,\dots, B_{k+1}$$To compare this to the semiorthogonal decomposition on the B-side, a further step is required, mutating the subcollection $\langle \tilde{L}_{k-1},\dots, \tilde{L}_2\rangle$ into its left dual $\langle L_2, \dots, L_{k-1}\rangle$. After this is done\footnote{And also mutating $P_{-1}$ past $P_0$}, one obtains Lagrangian vanishing cycles which can be graded so that all morphism spaces are concentrated in degree $0$, provided that all $\varepsilon_i\neq 0$ in \ref{def:symplecticform}. As such, the only non-trivial $A_\infty$ operation is the product $\mu^2$. We compute the endomorphism algebra of this exceptional collection and match it with the algebra of a full, strong exceptional collection on the B-side, proving Theorem \ref{thm:mainthm}.

  \begin{thm}\label{thm:mainthm}
    Let $\omega$ be as in \ref{def:symplecticform}. There is an equivalence of $A_\infty$ categories $$D^b(\mathcal{X})\simeq \mathcal{F}(M^0, \mathbf{f}_s, \omega)$$ where $\mathcal{X}$ is the orbifold del Pezzo surface in the family of hypersurfaces $X_{k+1}$ obtained by blowing up the $k+1$ points $$pt_1=[1:-1:0], pt_2=[1:-q_{1,2}:0], \dots, pt_{k+1}=[1:-q_{1,k+1}:0]$$
        on $\mathbb{P}^2$ and contracting the strict transform of $\{z=0\}$.
  \end{thm}
\subsection{Organization of the paper}
  The structure of the paper is divided as follows: in Section \ref{sec:LCSL}, we prove HMS at the large complex structure limit in the form of Theorem \ref{thm:mainthmabstract}. We do so by explicitly computing the restrictions to $D^b(\mathcal{D})$ of the full exceptional collection of $D^b(\mathcal{X})$ induced by the special McKay correspondence of \cite{ishii_special_2013} and constructing a mirror abstract Lefschetz fibration using the Lekili-Polishchuk equivalence from \cite{lekili_auslander_2018}. In Sections \ref{sec:explicitcase} and \ref{sec:Lagvc}, we study the explicit example $X_{k+1}$ and relate it to our abstract construction. We introduce the mirror Landau-Ginzburg model to $X_{k+1}$, which is constructed using a toric degeneration. We compute the $A_\infty$ structure on the category of vanishing cycles associated to the Landau-Ginzburg model mirror to $X_{k+1}$. This is done by showing there is grading data with all intersection points in degree $0$, after which we compute the $\mu^2$ product. The disks are enumerated using a bifibration method, realizing the general fiber $\Sigma^0$ as a branched double cover. This results in the proof of Theorem \ref{thm:mainthm}.

\subsection{Acknowledgements}
The author would like to thank his supervisor Yankı Lekili for suggesting this project and for numerous invaluable discussions. Moreover, the author would like to thank Matthew Habermann for reading an early draft of this paper and for his many helpful suggestions, as well as Calum Crossley for his insights on derived categories. The author is funded by a London School of Geometry and Number Theory–Imperial College London PhD studentship, which is supported by the Engineering and Physical Sciences Research Council [EP/S021590/1].

  \section{Homological mirror symmetry at the large complex structure limit}\label{sec:LCSL}
  \subsection{The derived category of an orbifold surface}
  \subsubsection{Surfaces with cyclic quotient singularities}
  Recall that the cyclic quotient singularity $\mathbb{C}^2/\frac{1}{r}(1,l)$ has a minimal resolution whose fundamental cycle consists of a chain of curves $E_1, \dots, E_n$ satisfying $E_i\cdot E_i=-b_i$ with the constants $b_i$ appearing in the Hirzebruch-Jung continued fraction expansion $$\frac{r}{l}=b_1-\frac{1}{b_2-\frac{1}{b_3-\dots}}=[b_1,\dots, b_n]$$

  The I-series $I(r,l)=\{i_0=r,i_1=l, \dots, i_n=1, i_{n+1}=0\}$ of the cyclic quotient singularity is a descending sequence defined recursively by $$i_0=r, \quad i_1=l, \quad i_0=b_1i_1-i_2, \dots$$
  The dual J-series $J(r,l)=\{j_0=0, j_1=1,\dots, j_{n+1}=r\}$ is defined via the recursion $j_t=b_{t-1}j_{t-1}-j_{t-2}$. By \cite[Lemma 2]{wunram_reflexive_1987}, $j_t\equiv l^{-1}i_t \mod r$.
   \begin{defn}\label{def:effectivelogcy}
     An effective log Calabi-Yau orbifold surface $(\mathcal{X},\mathcal{D})$ is defined to be any surface which can be obtained by the following construction:\begin{itemize}
  \item Start with a smooth log Calabi-Yau surface $(Y,D)$ with maximal boundary. This means that $Y$ is a smooth, rational projective surface and $D\in |-K_Y|$ is either an irreducible nodal rational curve, or a cycle of at least $2$ rational curves.
  \item Contract a (possibly empty) collection $\mathcal{C}=\{ \mathcal{E}_1, \dots, \mathcal{E}_m\}, \mathcal{E}_i=E_{i,1}\cup \dots \cup E_{i,r_i}$ of chains of rational curves $E_{i,j}\simeq \mathbb{P}^1, E_{i,j}\cdot E_{i,j}\leq -2$ contained in $D$ which are pairwise disjoint: $E_{i,j}\cdot E_{i',j'}=0$ whenever $i\neq i'$. The contraction map $(Y,D)\rightarrow (X,D^{orb})$ sends $D$ to a cycle $D^{orb}\subset X$. The result upon considering $X$ as an orbifold is the effective log Calabi-Yau orbifold pair $(\mathcal{X},\mathcal{D})$.
  \end{itemize} 
      \end{defn}
  Contractibility of the chains $\mathcal{E}_i$ is guaranteed by Artin's theorem \cite{artin_numerical_1962}. The cycle $\mathcal{D}$ is a nodal stacky curve in the sense of \cite{lekili_auslander_2018}.

  In other words: the effective orbifold surfaces are the rational orbifold surfaces with cyclic quotient orbifold points admitting an effective, rational anticanonical cycle $\mathcal{D}$ which contains all of the orbifold points of $\mathcal{X}$ in its singular locus.
  \subsubsection{The special McKay correspondence}
  The special McKay correspondence is a generalization of the derived McKay correspondence to non-Gorenstein quotient singularities. It was proved in Ishii-Ueda \cite{ishii_special_2013} which builds on earlier work of Wunram \cite{wunram_reflexive_1988}, where the notion of special representation is defined. In the cyclic quotient case that we will be interested, the weights of the special representations are just given by the I-series.

  \begin{thm}[Wunram]
    There is a correspondence between the exceptional curves in the minimal resolution $Y\rightarrow \mathbb{C}^2/\frac{1}{r}(1,l)$ and the non-trivial special representations $\rho_i$ of $G$.
    Moreover, there is a collection of line bundles $\mathcal{R}_{\rho_i}$ on $Y$ indexed by the special representations of $G$ such that $$c_1(\mathcal{R}_{\rho_i})\cdot E_j=\delta_{ij}$$
  \end{thm}

  Given $\mathcal{X}$ an orbifold surface with quotient singularities of the form $\mathbb{C}^2/G$ with $G\subset GL_2(\mathbb{C})$ a small finite subgroup whose underlying singular coarse space $X$ admits a minimal resolution $Y$, one can consider the correspondence \[\begin{tikzcd}
\mathcal{Z} \arrow[r, "\mu"] \arrow[d, "\nu"'] & Y \arrow[d, "f"] \\
\mathcal{X} \arrow[r, "s"']                    & X               
\end{tikzcd}\]
Here, $\mathcal{Z}$ is the stack $(\mathcal{X}\times_X Y)_{red}$. It can be thought of as $Y$ with a root stack structure along the exceptional curves of the resolution. This correspondence induces a fully faithful Fourier-Mukai functor $$\Phi:D^b(Y)\rightarrow D^b(\mathcal{X})$$ by pulling back along $\mu$ and pushing down along $\nu$. Furthermore, there is a semiorthogonal decomposition $$D^b(\mathcal{X})=\langle \mathbf{e}_{q_1},\dots, \mathbf{e}_{q_N}, \Phi D^b(Y)\rangle $$where the $\mathbf{e}_{q_j}$ are themselves generated by sheaves with support at the orbifold points of $\mathcal{X}$.

When $\mathcal{X}=[\mathbb{C}^2/\frac{1}{r}(1,l)]$, the functor $\Phi$ sends $\mathcal{R}_{\rho_i}^\lor$ to $\mathcal{O}\otimes \rho_i$ (see \cite[Section 3.1]{gugiatti_full_2023}) and the remaining category is generated by the non-special skyscraper sheaves $\mathcal{O}_0\otimes \rho_d$. However, these do not always form a full exceptional collection: for example, in the case $\frac{1}{8}(1,3)$ there is a morphism from $\mathcal{O}_0\otimes \rho_2$ to $\mathcal{O}_0\otimes \rho_6$ and vice versa. To circumvent this issue, Ishii and Ueda  \cite[Theorem 1.4]{ishii_special_2013} defined an alternative collection of sheaves supported at the orbifold point defined as follows: for non-special $d$, the sheaf $N_d$ is associated to the equivariant module $$\mathbb{C}[u,v]/(u,v^{j_t})\otimes \rho_{d-(j_t-1)l},\qquad i_t <d< i_{t-1}$$

\begin{thm}[Ishii-Ueda]
Let $\mathcal{X}$ be an orbifold surface with isolated cyclic quotient orbifold points and let $Y$ be the minimal resolution of its coarse space. Then there is a semiorthogonal decomposition $$D^b(\mathcal{X})=\langle \mathbf{e}, \Phi D^b(Y)\rangle$$Moreover, the category $\mathbf{e}$ admits a full exceptional collection consisting of the sheaves $N_d$ which are local to the orbifold points of $\mathcal{X}$ and are indexed by the non-special representations.
\end{thm}
\subsubsection{Restricting to the boundary}
  From now on, assume $\mathcal{X}$ is an effective orbifold surface with cyclic quotient singularities, admitting an anticanonical cycle $\mathcal{D}$. 

  \begin{prop}\label{prop:restobdry} Let $A,B$ be exceptional objects in $D^b(\mathcal{X})$ with $\mathrm{Ext}_\mathcal{X}^\bullet(B,A)=0$. The restriction functor $\iota^*_{\mathcal{D}}:D^b(\mathcal{X})\rightarrow \mathrm{Perf}\,(\mathcal{D})$ induces an isomorphism $$\mathrm{Ext}_\mathcal{X}^\bullet(A,B)\simeq \mathrm{Ext}_\mathcal{D}^\bullet(\iota^*_{\mathcal{D}}A,\iota^*_{\mathcal{D}}B)$$
\end{prop}
\begin{proof}
The proof is the same as in \cite[Lemma 4.2]{hacking_homological_2023}, using the exact triangle $$C_{\iota_\mathcal{D}^*}\rightarrow \mathrm{id}\rightarrow \iota_{\mathcal{D}*}\iota^*_{\mathcal{D}}$$
with $C_{\iota^*_{\mathcal{D}}}\simeq -\otimes \mathcal{I}_{\mathcal{D}}\simeq -\otimes \omega_{\mathcal{X}}$ as in \cite[Lemma 2.8]{kuznetsov_serre_2021}. Namely, the map $$\mathrm{Ext}^\bullet_\mathcal{X}(A,B)\rightarrow \mathrm{Ext}^\bullet_{\mathcal{D}}(\iota^*A,\iota^*B)$$ is identified with the composition with the counit $B\rightarrow \iota_{\mathcal{D}*}\iota^*_\mathcal{D}B$, giving a long exact sequence $$\dots \rightarrow \underbrace{\mathrm{Ext}^\bullet_\mathcal{X}(A,B\otimes \omega_\mathcal{X})}_{0 \text{ by Serre duality}}\rightarrow \mathrm{Ext}^\bullet_\mathcal{X}(A,B)\rightarrow \mathrm{Ext}^\bullet_{\mathcal{X}}(A,\iota_{\mathcal{D}*}\iota^*_\mathcal{D}B)\rightarrow \dots$$
\end{proof}

  \subsection{Mirror symmetry for nodal stacky cycles and the Lekili-Polishchuk equivalence}
  \subsubsection{The derived category of a nodal stacky curve}
  We recall briefly the equivalence from \cite{lekili_auslander_2018}: consider a cycle $\mathcal{D}=\cup \mathcal{D}_m$ of weighted projective lines, with each $\mathcal{D}_m\simeq \mathbb{P}^1_{k_m, k_{m+1}}$ such that $\mathcal{D}_m\cap \mathcal{D}_{m+1}=q_m$ is an orbifold point, glued so that locally near $q_m$, $\mathcal{D}$ is identified with the quotient of $\{uv=0\}$ by the cyclic group action of $\mu_{r_m}$ such that $\zeta\cdot (u,v)=(\zeta^{k_m}u, \zeta v)$. This is analytically equivalent to the quotient by the action $\zeta \cdot (u,v)=(\zeta u, \zeta^{l_m}v)$ where $k_ml_m\equiv 1 \mod r_m$. 
  
  To this cycle there are three associated categories: $\mathrm{Perf}\, \mathcal{D}$ (which is proper but not smooth), $D^b(\mathcal{D})$ (which is smooth but not proper) and finally $D^b(\mathcal{A}_\mathcal{D}-\mathrm{mod})$, the derived category of modules over the Auslander order, which is both smooth and proper and is a categorical resolution of singularities in the sense of Kuznetsov. In other words, there is a diagram of functors \[\begin{tikzcd}
D^b(\mathcal{A}_\mathcal{D}) \arrow[d, "\pi_*"'] & \mathrm{Perf}(\mathcal{D}) \arrow[l, "\pi^*"', hook] \arrow[ld, hook] \\
D^b(\mathcal{D})                                 &                                                                      
\end{tikzcd}\]
with $\pi^*, \pi_*$ satisfying an adjoint property.

  The normalization of $\mathcal{D}$ consists of a disjoint union $\coprod \pi_m:\coprod \tilde{\mathcal{D}}_m\rightarrow \mathcal{D}$ with $q_{m,+}\in \tilde{\mathcal{D}}_m$ projecting to $q_m$ and $q_{m,-}\in \tilde{\mathcal{D}}_m$ projecting to $q_{m-1}$. The Auslander order is defined as $$\mathcal{A}_\mathcal{D}:=\begin{pmatrix}
  \tilde{\mathcal{O}} & \mathcal{I} \\ \tilde{\mathcal{O}} & \mathcal{O}_\mathcal{D}
  \end{pmatrix}$$ 
  There are $\mathcal{A}_\mathcal{D}$-modules $$\mathcal{P}_m(j_1,j_2):=\begin{pmatrix}
    \pi_{m*} \mathcal{O}(j_1[q_{m,-}]+j_2[q_{m,+}])\\
    \pi_{m*} \mathcal{O}(j_1[q_{m,-}]+j_2[q_{m,+}])
  \end{pmatrix}$$ which project to pushforwards of line bundles from $\mathcal{D}_m$ in $D^b(\mathcal{D})$. There is moreover, for each node $q_m$, a simple module $$\mathcal{S}_{q_m}:=\begin{pmatrix}
    0 \\
    \mathcal{O}_{q_m}
  \end{pmatrix}$$
  Each stacky node $q_m$ with $\mathrm{Aut}(q_m)=\mu_{r_m}$ (where $r_m=r_{m,+}=r_{m+1,-}$) is such that $\mu_{r_m}$ acts on $\mathcal{O}(-q_{m+1,-})$ at $q_{m+1,-}$ via multiplication by $\zeta$, whereas the action on $\mathcal{O}(-q_{m,+})$ at $q_{m,+}$ is through the character $\zeta\mapsto \zeta^{k_m}$. Moreover, for each character of $\mu_{r_m}$ there is a twist operation for modules supported at the stacky node. In \cite{lekili_auslander_2018}, this twisting operation is denoted by $\mathcal{M}\mapsto \mathcal{M}\{c\}$. We will sometimes denote it by $\mathcal{M}\mapsto \mathcal{M}\otimes \rho^{(k_m,1)}_c$. 
  \begin{rem}
  If one instead uses the presentation of a neighbourhood of the node via the action $\zeta\cdot (u,v)=(\zeta u, \zeta^{l_m}v)$, the twist operation is equivalent to $\mathcal{M}\mapsto \mathcal{M}\otimes \rho^{(1,l_m)}_{l_mc}$: $$-\{c\}=-\otimes \rho^{(k_m,1)}_c=-\otimes \rho^{(1,l_m)}_{l_mc}$$
  \end{rem}
  
  Moreover, on $\tilde{\mathcal{D}}_i$ there are morphisms \[\begin{tikzcd}[column sep=5em]
{\mathcal{O}(-r_{m-1}q_{m,-})} \arrow[r, "x_m(-(r_{m-1}-1))"] \arrow[d, "="'] & {\mathcal{O}(-(r_{m-1}-1)q_{m,-})} \arrow[r] & \dots \arrow[r, "x_m(-1)"]  & {\mathcal{O}(-q_{m,-})} \arrow[r, "x_m(0)"]  & \mathcal{O} \arrow[d, "="] \\
{\mathcal{O}(-r_mq_{m,+})} \arrow[r, "y_m(-(r_{m}-1))"']                      & {\mathcal{O}(-(r_i-1)q_{i,+})} \arrow[r]     & \dots \arrow[r, "y_m(-1)"'] & {\mathcal{O}(-q_{m,+})} \arrow[r, "y_m(0)"'] & \mathcal{O}               
\end{tikzcd}\]
  It is shown in \cite[Theorem 1.2.3]{lekili_auslander_2018} that the derived category of modules over the Auslander has a generating strong exceptional collection consisting of the modules \begin{itemize}
  \item $\mathcal{S}_{q_m}\{c\}[-1], \qquad c=0,1,\dots, r_m-1$
  \item $\mathcal{P}_m(j, -1), \qquad j=0,1,2,\dots, r_{m-1}$
  \item  $\mathcal{P}_i(0,j-1), \qquad j=0,1,2,\dots, r_m$.
  \end{itemize} 
  Moreover, the authors show that the functor $\pi_*:D^b(\mathcal{A}_\mathcal{D}-\mathrm{mod})\rightarrow D^b(\mathcal{D})$ is a Serre quotient by the category generated by a collection of exceptional sheaves $\mathcal{E}^\pm_m(j)$.

  \subsubsection{The mirror Riemann surface}

  To a nodal stacky cycle $\mathcal{D}$ one can associate a Riemann surface obtained as follows: for each component $\mathcal{D}_m$, take a cylinder $C_m$. Moreover, for each stacky node $q_m\in \mathcal{D}_m\cap \mathcal{D}_{m+1}$, attach $r_m$ $1$-handles $\{H_{m,a}^-\}_{a=0}^{r_m-1}$ from $\partial_{-}C_m$ to $\partial_{+} C_{m+1}$ whose attaching $S^0$ consists of the two points $\frac{a}{r_m}\in \mathbb{R}/\mathbb{Z}\simeq \partial_{-}C_m,-\frac{k_m a}{r_m}\in \mathbb{R}/\mathbb{Z}\simeq \partial_{+}C_{m+1}$. The resulting surface we denote by $\Sigma^0$. Each $1$-handle is furthermore equipped with two stops, with the total stopping set denoted by $\Lambda$. If one only attaches the $1$-handles $H^-_{r_m-a}$ with $a$ being a special weight for $\mathbb{C}^2/\frac{1}{r_m}(1,l_m)$, the resulting Riemann surface is a torus $T^0$ which we call the \textit{special torus}.
  
  \begin{defn}[Special torus] \label{def:specialtorus} The torus obtained by taking the cylinders $C_m$ and attaching only the special $1-$handles $H^-_{r_m-i}$ with $i\in I(r_m,l_m)$ is denoted $T^0$. After attaching the non-special $1$-handles, one obtains the Riemann surface with boundary $\Sigma^0$.
    
  \end{defn}
  Recall that the partially wrapped Fukaya category $\mathcal{W}(\Sigma^0, \Lambda)$ has objects arcs and circles and whose morphism spaces are intersection points, as well as boundary paths following the orientation of $\partial\Sigma^0$. It is shown in \cite{lekili_auslander_2018} that $\mathcal{W}(\Sigma^0;\Lambda^0)$ admits generators in the form of a strong exceptional collection consisting of objects \begin{itemize}
    \item $S_{m,j}, j=0,1,\dots, r_{m}-1$. These are the cocores of the handles $H^-_{m,j}$.
    \item $P^+_{m,j}, j=0,1,\dots, r_m$. These correspond to $\mathcal{P}_m(0,j-1)$.
     \item $P^-_{m,j}, j=0,1,\dots, r_{m-1}$. These correspond to $\mathcal{P}_m(j,-1)$.
  \end{itemize}
  By general results about stop removal (see for example \cite{ganatra_sectorial_2023}), the wrapped Fukaya category $\mathcal{W}(\Sigma^0)$ is equivalent to the quotient of $\mathcal{W}(\Sigma^0; \Lambda)$ by the collection of linking disks $E^\pm_{m,j}$.

  \subsubsection{Mirror symmetry at the level of the boundary}

  \begin{thm}[Lekili-Polishchuk] There is an equivalence of $A_\infty$ categories $$D^b(\mathcal{A}_\mathcal{D}-\mathrm{mod})\simeq \mathcal{W}(\Sigma^0; \Lambda)$$ which identifies the strong exceptional collections \begin{gather*}
    \mathcal{S}_{q_m}\{-k_mj\}[-1]\leftrightarrow S_{m,j} , \qquad
    \mathcal{P}_m(j,-1)\leftrightarrow P^-_{m,j} , \qquad
    \mathcal{P}_m(0,j-1)\leftrightarrow P^+_{m,j}
  \end{gather*}
and moreover identifies the exceptional objects $$\mathcal{E}^+_m(j)\leftrightarrow E^+_{m,j}[-1], \qquad \mathcal{E}^-_m(j)\leftrightarrow E^-_{m,j-1}[-1]$$ which induces an equivalence $$\varphi:D^b(\mathcal{D})\simeq \mathcal{W}(\Sigma^0)$$

Furthermore, both of these equivalences identify the fully faithful images of $\mathrm{Perf}(\mathcal{D})$ in $D^b(\mathcal{A}_\mathcal{D}-\mathrm{mod})$ (resp. $D^b(\mathcal{D})$) with the fully faithful image of $\mathcal{F}(\Sigma^0)$ in $\mathcal{W}(\Sigma^0;\Lambda)$ (resp. $\mathcal{W}(\Sigma^0)$).

  \end{thm}
  \begin{rem}
  To get the functor $\varphi$, one needs to invert the $A_\infty$ equivalence $\mathcal{W}(\Sigma^0;\Lambda)/E\rightarrow \mathcal{W}(\Sigma^0)$, as in \cite[Theorem 2.9]{seidel_fukaya_2008} and \cite[Theorem 8.6]{fukaya_floer_nodate}.
  \end{rem}
  \begin{figure}[H]
    \input{./TikzFiles/LekiliPolishchukGenerators.tikz}
            \caption{Example when $\mathcal{D}_{m+1}\simeq \mathbb{P}^1_{5,5}$ with $k_m=2, k_{m+1}=1$. }

  \end{figure}
  \begin{rem}
  The labelling for the handles $H^-_{m,j}$ is the same as the one for $S_{m,j}$ namely $m$ is the index corresponding to the point on the left lying on $\partial_+ C_m$.
  It will be convenient to denote the same handle using the index on the right, lying on $\partial_- C_{m+1}$. We set $H^+_{m, -ka}:=H^-_{m,a}$.
  \end{rem}
  \begin{rem}
  The grading data for the Riemann surface $\Sigma^0$ is induced by the horizontal line field.
  \end{rem}
  \begin{rem}[Incorporating all line bundles on the $\mathcal{D}_m$]\label{rem:alllinebundles}
  In this equivalence, the skyscraper sheaf at a smooth point on $\mathcal{D}_m$ is identified with the meridian of the cylinder $C_m$ (with a local system). To see this, note that there is an exact sequence $\mathcal{P}_{m}(0,-1)\rightarrow \mathcal{P}_m(r_{m-1},-1)\rightarrow \mathcal{O}_{pt}$. Applying $\varphi$, the skyscraper sheaf becomes a twisted complex built out of $P^-_{m,0}$ and $P^-_{m, r_{m-1}}$ which is quasi-isomorphic to the meridian of $C_m$. Hence, while a priori one uses only the generators $P^-_{m,j}, 0\leq j \leq r_{m-1}$ mirror to $\mathcal{P}_{m}(j,-1)$, one can also incorporate any $j$: on the B-side, this corresponds to twisting $\mathcal{P}_{m}(j,-1)$ (with $0\leq j \leq r_{m-1})$ multiple times by $\mathcal{O}_{pt}$, whereas on the A-side, this corresponds to applying multiple positive or negative Dehn twists on $P^-_{m,j}$ along the meridian of $C_m$. The same holds for the Lagrangians $P^+_{m,j}$ and the corresponding $\mathcal{P}_m(0,j-1)$.
  \end{rem}
  \subsection{The special torus and a fully faithful embedding}
  The functor $\Phi:D^b(Y)\rightarrow D^b(\mathcal{X})$ is a priori defined on the level of the surfaces $Y$ and $\mathcal{D}$. We show that there is a boundary-level analogue of $\Phi$.
  \subsubsection{A fully faithful embedding on the B-side}
  \begin{prop}\label{prop:fullyfaithfulbdry}
    There is a fully faithful functor $\Phi^\partial$ fitting into a commutative diagram
 \[\begin{tikzcd}
D^b(Y) \arrow[r, hook, "\Phi"] \arrow[d, "\iota_D^*"'] & D^b(\mathcal{X}) \arrow[d, "\iota_{\mathcal{D}}^*"] \\
\mathrm{Perf}(D) \arrow[r, hook, "\Phi^\partial"]      & \mathrm{Perf}(\mathcal{D})                         
\end{tikzcd}\]
\end{prop}

\begin{proof}
    First of all, we define the functor $\Phi^\partial$ on $D^b(D)$ and later show that it respects perfect complexes. Consider $\tilde{D}$ defined via the outer pullback square (in the category of Deligne-Mumford stacks) \[\begin{tikzcd}
\tilde{D} \arrow[d, "\iota_{\tilde{D}}"', hook] \arrow[r] & D\times \mathcal{D} \arrow[d,"\iota_{D\times \mathcal{D}}", hook] \arrow[r] & \mathcal{D} \arrow[d, "\iota_{\mathcal{D}}", hook] \\
\mathcal{Z} \arrow[r]           & Y\times \mathcal{X} \arrow[r]       & \mathcal{X}      
\end{tikzcd}\]
By a version of the pasting law, since $D\times \mathcal{D} \subset Y\times \mathcal{D}$ is a monomorphism, this implies that the left square is also Cartesian.
By construction, $\tilde{D}$ is the substack of $\mathcal{Z}$ lying over $\mathcal{D}$ which is the same as the substack of $\mathcal{Z}$ lying over $D$ i.e. the following diagram is also Cartesian: \[\begin{tikzcd}
\tilde{D} \arrow[d, "\iota_{\tilde{D}}"', hook] \arrow[r] & D \arrow[d, "\iota_D", hook] \\
\mathcal{Z} \arrow[r]                                       & Y                           
\end{tikzcd}\]

Recall that the functor $\Phi$ is defined as $\Phi(-)=\pi_{\mathcal{X}*}(\pi_Y^*-\otimes \,\mathcal{O}_\mathcal{Z})$ which is the same as $\nu_*\mu^*$ by the projection formula. We define a functor \begin{gather*}
    \Phi^\partial: D^b(D)\rightarrow D^b(\mathcal{D})\\
    \Phi^\partial(-):=\pi_{\mathcal{D}*}(\pi_D^*- \otimes \,\mathcal{O}_\mathcal{D})=\tilde{\nu}_{*}\tilde{\mu}^{ *}(-)
\end{gather*}
with the arrows as in the diagram \[\begin{tikzcd}
\tilde{D} \arrow[rdddd, "\tilde{\mu}"', bend right] \arrow[rrrrd, "\tilde{\nu}", bend left] \arrow[rd, hook] &         &         &            &        \\
 &  D\times \mathcal{D} \arrow[ddd, "\pi_D"'] \arrow[rrr, "\pi_{\mathcal{D}}"] &     &     & \mathcal{D} \arrow[dd, "\iota_{\mathcal{D}}"] \\
 &     & \mathcal{Z} \arrow[rd, hook] \arrow[rdd, "\mu"', bend right] \arrow[rrd, "\nu", bend left] &    &    \\
 &     &      &  Y\times \mathcal{X} \arrow[d, "\pi_Y"'] \arrow[r, "\pi_\mathcal{X}"] & \mathcal{X}   \\
 & D \arrow[rr, "\iota_D"'] &    & Y     &                                      
\end{tikzcd}\]

First of all, a straightforward application of base change leads to
\begin{gather*}
    \Phi^\partial \iota_D^* = \tilde{\nu}_*\tilde{\mu}^*\iota_D^*=\tilde{\nu}_* \iota_{\tilde{D}}^* \mu^* \underbrace{=}_{\text{base change}} \iota_{\mathcal{D}}^*\nu_* \mu^*=\iota_{\mathcal{D}}^* \Phi
\end{gather*} and similarly 
\begin{gather*}
    \Phi \iota_{D*}=\nu_* \mu^* \iota_{D*} \underbrace{=}_{\text{base change}} \nu_* \iota_{\tilde{D}*}\tilde{\mu}^*=\iota_{\mathcal{D}*}\tilde{\nu}_*\tilde{\mu}^*=\iota_{\mathcal{D}*}\Phi^\partial
\end{gather*}

The functor $\Phi^\partial$ is given by a Fourier-Mukai kernel hence admits a left adjoint $\Psi^\partial$ \cite[Proposition 5.9]{huybrechts_fourier-mukai_2006}. The fact that $\Phi^\partial$ and $\Phi$ commute with restriction and pushforward implies that the left adjoints $\Psi^\partial$ and $\Psi$ also commute with the pushforwards and pullbacks: the functor $\iota_D^*\Psi$ is left adjoint to $\Phi \iota_{D*}$ which is equal to $\iota_{\mathcal{D}*}\Phi^\partial$ which has a left adjoint $\Psi^\partial \iota_D^*$, concluding that $\iota_D^*\Psi=\Psi^\partial \iota_{\mathcal{D}}^*$. The same kind of argument implies that $\Psi \iota_{\mathcal{D}*}=\iota_{D*}\Psi^\partial$.

It remains to show that $\Phi^\partial$ respects perfect complexes and is fully faithful on Perf. We recall that $\mathrm{Perf}(D)$ is generated by $\mathcal{O}_D$ and a set of skyscraper sheaves $\mathcal{O}_p$, one for each component of $D$. We proceed to compute $\Phi^\partial$ on this generating set.

The functor $\Phi^\partial$ acts as the identity outside the fundamental cycle $F$ of $Y\rightarrow X$. Moreover, $\Phi^\partial \mathcal{O}_D=\iota_{\mathcal{D}}^*\Phi \mathcal{O}_Y=\mathcal{O}_{\mathcal{D}}$ - this is because $\nu_*\mathcal{O}_\mathcal{Z}=\mathcal{O}_\mathcal{X}$ ($\nu$ is birational, proper with connected fibers) and therefore $\Phi^\partial$ gives an isomorphism $\mathrm{Ext}^0_D(\mathcal{O}_D, \mathcal{O}_D)\simeq \mathrm{Ext}^0_{\mathcal{D}}(\mathcal{O}_\mathcal{D}, \mathcal{O}_\mathcal{D})$ since both are generated by the identity morphism. The generator of $\mathrm{Ext}^1_D(\mathcal{O}_D, \mathcal{O}_D)$ is represented by the bundle which is the pullback of $\mathcal{O}_D\rightarrow \mathcal{O}_D(p)\rightarrow \mathcal{O}_p$ along the map $\mathcal{O}_D\rightarrow \mathcal{O}_p$, which is in other words the composition of $m_p\in\mathrm{Ext}^0_D(\mathcal{O}_D,\mathcal{O}_p)$ and $n_p\in\mathrm{Ext}^1_D(\mathcal{O}_p, \mathcal{O}_D)$. Applying $\Phi^\partial$, we get that the composition $\Phi^\partial m_p\in\mathrm{Ext}^0_\mathcal{D}(\mathcal{O}_\mathcal{D},\Phi^\partial\mathcal{O}_p)=\mathbb{C}$ with $\Phi^\partial n_p\in\mathrm{Ext}^1_\mathcal{D}(\Phi^\partial\mathcal{O}_p, \mathcal{O}_\mathcal{D})=\mathbb{C}$ is given similarly by the extension given by the pullback $\mathcal{O}_\mathcal{D}\times_{\mathcal{O}_p}\mathcal{O}_\mathcal{D} (p)$ which is the generator of $\mathrm{Ext}^1_{\mathcal{D}}( \mathcal{O}_\mathcal{D},\mathcal{O}_\mathcal{D})$. This shows that $\Phi$ is an isomorphism on the morphism spaces involving $\mathcal{O}_D$ and $\mathcal{O}_p$ with $p\notin F$.

It remains to compute $\Phi^\partial$ on the points which are on $F$. This is a purely local computation, so for the rest of the proof, we will assume $\mathcal{X}=[\mathbb{C}^2/\frac{1}{r}(1,l)]$ and $\mathcal{D}$ is the locus $[uv=0/\mu_r]$, which consists of two copies of $[\mathbb{C}/\mathbb{Z}_r]$ glued at a $\frac{1}{r}(1,l)$ point.
Let $E_i$ be the exceptional curve corresponding to the special representation $\rho_i$. By Wunram's theorem and \cite[Section 3.1]{gugiatti_full_2023}, $\Phi\mathcal{R}_{\rho_i}^\lor=\mathcal{O}\otimes \rho_i$ for special $\rho_i$, where $\mathcal{R}_{\rho_i}^\lor$ is a line bundle on $Y=\widetilde{\mathbb{C}^2/\frac{1}{r}(1,l)}$ whose first Chern class satisfies $$c_1(\mathcal{R}_{\rho_i}^\lor)\cdot E_j=-\delta_{i,j}$$

The restriction of such a line bundle to $D$ is exactly $\mathcal{O}(-p_i)$ for some $p_i\in E_i$. Hence, $$\Phi^\partial(\mathcal{O}_{p_i})=\Phi^\partial [\mathcal{O}_D(-p_i)\rightarrow \mathcal{O}_D]=\Phi^\partial \iota_D^* [\mathcal{R}^\lor_{\rho_i}\rightarrow \mathcal{O}_Y]=\iota_{\mathcal{D}}^* \Phi [\mathcal{R}^\lor_{\rho_i}\rightarrow \mathcal{O}_Y]=[\mathcal{O}_{\mathcal{D}}\otimes \rho_i \rightarrow \mathcal{O}_{\mathcal{D}}]$$
We will now describe the morphism $\mathcal{O}_\mathcal{D}\otimes \rho_i\rightarrow \mathcal{O}_\mathcal{D}$. First of all, since $\Phi \mathcal{R}^\lor_{\rho_0}=\mathcal{O}_{\mathcal{X}}$ and $\Psi \Phi \simeq \mathrm{id}$, we know that $\Psi \mathcal{O}_{\mathcal{X}}=\mathcal{O}_Y$ and hence $\Psi^\partial \mathcal{O}_{\mathcal{D}}=\mathcal{O}_D$. Therefore $$\mathbb{C}[0]=\mathrm{Ext}^\bullet_D(\Psi^\partial\mathcal{O}_{\mathcal{D}},\mathcal{O}_{p_i})=\mathrm{Ext}^\bullet_{\mathcal{D}}(\mathcal{O}_{\mathcal{D}}, [\mathcal{O}_{\mathcal{D}}\otimes \rho_i \xrightarrow{s} \mathcal{O}_{\mathcal{D}}])$$

The morphism $s$ is a sum of weight $i$ monomials on the node which are elements of the set  $\{ u^i, v^j, u^{r+i}, v^{r+j},\dots \}$, where $jl\equiv i \mod r$. We will show that $s$ has to be of the form $\lambda_1 u^i+\lambda_2 v^j$ with $\lambda_1, \lambda_2$ nonzero constants.

The cohomology long exact sequence associated to the cone of $s$ involves the cohomology groups $H^\bullet(\mathcal{D}, \mathcal{O}_\mathcal{D}\otimes \rho_i), H^\bullet(\mathcal{D}, \mathcal{O}_\mathcal{D})$ which are both concentrated in degree zero. As such, the long exact sequence takes the shape of a short exact sequence $$0 \rightarrow \mathbb{C}\langle u^{r-i}, v^{r-j},u^{2r-i},v^{2r-j},\dots \rangle \xrightarrow{s} \mathbb{C}\langle 1,u^r,v^r,u^{2r},v^{2r},\dots\rangle\rightarrow \mathbb{C}\rightarrow 0$$
Write $$s=\alpha u^{i}+\beta v^j$$ where $\alpha \in \mathbb{C}[u^r], \beta \in \mathbb{C}[v^r]$.
Then $$(u^{r-i} \gamma + v^{r-j}\delta)s=u^r(\alpha\gamma)+v^r(\beta \delta), \quad \gamma \in \mathbb{C}[u^r], \delta \in \mathbb{C}[v^r]$$
The morphism $s$ has no kernel precisely when both $\alpha$ and $\beta$ are nonzero. Moreover, the cokernel always contains $1$. It is one-dimensional precisely when $u^r, v^r\in \mathrm{im}\, s$. This is only possible when $s=\lambda_1 u^i+\lambda_2 v^j$ with $\lambda_1, \lambda_2 \in \mathbb{C}^\times$.

In conclusion, the result upon applying $\Phi^\partial$ on a skyscraper sheaf $\mathcal{O}_{p_i}$ is the sheaf supported at the orbifold point associated to the module $$\mathbb{C}[u,v]/(uv, \lambda_1 u^i+\lambda_2 v^j)$$ which is perfect. The constants $\lambda_1, \lambda_2$ depend on the point $p_i\in E_i$.

It remains to show that $\Phi^\partial$ is fully faithful on the hom spaces involving $\mathcal{O}_D$ and $\mathcal{O}_{p_i}$.
There are unique morphisms $m_i\in \mathrm{Ext}^0(\mathcal{O}, \mathcal{O}_{p_i}), n_i \in \mathrm{Ext}^1(\mathcal{O}_{p_i}, \mathcal{O})$ satisfying $0\neq m_in_i=m_jn_j \in \mathrm{Ext}^1_D(\mathcal{O}_D, \mathcal{O}_D)$. Hence, since $\Phi^\partial$ identifies the cohomology of $\mathcal{O}_D$ with the cohomology of $\Phi^\partial\mathcal{O}_D$ we conclude that $\Phi^\partial m_i, \Phi^\partial n_i \neq 0$ and therefore the functor is faithful. Observing that $\mathrm{Ext}_{\mathcal{D}}^\bullet(\mathcal{O}_\mathcal{D}, \Phi^\partial \mathcal{O}_{p_i})=\mathbb{C}[0]$ and $\mathrm{Ext}^\bullet_{\mathcal{D}}( \Phi^\partial \mathcal{O}_{p_i}, \mathcal{O}_\mathcal{D})\simeq \mathrm{Ext}^{1-\bullet}_{\mathcal{D}}( \mathcal{O}_\mathcal{D}, \Phi^\partial \mathcal{O}_{p_i}) =\mathbb{C}[1]$ implies that the functor is also full. 
\end{proof}

We proceed to express the sheaf $\Phi^\partial \mathcal{O}_{p_i}$ using the generators of $D^b(\mathcal{D})$.
\begin{prop}
  Let $p_i$ be a point on $E_i\subset D$ which contracts to a stacky node $q_m$. Then 
the object $\Phi^\partial \mathcal{O}_{p_i}\in D^b(\mathcal{D})$ is quasi-isomorphic to the complex
\begin{equation}\label{eq:complexforimageofpoint}\pi_*[\mathcal{P}_m(0,r_m-i)\xrightarrow{}\mathcal{P}_m(0,r_m)] \oplus \pi_*[\mathcal{P}_{m+1}(r_m-j,-1)\xrightarrow{}\mathcal{P}_{m+1}(r_m,-1)]\xrightarrow{}\pi_*\mathcal{S}_{q_m}\oplus \pi_*\mathcal{S}_{q_m}\{ki\}[1]\end{equation}
\end{prop}
\begin{proof}
  The complex in question is supported at the node, so we can restrict to a neighbourhood of the node which is $\{uv=0\}$ modulo the $\mu_r$ action $\zeta\cdot (u,v)=(\zeta^k u,\zeta v)$. Sheaves on this neighbourhood are the same as $\mu_r$-equivariant $R$-modules, where $R=\mathbb{C}[u,v]/(uv)$. The claim follows from the exact sequence of two-term complexes: 
  \[\begin{tikzcd} R\{ki\} \arrow[r] \arrow[d, "\lambda_1 u^i+\lambda_2 v^j"'] & (R/u\oplus R/v)\{ki\} \arrow[r] \arrow[d] & {R/(u,v)\{ki\}} \arrow[d, "0"] \\ R \arrow[r] & (R/u\oplus R/v)\arrow[r] & {R/(u,v)} \end{tikzcd}\]
  The object $\Phi^\partial \mathcal{O}_{p_i}$ is equivalent to the leftmost column, whereas the remaining two columns form the restriction of the complex \ref{eq:complexforimageofpoint} to a neighbourhood of the node $q$.
\end{proof}

\subsubsection{The fully faithful embedding on the A-side}
By using the surgery exact triangle in $\mathcal{W}(\Sigma)$, we can describe the underlying Lagrangian of the mirror to $\Phi^\partial \mathcal{O}_{p_i}$. First of all, by \ref{rem:alllinebundles}, the mirror to $\mathcal{P}_m(0,r_m)$ which we denote by $P^+_{m,r_m+1}$ is described by applying a negative Dehn twist to $P^+_{m,1}$ (which is mirror to $\mathcal{P}_m(0,0)$) along the meridian of $C_m$.  
Therefore, upon applying $\varphi$ to the complex representing $\Phi^\partial \mathcal{O}_{p_i}$ we get the twisted complex \begin{equation}\label{eq:complexmirrortoimageofpoint} [P^+_{m,r_m-i+1}[1]\dashrightarrow P^+_{m,r_m+1}][1] \oplus [P^-_{m+1, r_m-j}[1]\dashrightarrow P^-_{m+1, r_m}][1]\dashrightarrow S_{m,-i}[2]\oplus S_{m, 0}[1]\end{equation}
\begin{rem}
Recall that the cone on a degree zero morphism $f:A\rightarrow B$ is described via a twisted complex $(A[1]\oplus B, d_f)$ with the differential having degree $1$. Our convention is to denote this differential by a dashed arrow $A[1]\dashrightarrow B$, so as to not confuse it with the genuine morphism $f$.
\end{rem}

\begin{prop}\label{prop:specialmeridian}
The complex of Lagrangians \ref{eq:complexmirrortoimageofpoint} is quasi-isomorphic in $\mathcal{W}(\Sigma^0)$ to a Lagrangian $S^1$ (with a local system) denoted $M_{m,i}$ which passes through the handles $H^-_{m,0}$ and $H^-_{m,-i}$ once. This Lagrangian $M_{m,i}$ is identified with a meridian in the special torus $T^0$.
\end{prop}
\begin{proof}
The first term in this complex is quasi-isomorphic to the sum of two Lagrangian arcs. The first of these arcs starts at the top of $H^-_{0,m}$ and goes up inside $C_m$ to reach the bottom of $H^-_{m,-i}$. The other arc starts at the top of $H^-_{m,0}=H^+_{m,0}$ and then goes down, in $C_{m+1}$, to reach the top of $H^-_{m,-i}=H^+_{m, ki}=H^+_{m,j}$. The morphisms to the cocores come from the four unique Reeb chords to the cocores $S_{m,-i}$ and $S_{m,0}$ which pass through a single stop, which are the morphisms $$P^+_{m, r_m+1}\rightarrow S_{m,0}[1], \qquad P^-_{m+1, r_m}\rightarrow S_{m,0}[1],\qquad P^+_{m, r_m-i+1}\rightarrow S_{m,-i}[1], \qquad P^-_{m+1, r_m-j}\rightarrow S_{m,-i}[1]$$ Applying surgery along these Reeb chords (or, alternatively, using functors from $\mathcal{W}(0;3)$ as in \cite[Section 2.1]{lekili_auslander_2018}), the result follows.
\end{proof}

\begin{eg}
Consider a node $q_m$ of type $\frac{1}{5}(2,1)\equiv \frac{1}{5}(1,3)$ with $k=2$ and $i=3, j=1$. On the right side, we have the Lagrangians $P^-_{m+1,4}$ and $P^{-}_{m+1,5}$ which are glued along the unique degree zero morphism. On the left, we have the Lagrangian $P^{+}_{m+1, 5-3+1}$ as well as $$P^+_{m, 6}:=T_{C_m}^{-1}P^+_{m,1}$$
After performing surgery along Reeb chords, the result is a Lagrangian as in the right side of Figure \ref{fig:meridian}.
\end{eg}
\begin{figure}[H]
      \input{./TikzFiles/meridiansurgery.tikz}
         \caption{\label{fig:meridian} Left: the complex of Lagrangians which is given by $\varphi$ applied to the complex representing $\Phi^\partial \mathcal{O}_{p_i}$. Right: a Lagrangian which is isomorphic to the twisted complex on the left.}
      \end{figure}

      \begin{prop}\label{prop:fullyfaithfulspecialtorus}
        Under the equivalences $\mathrm{Perf}(D)\simeq \mathcal{F}(T^0), \mathrm{Perf}(\mathcal{D})\simeq \mathcal{F}(\Sigma^0)$, the functor $\Phi^\partial$ corresponds to a functor on the Fukaya category which is induced (on the level of underlying Lagrangians) by the inclusion of the special torus $T^0$ into $\Sigma^0$. In other words, there is a commutative diagram of functors: \[\begin{tikzcd}
\mathrm{Perf}(D) \arrow[r, "\Phi^\partial", hook] \arrow[d, "\simeq"', no head] & \mathrm{Perf}(\mathcal{D}) \arrow[d, "\simeq", no head] \\
\mathcal{F}(T^0) \arrow[r, hook]                                               & \mathcal{F}(\Sigma^0)                                 
\end{tikzcd}\]
      \end{prop}
      \begin{proof}
        The claim follows by combining the previous propositions and using the fact that $\mathcal{F}(T^0)$ is generated by a longitude and a collection of meridians. The meridians in $T^0$ (potentially equipped with a local system) get sent to $\mathcal{O}_p \in \mathrm{Perf}(D)$ for $p\in D$. The effect of $\Phi^\partial$ on these skyscrapers was described in the proof of Proposition \ref{prop:fullyfaithfulbdry}: for the $p$'s which do not lie on the fundamental cycle, the result is still a skyscraper sheaf (now on $\mathcal{D}$) which is mirror to a meridian on $\Sigma^0$. 
        
        On the points which lie on the fundamental cycle, we know by Propositions \ref{prop:fullyfaithfulbdry}, \ref{prop:specialmeridian} that the result is a complex of sheaves in $D^b(\mathcal{D})$ which, under the equivalence $D^b(\mathcal{D})\simeq \mathcal{W}(\Sigma^0)$ is given by a twisted complex equivalent to the meridian in the special torus as in Figure \ref{fig:meridian}. Finally, since $\Phi \mathcal{O}=\mathcal{O}$, the reference longitude remains unaffected.
      \end{proof}
    \begin{rem}
    In the above equivalence, one needs to pick which points on $D$ go to the exact Lagrangian $S^1$s with trivial local system: the other points on the same rational curve component will be sent to the same circle but with nontrivial local system. We will later ensure that it is the point $-1\in \mathbb{C}^\times$ in each rational curve component that is identified with the circle with trivial local system. 
    \end{rem}
  
  \subsection{Restricting the non-special sheaves to the boundary}
  We proceed to understand the restrictions of the Ishii-Ueda sheaves $N_d$ from $D^b(\mathcal{X})$ to $\mathrm{Perf}(\mathcal{D})$.  These sheaves are in fact pushed forward from the boundary, so we can use a push-pull exact triangle to compute.

  Let $q_m$ be an orbifold point of type $\frac{1}{r}(k,1)\equiv \frac{1}{r}(1,l)$ which is also a nodal point of $\mathcal{D}$. 
  \begin{defn}
  Let $\mathcal{S}_{q_m}^j$ be the $j$-thickened skyscraper in the $v$ direction on $\mathcal{D}$, namely the sheaf supported at $q_m$ associated to the module $\mathbb{C}[u,v]/(uv,u,v^j)$. For $d$ non-special with $i_{t-1}>d>i_{t}$, denote by $B_{m,d}:=\mathcal{S}_{q_m}^{j_t}\{l^{-1}d-j_t+1\}$.
  \end{defn}
  The sheaf $B_{m,d}$ is the sheaf such that $\iota_* B_{m,d} =N^{q_m}_d$ is the exceptional sheaf in the Ishii-Ueda collection.
  \begin{prop}
    The sheaf $\mathcal{S}_q^j$ in $D^b(\mathcal{D})$ is equivalent to the cone on the unique degree zero morphism $\mathcal{P}_{m+1}(r_m-j,-1)\rightarrow \mathcal{P}_{m+1}(r_m,-1)$ and $\mathcal{S}_q^j\{s\}$ is similarly given by the cone on the unique degree zero morphism $\mathcal{P}_{m+1}(r_m-j-s,-1)\rightarrow \mathcal{P}_{m+1}(r_m-s,-1)$. Moreover, there is an exact triangle $$\iota_\mathcal{D}^*\iota_{\mathcal{D}^*}\mathcal{S}_q^j\rightarrow \mathcal{S}_q^j \rightarrow \mathcal{S}_q^j \{k+1\}[2]$$
    
\end{prop}
\begin{proof}
    The first claim is a local computation: near the node, the morphism is given by $R/(u) \{j+s\} \xrightarrow{v^j}R/(u)\{s\}$ whose cone is $R/(u,v^j)$, where $R=\mathbb{C}[u,v]/(uv)$. 
    The second claim follows from the exact triangle of functors (as in \cite[Lemma 2.8]{kuznetsov_serre_2021}) $$\iota_\mathcal{D}^*\iota_{\mathcal{D}*}\rightarrow \mathrm{id}\rightarrow -\otimes \mathcal{O}(-\mathcal{D})|_{\mathcal{D}}[2]$$ and using the fact that $\mathcal{D}$ is anticanonical, so that $\mathcal{O}(-\mathcal{D})=\omega_{\mathcal{X}}$. The sheaf is supported near the node, where the canonical sheaf restricts to $\mathcal{O}\otimes \rho_{det}$. Tensoring with this is the same as applying the twist $\{k+1\}$.
\end{proof}
\begin{defn}\label{def:Ad}
  We define the arc $A_{m,d}\subset \Sigma^0$ to be an arc which begins at the bottom of $H^+_{m, l^{-1}d}$, goes inside the cylinder $C_{m+1}$ and ends at the top of $H^+_{m,l^{-1}d-(j-1)}$. We define the arc $A'_{m,d}$ to be a translation of $A_{m,d}$, beginning at the bottom of $H^+_{m, l^{-1}d+k+1}$, going inside the cylinder $C_{m+1}$ and ending at the top of $H^+_{m,l^{-1}d-(j-1)+k+1}$.
\end{defn}
A straightforward application of Proposition \ref{propappendix:overlap} from the Appendix implies that these two arcs do not intersect.

\begin{prop}\label{prop:mirrortoBd}
The objects $\varphi(B_{m,d})[1], \varphi(B_{m,d}\{k+1\})[1] \in \mathcal{W}(\Sigma^0)$ are equivalent to Lagrangian branes supported on $A_{m,d}, A'_{m,d}$, respectively.
\end{prop}
\begin{proof}
Recall that $B_{m,d}$ is equivalent to the cone on the unique degree zero morphism $\mathcal{P}_{m+1}(r_m-j-\overline{l^{-1}d-j+1},-1)\rightarrow \mathcal{P}_{m+1}(r_m-\overline{l^{-1}d-j+1},-1)$, where $\overline{l^{-1}d-j+1}$ denotes the residue modulo $r_m$. In fact, this residue satisfies $\overline{l^{-1}d-j+1}\leq r_m-j$: supposing for the sake of contradiction that $\overline{l^{-1}d-j+1}=r_m-j+x$ where $0<x\leq j$, then this would imply that $d-lj+l\equiv lx-lj$ which in turn implies that $d-(x-1)l\equiv 0$. However, by \cite[Lemma 3.9]{ishii_special_2013}, the residue of $d-(x-1)l$ mod $r_m$ is the weight of a non-special representation bigger than $d$, so it cannot be zero.

Hence, upon applying $\varphi$, we see that $\varphi(B_d)[1]$ is isomorphic to the twisted complex $$P^-_{m+1, r_m-j-\overline{l^{-1}d-j+1}}[1]\dashrightarrow P^-_{m+1,r_m-\overline{l^{-1}d-j+1}}$$
where the morphism in this twisted complex comes from the unique degree zero morphism. This is equivalent (after performing surgery, or equivalently observing a triangle giving a functor from the $A_2$ sector $\mathcal{W}(0,3)$) to the arc $A_{m,d}$ from Definition \ref{def:Ad}. The claim about $\varphi(B_d\{k+1\})$ follows along identical lines: $B_d\{k+1\}$ is equivalent to the cone on the morphism $\mathcal{P}_{m+1}(r_m-j-\overline{l^{-1}d-j+1}-k-1,-1)\rightarrow \mathcal{P}_{m+1}(r_m-\overline{l^{-1}d-j+1}-k-1,-1)$. Upon applying $\varphi$ (and recalling Remark \ref{rem:alllinebundles}), one gets a twisted complex built out of two Lagrangian arcs which is quasi-equivalent to $A'_{m,d}$.
\end{proof}

\begin{defn}\label{def:Ld}
  Consider the cores $c_{1+d+l}, c_{1+d+l-j_t l}, c_{d+l}$ and $c_{d+l-j_tl}$ of the respective handles on $\Sigma^0$ $H^-_{m,-(d+l)}, H^-_{m,-(d+l+1)}, H^-_{m,-(d-i_t+l)}, H^-_{m,-(d-i_t+l+1)}$ \footnote{If there are any duplicates, take disjoint parallel translates so that $c_{1+d+l}$ is above $c_{d+l}$, which in turn is above $c_{1+d+l-j_tl}$ which finally sits above $c_{d+l-j_tl}$}. Join these together by:
  \begin{itemize}
    \item An arc going from $\partial_- c_{1+d+l}$ down to $\partial_-c_{d+l}$
    \item An arc going from $\partial_+c_{1+d+l}$ up to $\partial_+c_{1+d+l-j_tl}$
    \item An arc going from $\partial_-c_{1+d+l-j_tl}$ down to $\partial_-c_{d+l-j_t}$
    \item An arc going from $\partial_+c_{d+l}$ up to $\partial_+c_{d+l-j_tl}$
  \end{itemize}
  The resulting Lagrangian $S^1$ is denoted $L_{m,d}\subset\Sigma^0$.
\end{defn}

\begin{prop}\label{prop:mirrortoLd}
        The object $\varphi(\iota_\mathcal{D}^* \iota_{\mathcal{D}*}) B_{m,d}\in \mathcal{W}(\Sigma^0)$ is equivalent to a Lagrangian brane supported on $L_{m,d}$ as in Definition \ref{def:Ld}.
\end{prop}
\begin{proof}
Recall that there is an exact triangle $\iota_\mathcal{D}^* \iota_{\mathcal{D}*} B_{m,d} \rightarrow B_{m,d} \rightarrow B_{m,d}\{k+1\}[2]$. 
There is a 2-dimensional $\mathrm{Ext}^2$ space between $B_{m,d}$ and $B_{m,d}\{k+1\}$\footnote{This follows by using the local projective resolution $\dots \rightarrow R\{k+1\}\oplus R\{k+j\}\rightarrow R\{k\} \oplus R\{j\}\rightarrow R$ of $\mathcal{S}_q^j$, with the $\mathrm{Ext}^2$ group identified with $\mathrm{Hom}_R(R\{k+1\}\oplus R\{k+j\}, R/(u,v^j)\{k+1\})=\mathbb{C}\langle 1\rangle \oplus \mathbb{C}\langle v^{j-1}\rangle$}. The complex $[B_{m,d}\rightarrow B_{m,d}\{k+1\}[2]]$ is perfect since it is quasi isomorphic to $\iota_\mathcal{D}^*N^{q_m}_d$. Applying $\varphi$ to this complex, the result is quasi-isomorphic to the twisted complex \begin{equation}\label{eq:twistedcomplex}
A_{m,d}\dashrightarrow A'_{m,d}[1]
\end{equation}
with connecting differential induced by a degree $2$ morphism. There are two Reeb chords generating $HW^2(A_{m,d}, A'_{m,d})$ described by the following paths on $\partial \Sigma^0$ which traverse exactly two stops:
\begin{itemize}
\item From the top of the handle $H^+_{m,kd-j+1}=H^-_{m, -(d-l(j-1))}$ follow the Reeb flow, reaching the bottom of handle $H^-_{m,-(d-l(j-1))-1}=H^+_{m,kd-(j-1)+k}$ and afterwards the top of $H^+_{m,kd-(j-1)+k+1}$ where the Reeb chord reaches the top of $A'_{m,d}$.
\item From the bottom of the handle $H^+_{m,kd}$ follow the Reeb flow, reaching the top of $H^+_{m,kd+1}=H^-_{m,-(d+l)}$ and afterwards the bottom of $H^-_{m,-(d+l+1)}=H^+_{m,kd+1+k}$ where the Reeb chord reaches the bottom of $A'_{m,d}$.
\end{itemize}
The handles which are traversed by these Reeb chords are $H^-_{m,-(d+l)},H^-_{m,-(d+l+1)}, H^-_{m,-(d-i_t+l)}, H^-_{m,-(d-i_t+l+1)}$ since $lj\equiv i \mod r$. 
These two Reeb chords are disjoint: to see this, the only possible problematic scenario is when the Reeb chords are passing over the same boundary component of the same handle.  However, out of the four handles $H^-_{m,-(d+l)},H^-_{m,-(d+l+1)}, H^-_{m,-(d-i_t+l)}, H^-_{m,-(d-i_t+l+1)}$ the only two which could be equal are $H^-_{m,-(d+l)}$ and $H^-_{m,-(d-i_t+l+1)}$ which occurs when $i_t=1$. However, on this handle, one of the Reeb chords traverses the bottom and the other the top part of the boundary.
      
The morphism in the complex \ref{eq:twistedcomplex} is a linear combination of these two Reeb chords. However, since the resulting complex is perfect, both morphisms have to be used (if only one morphism is used, the resulting surgery would produce an arc, rather than a compact Lagrangian). Applying surgery along both Reeb chords, the result follows.
\end{proof}

\begin{eg}
Consider again a $\frac{1}{5}(2,1)\equiv \frac{1}{5}(1,3)$ point $q_m$. In Ishii-Ueda's exceptional collection, there are two sheaves $N^{q_m}_2, N^{q_m}_4$ associated to the two non-special representations: $N^{q_m}_4$ is the effect of pushing forward $\mathcal{O}_{q_m}\otimes \rho_4=\mathcal{S}_{q_m}\{3\}$ to $\mathcal{X}$, whereas $N^{q_m}_2$ is the effect of pushing forward $\mathcal{S}_{q_m}^2\{3\}=[\mathcal{O}_{q_m}\otimes \rho_4 \rightarrow \mathcal{O}_{q_m} \otimes \rho_2[1]]=[\mathcal{S}_{q_m}\{3\}\rightarrow \mathcal{S}_{q_m}\{4\}]$ to $\mathcal{X}$. The push-pull sequence is $\iota_\mathcal{D}^*N^{q_m}_2\rightarrow \mathcal{S}_{q_m}^2\{3\}\rightarrow \mathcal{S}_{q_m}^2\{1\}$, which is described below on the mirror:
\end{eg}

\begin{figure}[H]
      \input{./TikzFiles/skyscrapersurgery.tikz}
          \smash{ \begin{tikzpicture}[overlay, remember picture] 
          \node[inner sep =0] at ($(E)+(7.5cm,3cm)$) {$\simeq$}; \end{tikzpicture} }
         \caption{ \label{fig:skyscrapersurgery}Left: the two arcs mirror to the twice-thickened skyscrapers $\mathcal{S}_{q_m}^2\{3\}=B_{m,2}, \mathcal{S}_{q_m}^2\{1\}=B_{m,2}\{3\}$. Right: a Lagrangian which is isomorphic to $\iota^*N^{q_m}_2$.}
      \end{figure}

\subsection{An abstract Lefschetz fibration and its destabilization}
Following \cite[Definition 1.9]{giroux_existence_2016} and inspired by \cite[Definition 3.2]{hacking_homological_2023} we finally construct the abstract Lefschetz fibration mirror to the pair $(\mathcal{X}, \mathcal{D})$.
\begin{constr}\label{constr:Lefschetz}
The Lefschetz fibration $w_{\mathcal{X}}:W'\rightarrow \mathbb{C}$ associated to an effective log Calabi-Yau orbifold $ (\mathcal{X}, D^{orb})$ is defined to be the exact Lefschetz fibration associated to the data $$\{ \Sigma^0, (\mathfrak{N}_{q_1}, \dots, \mathfrak{N}_{q_N}, \mathcal{V}_1, \mathcal{V}_2, \dots, \mathcal{V}_n)\}$$ consisting of:
\begin{itemize}
\item A bordered Riemann surface $\Sigma^0$, which is the Lekili-Polishchuk mirror to $\mathcal{D}$. This is the reference fiber of $w_{\mathcal{X}}:W'\rightarrow \mathbb{C}$.
\item A collection of Lagrangian $S^1$s $(\mathcal{V}_1, \dots, \mathcal{V}_n)$ which are the Hacking-Keating mirrors to a full exceptional collection of line bundles $\langle \mathcal{L}_1, \dots \mathcal{L}_n\rangle=D^b(Y)$. These Lagrangians are constructed using the identification of $F$ with the special torus $T^0$ from Definition \ref{def:specialtorus}, where $F$ is the reference fiber of the Hacking-Keating mirror $w_Y:W\rightarrow \mathbb{C}$ to $(Y,D)$.
\item For each orbifold point $q_m\in \mathcal{X}$ of type $\frac{1}{r}(1,l)$, a collection of Lagrangian $S^1$s indexed by the non-special representations: $\mathfrak{N}_m:=\{ L_{m,d}\}_{d\in \{0,1,\dots, r\}\setminus I(r,l)}$ as in Definition \ref{def:Ld}. These are ordered in increasing order.
\end{itemize}
\end{constr}

  \subsubsection{Some geometric properties and intersection numbers}
  On the B-side, since the Ishii-Ueda collection is exceptional, there are the following hom-spaces: $$\begin{gathered}
  0=\mathrm{Ext}^\bullet_{\mathcal{X}}(\iota_{\mathcal{D}*}B_{m,d}, \iota_{\mathcal{D}*}B_{m,d'})=\mathrm{Ext}^\bullet_{\mathcal{D}}(\iota_\mathcal{D}^*\iota_{\mathcal{D}*}B_{m,d} ,B_{m,d'}), d'<d\\
  \mathbb{C}=\mathrm{Ext}^\bullet_{\mathcal{X}}(\iota_{\mathcal{D}*}B_{m,d}, \iota_{\mathcal{D}*}B_{m,d})=\mathrm{Ext}^\bullet_{\mathcal{D}}(\iota_\mathcal{D}^*\iota_{\mathcal{D}*}B_{m,d} ,B_{m,d})
  \end{gathered}$$
  In $\Sigma^0$, the mirror to $\iota_\mathcal{D}^*\iota_{\mathcal{D}*} B_{m,d} $ is the Lagrangian $S^1$ $L_{m,d}$, whereas the mirror to $B_{m,d}$ is the arc $A_{m,d}$. This suggests that these arcs should act as destabilizing arcs.

  \begin{prop}\label{prop:intersections}
    The arcs $A_{m,d}$ and the circles $L_{m,d}$ have the following properties:
    \begin{itemize}
      \item The Lagrangians $L_{m,d}$ and $A_{m,d'}$ do not intersect when $d'<d$. When $d=d'$, they intersect at a single point.
      \item The arcs $A_{m,d}$ are disjoint from all special meridians $M_{m,i}$. Moreover, $A_{m,d}$ is completely disjoint from $L_{m',d'}$ if $m\neq m'$.
      \item The result upon cutting $\Sigma^0$ along all the arcs $A_{m,d}$ is the special torus $T^0$, which is the same as the result upon cutting the cocores $S_{m,d}$.
      \item The meridians in the special torus, together with the $L_{m,d}$, form a homologically linearly independent collection in $H_1(\Sigma^0)$. Together with a longitude, they form a basis of $H_1(\Sigma^0)$.
    \end{itemize}
  \end{prop}
  \begin{proof}
  When $m\neq m'$, it follows directly by construction that $A_{m,d}$ is disjoint from $L_{m',d'}$. The rest of the first claim is a reformulation of Ishii and Ueda's proof that the $N^{q_m}_d$ form an exceptional collection. By construction, the arc $A_{m,d'}$ will intersect $L_{m,d}$ precisely when either one of the four conditions are satisfied: $$kd'-a = kd+1 \text{ or } kd-j+1 \text{ or } kd+k+1 \text{ or } kd-j+k+1$$where $0\leq a < j'$. This is equivalent to $$d'-al = d+l \text{ or } d-i+l \text { or }d+l+1 \text{ or }d-i+l+1$$
  In turn, each of these cases corresponds to the existence of a morphism between $N^{q_m}_{d'}$ and $N^{q_m}_{d+l+1}$ in $D^b(\mathcal{X})$. However, by Serre duality, these morphisms are the same as the morphisms between $N^{q_m}_d$ and $N^{q_m}_{d'}$ and we conclude by using \cite[Proposition 3.10]{ishii_special_2013}.

  For the second claim: the arc $A_{m,d}$ goes from the bottom of $H^+_{m,kd}=H^-_{m,-d}$ to the top of $H^+_{m,kd-(j-1)}=H^-_{m,-d+l(j-1)}$. In other words, it jumps over all the handles $H^-_{m,-d}, H^-_{m,-(d-l)},\dots, H^-_{m,-(d-(j-1)l)}$. By \cite[Lemma 3.9]{ishii_special_2013}, $d-al$ is non-special for $0\leq a <j$. Therefore, the arc $A_{m,d}$ jumps over $j$ non-special handles (see Figure \ref{fig:skyscrapersurgery}) and hence is disjoint from the meridians of the special torus.

  The third claim is similarly a reformulation of the fact that $N^{q_m}_d$ for $d$ non-special generate the same subcategory as $\mathcal{O}_{q_m}\otimes \rho_d$. We use the same induction argument as in \cite[Proposition 3.5]{ishii_special_2013}: first of all, for $d>l$ the arc $A_{m,d}$ is just $S_{m,-d}$. Cutting these destroys some of the non-special handles, namely $H^-_{m,-d}$ for $d>l$. We will show, by induction, that after cutting all the arcs $A_{m,d}$ for $d>d'$, the effect of cutting out $A_{m,d}$ is the same as cutting $S_{m,-d}$.
  
  Assume by induction that all handles $H^-_{m,d}$ have been cut along $S_{m,-d}$ for all $d>d'$, with the resulting surface denoted by $\Sigma^{cut}$. The arc $A_{m,d'}$ is an arc starting at the top of $H^-_{m,-(d'-l(j-1))}=H^+_{m,kd'-(j-1)}$, jumping over $j$ handles (in $C_{m+1}$) and ending at the bottom of $H^-_{m,-d'}=H^+_{m,kd'}$. By \cite[Corollary 3.8]{ishii_special_2013} each of the $d'-al, 0<a<j$ are non-special and bigger than $d'$. By the induction hypothesis, all the cocores $S_{m,-(d'-al)}$ of the handles which $A_{m,d}$ jumps over have been cut and hence $A_{m,d'}$ is isotopic to the cocore $S_{m,-d'}$ inside the cut-up surface $\Sigma^{cut}$. This completes the induction step.

  For the last claim: one can take any linear dependence relation in $H_1(\Sigma^0)$ and intersect it with $A_{m,d_{min}}$, where $d_{min}$ is the smallest non-special weight associated to the orbifold point $q_m$. By the previous two claims, this shows that $L_{m,d_{min}}$ does not appear in the relation. Inductively like this, one shows that all the coefficients in front of the $L_{m,d}$ are zero. Finally, the meridians of the special torus and the longitude are of course linearly independent. Since the rank of $H_1(\Sigma^0)$ is the rank of $H_1(T^0)$ plus the sum of the number of non-special representations of each orbifold point, the claim about being a basis of $H_1(\Sigma^0)$ follows.
  \end{proof}

  \begin{thm}\label{thm:destabilization}
  The Lefschetz fibration $w_\mathcal{X}:W'\rightarrow \mathbb{C}$ can be obtained from $w_Y:W\rightarrow \mathbb{C}$ by a sequence of Lefschetz stabilizations. Therefore, $W'$ is Liouville deformation equivalent to $W$, which in turn is symplectomorphic to the total space of the almost toric fibration associated to $(Y,D)$.
  \end{thm}
  \begin{proof}
  Using Proposition \ref{prop:intersections}, one can use the arcs $A_d$ as a set of destabilizing arcs. The claim about the almost toric fibration is \cite[Theorem 1.2]{hacking_homological_2023}.
  \end{proof}
  \subsection{The proof of homological mirror symmetry at the large complex structure limit}
  \begin{prop}\label{prop:restrictandvarphi}
    Suppose $(Y,D)$ is at the large complex structure limit. Then, up to an autoequivalence of $\mathcal{F}(\Sigma^0)$, the composition $$D^b(\mathcal{X})\rightarrow \mathrm{Perf}(\mathcal{D})\simeq \mathcal{F} (\Sigma^0)$$ will send the exceptional objects in a full exceptional collection for $D^b(\mathcal{X})$ to objects equivalent to Lagrangian $S^1$s, equipped with the trivial local system. The exceptional objects $\Phi\mathcal{L}$ (where $\mathcal{L}$ is a line bundle on $Y$) are sent to the Hacking-Keating mirror vanishing cycles in the special torus $T^0\subset \Sigma^0$, whereas the Ishii-Ueda sheaves $N_d^{q_m}$ are sent to the Lagrangian $S^1$'s $L_{m,d}$ from \ref{def:Ld}.
  \end{prop}

  \begin{proof}
  We decompose $$D^b(\mathcal{X})=\langle \mathbf{e}_{q_1},\dots, \mathbf{e}_{q_N}, \Phi D^b(Y)\rangle$$

  Take a collection of exceptional line bundles $\mathcal{L}$ for $D^b(Y)$ as in Hacking-Keating. Then, $$\iota^*_\mathcal{D}\Phi\mathcal{L}=\Phi^\partial \iota^*_D \mathcal{L}$$However, $\iota^*_D\mathcal{L}=\prod_s T^{\mathcal{L}\cdot D_s}_{\mathcal{O}_{p_s}} \mathcal{O}$ is given by a product of spherical twists along the skyscrapers at the points $-1\in \mathbb{C}^\times\subset D_i$ applied to the structure sheaf. Therefore, \begin{equation}\label{eq:varphiofHK}
  \iota_\mathcal{D}^*\Phi \mathcal{L}=\prod_s T^{\mathcal{L}\cdot D_s}_{\Phi^\partial \mathcal{O}_{p_s}} \Phi^\partial \mathcal{O}
  \end{equation}
  Proposition \ref{prop:fullyfaithfulspecialtorus} identifies $\Phi^\partial \mathcal{O}_{p_s}$ with the meridians in the special torus $T^0$ under $\varphi$. Therefore, after applying $\varphi$ to \ref{eq:varphiofHK}, the object in $\mathcal{F}(\Sigma^0)$ will be a product of spherical twists along the special meridians $M_{m,i}$ (potentially with a nontrivial systems) applied to the longitude. By \cite[Section 17j]{seidel_fukaya_2008}, spherical twists are equivalent to Dehn twists in the Fukaya category and hence the result is a Lagrangian supported on the Hacking-Keating Lagrangian vanishing cycle mirror to $\mathcal{L}$ (potentially with a different local system).

  On the other hand, we saw already in Proposition \ref{prop:mirrortoLd} that the Lagrangian circle which is mirror to $\iota_\mathcal{D}^*\iota_{\mathcal{D}*} B_{m,d}$ is supported on $L_{m,d}$.

  However, since the $L_{m,d}$, together with the meridians and longitude of $T^0$, form a basis of $H_1(\Sigma^0)$ by Proposition \ref{prop:intersections}, by applying an autoequivalence given by a global local system on $\Sigma^0$ (these are in bijection with $\mathrm{Hom}(\pi_1(\Sigma^0), \mathbb{C}^\times)\simeq \mathrm{Hom}(H_1(\Sigma^0), \mathbb{C}^\times)$), all the potentially nontrivial local systems on $L_{m,d}$ and the meridians and longitude can be cancelled out. 
  \end{proof}
  \begin{rem}
  Note that when the complex structure of $Y$ is not the distinguished one, this result no longer holds true, since line bundles on $Y$ will not be determined simply by their intersections with the components $D_s$ of $D$.
  \end{rem}

  \begin{proof}[Proof of HMS part of Theorem \ref{thm:mainthmabstract}]
    By taking dg enhancements (which exist: see \cite{lunts_uniqueness_2010}), one can take a model of $D^b(\mathcal{X})$ as the directed subcategory of $D^b(\mathcal{D})$ with objects the ordered set of sheaves $$\{\iota^*_\mathcal{D} (N_{m,d})\}_{m,d} \text{ and } \iota^*_\mathcal{D} \Phi \mathcal{L}_1, \dots , \iota^*_\mathcal{D} \Phi \mathcal{L}_n$$ with the ordering on $N_{m,d}$ being the lexicographic one induced by the ordering on $m$ and the one on $d$. Call the resulting category $\mathcal{B}_{dg}$: by Proposition \ref{prop:restobdry}, $H^0(\mathrm{perf} \mathcal{B}_{dg})\simeq D^b(\mathcal{X})$.

    On the A-side, we take a strictly unital model of the pretriangulated Fukaya-Seidel category as in Definition \ref{def:FukayaCategory} by taking the ordered set of Lagrangian $S^1$s $$\{ L_{m,d}\}_{m,d} \text{ and } \mathcal{V}_1, \dots, \mathcal{V}_n$$ in $\Sigma^0$ and imposing directedness. We call the resulting category $\mathcal{A}$ and by definition \ref{def:FukayaCategory}, $\mathcal{F}(W', w_\mathcal{X})\simeq H^0(\mathrm{perf} \mathcal{A})$.

    Now, the $A_\infty$ equivalence $\varphi:D^b(\mathcal{D})\rightarrow \mathcal{W}(\Sigma^0)$ identifies the subcategory $\mathcal{B}_{dg}$ with its image under $\varphi$, whose triangulated closure is the same as that of $\mathcal{A}$, by Propositions \ref{prop:mirrortoLd}, \ref{prop:restrictandvarphi}. With this, the theorem follows.
  \end{proof}
\section{An explicit case study: the hypersurfaces $X_{k+1}\subset \mathbb{P}(1,1,1,k)$}\label{sec:explicitcase}
In the second half of this paper, we compare the abstract construction \ref{constr:Lefschetz} which we used to prove homological mirror symmetry at the large complex structure limit \ref{thm:mainthmabstract} to an explicit Landau-Ginzburg model associated to the family of hypersurfaces $X_{k+1}\subset \mathbb{P}(1,1,1,k)$. This example is sufficiently explicit that we can work out a mirror map. This allows us to prove HMS for a generic complex structure - this time, the strategy is to find full, strong exceptional collections on both sides and identify their endomorphism algebras.

Recall that an orbifold del Pezzo surface in this family is obtained by taking $\mathbb{P}^2$, blowing up $k+1$ points on a line (resulting in a log CY surface $Y$) and then contracting the strict transform of the line and considering the result as an orbifold.

\subsection{The left adjoint of $\Phi$}
Let $E$ be the exceptional $-k$ curve in the minimal resolution $Y$ of $X$, and let $q$ be the unique orbifold point of type $\frac{1}{k}(1,1)$ in $\mathcal{X}$. The category $D^b(\mathcal{X})$ is obtained by gluing $D^b(Y)$ and $\langle \mathcal{O}_q \otimes \rho_2, \dots, \mathcal{O}_q \otimes \rho_{k-1}\rangle$. The gluing is most efficiently described by the left adjoint $\Psi$ to the fully faithful $\Phi:D^b(Y)\rightarrow D^b(\mathcal{X})$, which is computed in \cite[Theorem 3.1]{gugiatti_full_2023}, and in particular, the case of $\frac{1}{k}(1,1)$ points in \cite[Example 3.5]{gugiatti_full_2023}. Writing $e_d:=\mathcal{O}_q\otimes \rho_d$ for the orbifold skyscraper sheaf, then $$\Psi e_d = \begin{cases}
      0, \quad d=0,\dots, k-3\\
      \mathcal{O}_{E}(-k)[1],\quad d=k-2\\
      \mathcal{O}_{E}(-k+1), \quad d=k-1
    \end{cases}$$We denote by $\iota_{E}:E \rightarrow Y$ the inclusion.

    \begin{prop}\label{prop:Psiisiso}
      The functors $\Psi$ and $\iota_{E*}$ induce isomorphisms: $$\mathrm{Ext}_{\mathcal{X}}^0(e_{k-2}[-2], e_{k-1}[-1])\xrightarrow{\Psi} \mathrm{Ext}^0_Y(\mathcal{O}_{E}(-k), \mathcal{O}_{E}(-k+1))\xleftarrow{\iota_{E*}} \mathrm{Ext}^0_{E}(\mathcal{O}(-k), \mathcal{O}(-k+1))$$ 
    \end{prop}
    \begin{proof}
      The map $\iota_{E*}$ is the same as the composition of the counit with an adjunction isomorphism: \begin{equation}\label{eq:iotacounit}
    \mathrm{Ext}^0_{E}(\mathcal{O}(-k), \mathcal{O}(-k+1))\rightarrow \mathrm{Ext}^0_{E}(\iota_{E}^* \iota_{E*}\mathcal{O}(-k), \mathcal{O}(-k+1))\xrightarrow{\simeq} \mathrm{Ext}^0_Y(\iota_{E*} \mathcal{O}(-k),  \iota_{E*}\mathcal{O}(-k+1)) \end{equation}
    By \cite[Corollary 11.4]{huybrechts_fourier-mukai_2006}, since $\mathcal{O}(-k)=\iota_{E}^* \mathcal{O}(E)$, the distinguished triangle for the counit $$\mathcal{O}[1]\rightarrow \iota_{E}^* \iota_{E*} \iota_{E}^*\mathcal{O}(E)\rightarrow \iota_{E}^* \mathcal{O}(E)$$ splits, so that $\iota_{E}^*\iota_{E*} \mathcal{O}(-k)=\mathcal{O}(-k)\oplus \mathcal{O}[1]$ and the counit morphism is just projection to the first factor, inducing an isomorphism on $\mathrm{Ext}^0$ in \ref{eq:iotacounit}.

    We proceed to show that the functor $\Psi$ also induces an isomorphism. Since it is a purely local statement around the orbifold point $q$, we will from now on assume $\mathcal{X}=[\mathbb{C}^2/\frac{1}{k}(1,1)], Y=\mathrm{Tot}\,\mathcal{O}(-k)$.

    As a first step, note that there are two decompositions $$D^b[\mathbb{C}^2/\frac{1}{k}(1,1)]=\langle \underbrace{\ker \Theta}_{e_2, \dots, e_{k-1}},\Phi D^b(\mathrm{Tot}\,\mathcal{O}(-k))\rangle = \langle \Phi D^b(\mathrm{Tot}\,\mathcal{O}(-k)), \underbrace{\ker\Psi}_{e_0,\dots, e_{k-3}}\rangle$$where $\Theta$ is the right adjoint to $\Phi$ as in \cite{gugiatti_full_2023}. Next, we will right mutate $e_{k-2},e_{k-1}$ through $\Phi D^b(Y)$: this mutation will be denoted by $\mathbb{R}$.

    The counit of the adjunction produces canonical exact triangles $\mathrm{id}\rightarrow \Phi \Psi \rightarrow \mathbb{R} [1]$ which induces a long exact sequence $$\dots \rightarrow \mathrm{Ext}_{\mathcal{X}}^\bullet(e_{k-2}[-2],e_{k-1}[-1])\rightarrow \mathrm{Ext}_{\mathcal{X}}^\bullet(e_{k-2}[-2],\Phi \Psi e_{k-1}[-1])\rightarrow \mathrm{Ext}_{\mathcal{X}}^\bullet(e_{k-2}[-2],\mathbb{R}e_{k-1})\rightarrow \dots$$

    The first term in this long exact sequence is $\mathbb{C}^2[0]$, whereas the second, by adjunction, is the morphism space $\mathrm{Ext}_{\mathcal{X}}^0(\Psi e_{k-2}[-2], \Psi e_{k-1}[-1])$. By naturality of adjunction, the map $$\mathrm{Ext}_{\mathcal{X}}^\bullet(e_{k-2}[-2],e_{k-1}[-1])\rightarrow \mathrm{Ext}_{\mathcal{X}}^\bullet(e_{k-2}[-2],\Phi \Psi e_{k-1}[-1])\xrightarrow{\simeq} \mathrm{Ext}_{\mathcal{X}}^\bullet(\Psi e_{k-2}[-2], \Psi e_{k-1}[-1])$$is given by the functor $\Psi$ on morphisms.
      
    To complete the proof of the lemma, we need to show that $\mathrm{Ext}^{\leq 0}_{\mathcal{X}}( e_{k-2}[-2], \mathbb{R}e_{k-1})=0$. 
      \begin{gather*}
        \mathrm{Ext}_{\mathcal{X}}^\bullet(e_{k-2}[-2],\mathbb{R}e_{k-1})\simeq 
        \mathrm{Ext}_{\mathcal{X}}^{-\bullet}(\mathbb{R}e_{k-1},e_0)^\lor \simeq 
        \mathrm{Ext}_{\mathcal{X}}^{-\bullet}(\mathbb{R}e_{k-1},[\mathcal{O}\otimes \rho_2 \rightarrow \underbrace{\mathcal{O}^{\oplus 2}\otimes \rho_1 \rightarrow \mathcal{O}\otimes \rho_0}_{\in \mathrm{im} \Phi} ])^\lor\simeq \\
        \simeq \mathrm{Ext}_{\mathcal{X}}^{-\bullet}(\mathbb{R}e_{k-1},\mathcal{O}\otimes \rho_2[2])^\lor \simeq 
        \mathrm{Ext}_{\mathcal{X}}^{\bullet}(\mathcal{O},\mathbb{R}e_{k-1})\simeq 
        \mathrm{Ext}_{\mathcal{X}}^{\bullet}(\mathcal{O},[e_{k-1}\rightarrow \Phi\Psi e_{k-1}])\simeq \\
        \simeq \mathrm{Ext}_{\mathcal{X}}^{\bullet-1}(\mathcal{O},\Phi \Psi e_{k-1})\underbrace{\simeq}_{\Phi^{-1}} \mathrm{Ext}^{\bullet-1}_{Y}(\mathcal{O}, \Psi e_{k-1})\simeq \mathrm{Ext}^{\bullet-1}_{Y}(\mathcal{O}, \mathcal{O}_{E}(-k+1))
        \simeq \\
        \simeq \mathrm{Ext}^{\bullet-1}_{E}(\mathcal{O}, \mathcal{O}(-k+1))\begin{cases}\mathbb{C}^{k-2}, \bullet =2\\
        0,\,\mathrm{otherwise}\end{cases}
      \end{gather*}

    We are using Serre duality, the Koszul sequence, the fact that $\mathbb{R}e_{k-1}\in \Phi D^b(Y)^\perp$ and the fact that $e_{k-1}\in ^\perp\! \Phi D^b(Y)$. Moreover, we use the fact that the sheaves $\mathcal{O}\otimes \rho_0, \mathcal{O}\otimes \rho_1$ lie in the essential image of $\Phi$. This follows from \cite[Proposition 1.1]{ishii_special_2013} since $\rho_0, \rho_1$ are special representations.
      
    \end{proof}

    \begin{cor}\label{cor:naturality}
    If $\mathcal{L}, \mathcal{L}'\in D^b(Y)$ are exceptional, then the compositions on $D^b(\mathcal{X})$ involving the sheaves $e_{k-2}[-2], e_{k-1}[-1]$ and $\Phi \mathcal{L}, \Phi \mathcal{L}'$ can be computed via the compositions involving $\mathcal{O}, \mathcal{O}(1), \iota_{E}^*\mathcal{L}, \iota_{E}\mathcal{L}'$ in $D^b(E)$.
    \end{cor}

    \begin{proof}
      The claims follow by naturality of adjunction and a big commutative diagram:
    \[\begin{tikzcd}
      { \mathrm{Ext}^\bullet_\mathcal{X} ( e_{k-1}[-1], \Phi \mathcal{L})\otimes \mathrm{Ext}^0_\mathcal{X} ( e_{k-2}[-2], e_{k-1}[-1])} &{}& \mathrm{Ext}^\bullet_\mathcal{X} ( e_{k-2}[-2], \Phi\mathcal{L}) \\
      { \mathrm{Ext}^\bullet_Y (\mathcal{O}_{E}(-k+1)[-1], \mathcal{L})\otimes \mathrm{Ext}^0_Y(\mathcal{O}_{E}(-k), \mathcal{O}_{E}(-k+1))} &{}& {\mathrm{Ext}^\bullet_Y( \mathcal{O}_{E}(-k)[-1], \mathcal{L})}\\
      { \mathrm{Ext}^\bullet_{E} ( \mathcal{O}(-1), \iota_{E}^!\mathcal{L})\otimes \mathrm{Ext}^0_{E}(\mathcal{O}(-2),  \mathcal{O}(-1))} &{}& {\mathrm{Ext}^\bullet_{E}( \mathcal{O}(-2), \iota_E^!\mathcal{L})} \\
      { \mathrm{Ext}^\bullet_{E} ( \mathcal{O}(1), \iota^*_{E}\mathcal{L})\otimes \mathrm{Ext}^0_{E}(\mathcal{O},  \mathcal{O}(1))} &{}& {\mathrm{Ext}^\bullet_{E}( \mathcal{O}, \iota^*_{E}\mathcal{L})} 
      \arrow["{\mu^2}", from=1-1, to=1-3]
      \arrow["{\mu^2}"', from=2-1, to=2-3]
      \arrow["{\mu^2}"', from=3-1, to=3-3]
      \arrow["{\mu^2}"', from=4-1, to=4-3]
      \arrow["{(\Psi \dashv \Phi)\otimes\Psi}"', from=1-1, to=2-1]
      \arrow["{\Psi \dashv \Phi}", from=1-3, to=2-3]
      \arrow["{(\iota_{E*}\dashv \iota_{E}^!)\otimes \iota_{E*}}", from=3-1, to=2-1]
      \arrow["{\iota_{E*}\dashv \iota_{E}^!}"', from=3-3, to=2-3]
      \arrow["{}", from=4-1, to=3-1]
      \arrow["{}"', from=4-3, to=3-3]
    \end{tikzcd}\]
    We use the fact that $\iota_{E}^!=\iota_{E}^*\otimes \mathcal{O}(-2)$. All the vertical morphisms are isomorphisms, by Proposition \ref{prop:Psiisiso}.  There is an analogous diagram, involving compositions between one of $e_{k-2}[-2]$ and $e_{k-1}[-1]$ and two objects of $D^b(Y)$.
    \end{proof}
    \subsection{A description of $D^b(X_{k+1})$ via a quiver with relations}

    \subsubsection{The McKay algebra of a $\frac{1}{k}(1,1)$ point}\label{sec:McKayalgebra}
    The material in this subsection is to a large extent borrowed from \cite[Section 2.3]{gugiatti_full_2023}. The algebra of the local orbifold piece is governed by the McKay graph, see Figure \ref{fig:quiver_1/7(1,1)}. Since the sheaves are supported only at the orbifold point, the Ext-algebra can be computed locally on the stack $[\mathbb{C}^2/\frac{1}{k}(1,1)]$. Coherent sheaves on this stack are equivalent to $G$-equivariant sheaves on $\mathbb{C}^2$, where $G\subset GL(2,\mathbb{C})$ is the cyclic group generated by the matrices $\begin{pmatrix}
    \chi^i & 0 \\ 0 & \chi^i
    \end{pmatrix}$ with $\chi^k=1$ a primitive root of unity. For each $i=0,1,\dots, k-1$ define $\rho_i$ to be the irreducible weight $i$ representation of $G$. 
    The structure sheaf of the origin admits the equivariant Koszul resolution
    $$[\mathcal{O}\otimes \rho_{2}\xrightarrow{(y,-x)^T} \mathcal{O}\otimes \rho_{1}^{\oplus 2}\xrightarrow{(x,y)} \mathcal{O}]\rightarrow \mathcal{O}_0$$ where $\rho_2=\rho_{det}$ is the determinant representation and $\mathcal{O}^{\oplus 2}\otimes \rho_1$ is the natural representation induced by $G\subset GL(2,\mathbb{C})$. The sheaf $\mathcal{O}\otimes \rho_i$ is the sheaf associated to the $G$-equivariant $\mathbb{C}[u,v]$-module which is $\mathbb{C}[u,v]$ equipped with the action $f(u,v)\mapsto \chi^i f(\chi u, \chi v)$. The Koszul resolution yields a resolution of the sheaves $e_i:=\mathcal{O}_0\otimes \rho_i$:$$[\mathcal{O}\otimes \rho_{i+2}\xrightarrow{(y,-x)^T} \mathcal{O}^{\oplus 2}\otimes \rho_{i+1}\xrightarrow{(x,y)} \mathcal{O}\otimes \rho_{i}]\rightarrow \mathcal{O}_0 \otimes \rho_i$$

    As such, we have generators $p_{1,i},p_{2,i}\in \mathrm{Ext}^1(e_i, e_{i+1})=\mathbb{C}^2$ represented by the following maps of complexes: 
    \[\begin{tikzcd}
       & \mathcal{O}\otimes \rho_{i+2} \arrow[r] \arrow[d, "(-1)^j\iota_j"] & \mathcal{O}^{\oplus 2}\otimes \rho_{i+1} \arrow[r] \arrow[d, "\pi_j"] & \mathcal{O}\otimes \rho_{i} \\
      \mathcal{O}\otimes \rho_{i+3} \arrow[r] & \mathcal{O}^{\oplus 2}\otimes \rho_{i+2} \arrow[r]                 & \mathcal{O}\otimes \rho_{i     +1}                                         &                            
    \end{tikzcd}\]

    Here, $\pi_j, j=1,2$ denotes the projection to the first respectively second factor, and $\iota_j$ denotes the inclusions $f\mapsto (0,f)$ respectively $f\mapsto (f,0)$. There are the following relations: $$p_{1,i+1}\circ p_{1,i} = -p_{2,i+1}\circ p_{2,i}, \qquad p_{1,i+1}\circ p_{2,i}=p_{2,i+1}\circ p_{1,i}=0$$
    with $p_{1,i+1}\circ p_{1,i}$ a generator of $\mathrm{Ext}^2(e_i, e_{i+2})=\mathbb{C}$. After replacing $e_i$ by $e_i[i-k]$, all the Ext groups become concentrated in degree zero, completing the description of the dg algebra associated to the orbifold skyscrapers.
      \begin{figure}[H]
    \centering
    \begin{tikzpicture}[scale=1,every path/.style={->, >=latex}]

    \node (v7) at (2,0) {};
    \node (v1) at (1.247,1.5637) {};
    \node (v2) at (-0.445,1.9499) {};
    \node (v3) at (-1.8019,0.8678) {};
    \node (v4) at (-1.8019,-0.8678) {};
    \node (v5) at (-0.445,-1.9499) {};
    \node (v6) at (1.247,-1.5637) {};

    \draw[->,gray]  (v1) edge  (v2);
    \draw[->,thick]   (v2) edge  (v3);
    \draw[->,thick]   (v3) edge node[midway, right=-2pt] {\tiny $p_2$} (v4);
    \draw[->,thick]   (v4) edge node[midway, above] {\tiny $p_2$} (v5);
    \draw[->,thick]   (v5) edge (v6);
    \draw[->,gray]   (v6) edge  (v7);
    \draw[->,gray] (v7) edge  (v1);

    \node (v9) at (1.3717,-1.72) {};
    \node (v10) at (2.2,0) {};
    \node (v11) at (1.3717,1.72) {};
    \node (v12) at (-0.4895,2.1448) {};
    \node (v8) at (-0.4895,-2.1448) {};
    \node (v14) at (-1.9821,-0.9545) {};
    \node (v13) at (-1.9821,0.9545) {};

    \draw [->, shorten >=2pt,shorten <=2pt,thick] (v8) edge  (v9);
    \draw [->, shorten >=2pt,shorten <=2pt,gray] (v10) edge (v11);
    \draw [->, shorten >=2pt,shorten <=2pt,gray] (v9) edge (v10);
    \draw [->, shorten >=2pt,shorten <=2pt,gray] (v11) edge  (v12);
    \draw [->, shorten >=2pt,shorten <=2pt,thick] (v12) edge (v13);
    \draw [->, shorten >=2.5pt,shorten <=2.5pt,thick] (v13) edge node[midway, left] {\tiny $p_1$} (v14);
    \draw[->, shorten >=2pt,shorten <=2pt ,thick ] (v14) -- node[midway, below] {\tiny $p_1$} (v8);


    \node [circle,fill,black, inner sep = 2pt] at (-1.8583,-0.8837) {};
    \node [circle,fill,black, inner sep = 2pt] at (-0.4702,-2.0256) {};
    \node [circle,fill,black, inner sep = 2pt] at (1.3093,-1.6418) {};
    \node [circle,fill,gray, inner sep = 2pt] at (2.0755,-0.0082) {};
    \node [circle,fill,gray, inner sep = 2pt] at (1.2941,1.6224) {};
    \node [circle,fill,black, inner sep = 2pt, yshift=-1.5pt] at (-0.4673,2.0473) {};
    \node [circle,fill,black, inner sep = 2pt, xshift=1pt] at (-1.892,0.9112) {};

    \draw [dashed,gray] (v5) edge (v7);
    \draw [dashed,gray] (v6) edge (v1);
    \draw [dashed,gray] (v7) edge (v2);
    \draw [dashed,gray] (v1) edge (v3);
    \draw [dashed, thick] (v2) edge (v4);
    \draw [dashed, thick] (v3) edge  node[midway,right] {\tiny $p_1^2=-p_2^2$}(v5);
    \draw [dashed, thick] (v4) edge   (v6);

    \node[gray]  at (2.334,-0.02) {$e_0$};
    \node[gray]  at (1.532,1.8141) {$e_1$};
    \node  at (-0.4917,2.3155) {$e_2$};
    \node  at (-2.1254,0.9728) {$e_3$};
    \node  at (-2.1609,-1.0674) {$e_4$};
    \node  at (-0.5319,-2.3008) {$e_5$};
    \node  at (1.4517,-1.8525) {$e_6$};
    \end{tikzpicture}
    \caption{\label{fig:quiver_1/7(1,1)} The case $\tfrac 17(1,1)$. Thick arrows are degree 1 and dashed ones are degree 2.}
    \end{figure}
    \subsubsection{The derived category of the resolution}
    An application of Orlov's blowup formula, as in \cite[Theorem 2.5]{auroux_mirror_2006}, yields a full exceptional collection: $$D^b(Y)=\langle \mathcal{O}, \mathcal{T}(-H), \mathcal{O}(H), \begin{array}{c}
      \mathcal{O}_{B_1} \\
      \mathcal{O}_{B_2} \\
      \vdots \\
      \mathcal{O}_{B_{k+1}}
      \end{array}\rangle $$
    
    The assumption on the points being distinct ensures that $\mathcal{O}_{B_i}$ is completely orthogonal to $\mathcal{O}_{B_j}$ for $i\neq j$.
    The most important feature of the composition law is that the compositions $\mathcal{T}(-H)\rightarrow \mathcal{O}(H)\rightarrow \mathcal{O}_{B_i}$ induce a surjective map from the three-dimensional vector space $\mathbb{C}^3=\mathrm{Hom}(\mathcal{T}(-H), \mathcal{O}(H))$ to the two-dimensional space $\mathrm{Hom}(\mathcal{T(-H)}, \mathcal{O}_{B_i})$. The kernel is a line in $\mathbb{C}^3$ which determines the point in $\mathbb{P}^2$ associated to the exceptional curve $B_i$.
    \subsubsection{The gluing morphism spaces}
    \begin{prop}\label{prop:gluingmorphisms}
      The morphism spaces between $e_{k-2}[-2], e_{k-1}[-1]$ and $\Phi D^b(Y)$ are all concentrated in degree $0$ and are given by \begin{gather*}\mathrm{Ext}_Y^\bullet(e_{k-2}[-2],\Phi \mathcal{O})=\mathbb{C}, \quad
      \mathrm{Ext}_Y^\bullet(e_{k-2}[-2],\Phi \mathcal{T}(-H))\simeq \mathbb{C}^3\quad\\ \mathrm{Ext}_Y^\bullet(e_{k-2}[-2], \Phi \mathcal{O}(H))\simeq \mathbb{C}^2,\quad \mathrm{Ext}_Y^\bullet(e_{k-2}[-2], \Phi \mathcal{O}_{B_i})\simeq \mathbb{C}\\
      \mathrm{Ext}_Y^\bullet(e_{k-1}[-1],\Phi \mathcal{T}(-H))\simeq \mathbb{C},\quad
        \mathrm{Ext}_Y^\bullet(e_{k-1}[-1], \Phi \mathcal{O}(H))\simeq \mathbb{C}\\
        \mathrm{Ext}_Y^\bullet(e_{k-1}[-1], \Phi \mathcal{O}_{B_i})\simeq \mathbb{C},\quad \mathrm{Ext}_Y^\bullet(e_{k-1}[-1],\Phi \mathcal{O})=0
    \end{gather*}
    \end{prop}

    \begin{proof}
    Corollary \ref{cor:naturality} shows that $\mathrm{Ext}^\bullet(e_{k-2}[-2],\mathcal{L})\simeq \mathrm{Ext}^\bullet(\mathcal{O}, \iota_E^*\mathcal{L}), \mathrm{Ext}^\bullet(e_{k-1}[-1],\mathcal{L})\simeq \mathrm{Ext}^\bullet(\mathcal{O}(1), \iota_E^*\mathcal{L})$. The rest of the computation is straightforward, using the Euler sequence $\mathcal{O}(-H)\rightarrow \mathcal{O}^{\oplus 3}\rightarrow \mathcal{T}(-H)$ and the fact that $ \iota_E^*\mathcal{O}_{B_i}=\mathcal{O}_{pt_i}$ where $pt_i$ is one of the $k+1$ points on the $-k$ curve $E\subset Y$.
    \end{proof}
      We now choose a set of generators in order to describe the directed quiver with relations for the exceptional collection.  We pick these so that $\{z=0\}$ defines the line on $\mathbb{P}^2$ which is blown up $k+1$ times, and moreover so that $\Psi p_{1,k-2}, \Psi_{2,k-2}$ are identified with $\iota^*_E x, \iota^*_E y$ under the isomorphism $\mathrm{Ext}^0_E(\mathcal{O},\mathcal{O}(1))\simeq \mathrm{Ext}^0_Y(\mathcal{O}_E(-k), \mathcal{O}_E(-k+1))$.
    \[\begin{tikzcd}
    {e_2[2-k]} \arrow[r, "{p_{1,2}}", shift left] \arrow[r, "{p_{2,2}}"', shift right] & \dots \arrow[r, "{p_{1,k-3}}", shift left] \arrow[r, "{p_{2,k-3}}"', shift right] & {e_{k-2}[-2]} \arrow[r, "{p_{1,k-2}}", shift left] \arrow[r, "{p_{2,k-2}}"', shift right] \arrow[rr, "\epsilon'"', bend right=60] & {e_{k-1}[-1]} \arrow[rr, "\tilde{\epsilon}", bend left=60] & \Phi \mathcal{O} \arrow[r, "{x_0,y_0,z_0}", shift left] \arrow[r, shift right] \arrow[r] & \Phi \mathcal{T}(-H) \arrow[r, "{x_1,y_1,z_1}", shift left] \arrow[r, shift right] \arrow[r] & \Phi \mathcal{O}(H) \arrow[r, "{r_{H,i}}"] & \Phi \mathcal{O}_{B_i}
    \end{tikzcd}\]
    \begin{prop}\label{prop:quiverwithrelsB}
      The endomorphism algebra of the exceptional collection above can be described as the algebra of the quiver with relations \begin{gather*}
        p_{1,i+1}\circ p_{1,i}=-p_{2,i+1}\circ p_{2,i},\qquad p_{1,i+1}\circ p_{2,i}=p_{2,i+1}\circ p_{1,i}=0, \qquad \epsilon '\circ p_{1,k-3}=\epsilon '\circ p_{2,k-3}=0\\
        y_1 \circ x_0 = -x_1 \circ y_0=-z\\
        z_1 \circ x_0 = -x_1 \circ z_0=y\\
        z_1 \circ y_0 = - y_1 \circ z_0=-x\\
        x_1 \circ x_0 = y_1 \circ y_0 = z_1 \circ z_0 =0\\
        r_{H,i}\circ(u_ix_1+v_iy_1)=0
      \end{gather*}
      where the coefficients $u_i,v_i$ define the point $[u_i:v_i:0]\in \mathbb{P}^2$ which is blown up corresponding to the exceptional divisor $B_i$. There are also relations for the gluing maps $\epsilon',\tilde{\epsilon}$: \begin{gather*}
        x_1\circ \tilde{\epsilon}=y_1\circ \tilde{\epsilon}=0, \qquad
        x_0\circ\epsilon'=\tilde{\epsilon}\circ p_{2,k-2}, \qquad
        y_0\circ\epsilon'=\tilde{\epsilon}\circ p_{1,k-2}
      \end{gather*}
    \end{prop}
    \begin{proof}
    The relations for $p_{1,i}, p_{2,i}$ were described in \ref{sec:McKayalgebra}. For the rest of the relations, under the Euler sequence isomorphism $[\mathcal{O}(-H)\rightarrow \mathcal{O}^3]\simeq \mathcal{T}(-H)$, the maps $x_0,y_0,z_0$ are represented by the morphisms of chain complexes which are the three inclusions of $\mathcal{O}$ into $\mathcal{O}^3$. On the other hand, $x_1,y_1,z_1$ are represented by the morphisms of chain complexes which are given in vector form $(0,z,-y)^T, (-z,0,x)^T, (y,-x,0)^T$ as maps from $\mathcal{O}^3\rightarrow \mathcal{O}$. From this, the first set of relations readily follow.

    For the gluing compositions, by Corollary \ref{cor:naturality} we can compute by restricting to $E$. We have the splitting $\iota^*\mathcal{T}(-H)=\mathcal{O}\oplus \mathcal{O}(1)$, under which we can identify $\tilde{\epsilon}$ as $(0,1)\in \mathrm{Hom}_E(\mathcal{O}(1), \mathcal{O}\oplus \mathcal{O}(1))$, $\iota^*z_0$ with $(1,0)\in \mathrm{Hom}_E(\mathcal{O},\mathcal{O}\oplus \mathcal{O}(1))$ and $\iota^*z_1$ with $(0,1)^T\in \mathrm{Hom}(\mathcal{O}\oplus \mathcal{O}(1), \mathcal{O}(1))$. Furthermore, $\iota^*x_0$ can be identified with $(0,\iota^*y)$ and $\iota^*y_0$ with $(0,\iota^*x)$, and similarly $\iota^*x_1$ with $(\iota^*y,0)^T$ and $\iota^*y_1$ with $(\iota^*x,0)^T$. The compositions $x_1\circ \tilde{\epsilon}, y_1 \circ \tilde{\epsilon}$ are then immediately seen to be $0$ in the composition $$ \mathrm{Hom}_E(\mathcal{O} \oplus \mathcal{O}(1), \mathcal{O}(1))\otimes \mathrm{Hom}_E (\mathcal{O}(1), \mathcal{O}\oplus \mathcal{O}(1)) \rightarrow \mathrm{Hom}_E(\mathcal{O}(1), \mathcal{O}(1))$$

    Finally, the composition $x_0 \circ \epsilon', y_0 \circ \epsilon'$ are given by $$ \underbrace{\mathrm{Hom}_E(\mathcal{O}, \mathcal{O}\oplus \mathcal{O}(1))}_{\iota^*x_1 = (0, \iota^*y), \iota^*y_1=(0,\iota^*x)}\otimes \underbrace{\mathrm{Hom}_E(\mathcal{O}, \mathcal{O})}_{\epsilon'=1}\rightarrow \mathrm{Hom}_E(\mathcal{O}, \mathcal{O}\oplus \mathcal{O}(1))$$
    which are the same as the compositions $$ \underbrace{\mathrm{Hom}_E(\mathcal{O}(1), \mathcal{O}\oplus\mathcal{O}(1))}_{\tilde{\epsilon}=(0,1)}\otimes \underbrace{\mathrm{Hom}_E(\mathcal{O},\mathcal{O}(1))}_{\Psi p_{2,k-2}=y, \Psi p_{1,k-2}=x} \rightarrow \mathrm{Hom}_E(\mathcal{O}, \mathcal{O}\oplus \mathcal{O}(1))$$

    The fact that these morphisms generate the algebra is a simple consequence of Proposition \ref{prop:gluingmorphisms}.
    \end{proof}
 
\subsection{The LG model mirror to $X_{k+1}$ and its critical points}\label{sec:LGmodel}
We now switch gears and study an explicit Landau-Ginzburg model mirror to the orbifold del Pezzo surfaces in the family $X_{k+1}$ and relate it to the abstract Lefschetz fibration from \ref{constr:Lefschetz}. We will construct a collection of vanishing paths whose vanishing cycles will be used to prove homological mirror symmetry for the family $X_{k+1}$ in Section \ref{sec:Lagvc}.

The orbifold $\mathcal{X}$ admits a toric degeneration to $\mathbb{P}(1,k,k+1)$. This is most easily seen by viewing $\mathbb{P}(1,k,k+1)$ as $\mathbb{V}(x^{k+1}+zw)\subset \mathbb{P}(1,1,1,k)$ (see e.g. \cite[Section 4]{cavey_pezzo_2020}). Standard mirror constructions of toric degenerations yield the following process: we pass to the polar dual toric variety $V$ of $\mathbb{P}(1,k,k+1)$ and produce a pencil of algebraic curves on it. The pencil arises from two sections of a line bundle: the first section corresponds to the origin lattice point in the Fano polygon, and the other section corresponds to the boundary lattice points. In our case $V$ is polarized by the line bundle associated to $D_V=V_0+V_1+kV_2$ and given by the convex hull of the points $(-1,-1),(-1,k), (1,0)$:
    \begin{figure}[H]
    \noindent
    \minipage{0.48 \textwidth}
    \centering
    \begin{tikzpicture}[scale =0.5]
    \draw[step=1.0,black,thin,xshift=0cm,yshift=0cm, opacity=0.5] (-2,-2) grid (2,8);
    \node[inner sep = 0] (v2) at (1,0) {};
    \node[inner sep = 0] (v1) at (-1,-1) {};
    \node[inner sep = 0] (v3) at (-1,7) {};
    \draw[thick,red]  (v1) edge node[midway, below] {$V_1$}  (v2);
    \draw[thick,red]  (v2) edge node[midway, right] {$V_2$} (v3);
    \draw[thick,red]  (v1) edge node[midway, left=8pt, below] {$V_0$} (v3);

    \node [circle, fill, gray, inner sep = 1] at (0,0) {};
    \node [circle, fill, black, inner sep = 1] at (0,1) {};
    \node [circle, fill, black, inner sep = 1] at (0,2) {};
    \node [circle, fill, black, inner sep = 1] at (0,3) {};
    \node [circle, fill, black, inner sep = 1] at (-1,-1) {};
    \node [circle, fill, black, inner sep = 1] at (1,0) {};
    \node [circle, fill, black, inner sep = 1] at (-1,5) {};
    \node [circle, fill, black, inner sep = 1] at (-1,0) {};
    \node [circle, fill, black, inner sep = 1] at (-1,1) {};
    \node [circle, fill, black, inner sep = 1] at (-1,2) {};
    \node [circle, fill, black, inner sep = 1] at (-1,3) {};
    \node [circle, fill, black, inner sep = 1] at (-1,4) {};
    \node [circle, fill, black, inner sep = 1] at (-1,5) {};
    \node [circle, fill, black, inner sep = 1] at (-1,6) {};
    \node [circle, fill, black, inner sep = 1] at (-1,7) {};
    \end{tikzpicture}
    \endminipage
    \hfill
    \minipage{0.48\textwidth}%
    \hspace*{-3cm}
    \centering
    \begin{tikzpicture}[scale=0.5,every path/.style={->, >=latex}]
    \draw[step=1.0,black,thin,opacity=0.5,xshift=0cm,yshift=0cm] (-8,-1) grid (2,5);

    \node [circle,fill,black,inner sep =1] (v1) at (0,2) {};
    \node [circle,fill,black,inner sep =1] (v2) at (-7,0) {};
    \draw [thick,->,teal] (v1) edge (v2);

    \node [circle,fill,black,inner sep =1] (v7) at (-1,2) {};
    \node [circle,fill,black,inner sep =1] (v6) at (-1,3) {};
    \node [circle,fill,black,inner sep =1] (v3) at (-1,4) {};
    \node [circle,fill,black,inner sep =1] (v4) at (1,2) {};
    \node [circle,fill,black,inner sep =1] (v5) at (0,3) {};
    \node [circle,fill,black,inner sep =1] (v8) at (-4,1) {};
    \node [circle,fill,black,inner sep =1] (v9) at (-3,0) {};
    \draw [thick,->,teal] (v1) edge (v3);
    \draw [thick,->,teal] (v1) edge (v4);
    \draw [dashed,->,teal] (v1) edge (v5);
    \draw [dashed,->,teal] (v1) edge (v6);
    \draw [dashed,->,teal] (v1) edge (v7);
    \draw [dashed,->,teal] (v1) edge (v8);
    \draw [dashed,->,teal] (v1) edge (v9);
    \end{tikzpicture}
    \endminipage
    \caption{\label{fig:polygon} The polygon (left) and its fan (right) and also the resolution rays (dashed) in the case $k=7$.} 
    \end{figure}
    The resulting potential, for $k$ odd, has the following form: $$\mathbf{f}|_{(\mathbb{C}^\times)^2}=\frac{(y+\mathbf{q}_1)\dots (y+\mathbf{q}_{k+1})}{xy}+x+\tau_1 y+ \dots \tau_{\frac{k-1}{2}}y^{\frac{k-1}{2}}$$

    The first term comes from the left column of the polygon, the second from the lattice point $(1,0)$ and the remaining terms arise from the interior lattice points. This Laurent polynomial appears in the case $k=3, \mathbf{q}_i=1$ in \cite[Section 7.2]{oneto_quantum_2018}. 
    \begin{rem}
      When $k$ is even, there is an extra boundary lattice point $(0,\frac{k}{2})$ and hence the potential will have a different form. We choose $k$ to be odd purely for the sake of convenience. 
    \end{rem}

    The precise values of $\mathbf{q}_i, \tau_j$, if generic, will not affect the symplectic geometry of this Landau-Ginzburg model. We will take $\mathbf{q}_i\approx 1$ but all distinct, $\tau_1=1, \tau_j=0,j>1$. This suffices for our purposes. Moreover, we introduce an auxiliary $s$-parameter in front of $x$ and define: $$\mathbf{f}_s=\frac{\prod_{i=1}^{k+1} (y+\mathbf{q}_i)}{xy}+sx+y$$
    \begin{rem}
      The Newton polytope of $\mathbf{f}$ is exactly the Fano polygon described above, in other words it corresponds to a section of $\mathcal{O}(D_V)$. For $\mathbf{f}$ to be non-degenerate (in the sense of \cite{kouchnirenko_polyedres_1976}, see also \cite[Section 6]{seidel_suspending_2009}), the values of $\mathbf{q}_i$ have to be generic i.e. $P(y)=\prod (y+\mathbf{q}_i)$ must not admit double roots, since the face polynomial on the left face $V_0$ of $V$ is given by $P(y)/xy$. This non-degeneracy precludes the existence of critical points on the base locus of the pencil, which is related to the Palais-Smale condition.
    \end{rem}

    \begin{rem}
      The parameters $\mathbf{q}_i$ are Novikov parameters associated to $H^2(\mathcal{X})=\mathbb{Z}^{k+1}$, whereas the $\tau_j$ correspond to twisted sectors in the Chen-Ruan cohomology of $\mathcal{X}$. A piece of computer code made by the author which can compute the small quantum cohomology of $\mathcal{X}$ can be found \href{https://bogdan-simeonov.github.io/visualizations/}{here}.
    \end{rem}

    \subsubsection{Homogeneous coordinates and the genus of the general fiber}
    The fan of $V$ consists of the rays $(1,0), (-1,2), (-k,-2)$. It has two $\frac{1}{2}(1,1)$ singularities where $V_0$ meets $V_1, V_2$ and also a $\frac{1}{2k+2}(1,k+2)$ singularity where $V_1, V_2$ meet. Moreover, $V_0^2=\frac{k+1}{2}, V_1^2=V_2^2=\frac{1}{2k+2}, D_V^2=2k+2$. The homogeneous coordinate ring has variables $v_0, v_1, v_2$ corresponding to the toric divisors and the lattice points correspond to the following monomials: \begin{gather*}
      (0,0)\leftrightarrow v_0v_1 v_2^k,
      (1,0)\leftrightarrow v_0^2,
      (-1,i)\leftrightarrow v_1^{2i+2}v_2^{2k-2i}=\{ v_2^{2k+2}, v_2^{2k}v_1^2, \dots, v_1^{2k+2}\}
    \end{gather*}

    A general section $\Sigma$ of $\mathcal{O}(D_V)$ has Euler characteristic given by $$\int_\Sigma c_1(\mathcal{T}_\Sigma)=\int_\Sigma c_1(\mathcal{T}_V)-c_1(\mathcal{O}(D_V))=\int_\Sigma c_1(\mathcal{O}((1-k)V_2))=(1-k)\,\Sigma\cdot V_2=1-k$$
    Hence its genus will be $\frac{k+1}{2}$, which is also the number of interior lattice points in the Newton polytope. Moreover, it satisfies $\Sigma\cdot V_0=k+1, \Sigma\cdot V_1= \Sigma \cdot V_2=1$.

    Note that after resolving the $A_1$ singularity where $V_0$ meets $V_1$, introducing an exceptional divisor $V_3$, we can write $\mathbf{f}_s$ in a chart $U_{0,3}\simeq \mathbb{C}^2$ as: \begin{equation}\label{eq:chart}
      \mathbf{f}_s=P(v_3)+sv_0^2v_3+v_0v_3^2, o=v_0v_3\end{equation}
      where $o$ is the monomial section of $\mathcal{O}(D_V)$ corresponding to the origin. These two functions define a pencil on $\mathbb{C}^2$. We will be mostly working in this chart.

    \subsubsection{The total space of the Lefschetz fibration}
    So far, we have defined a pencil on a toric surface $V$. We explain how to turn this into a Lefschetz fibration.

    The pencil on $V$ has $k+1$ base points on $V_0$ and one base point each on $V_1, V_2$. We blow these up, then remove the strict transform of the toric boundary, producing a manifold $M$. This admits a Lefschetz fibration whose general fiber is a compact genus $\frac{k+1}{2}$ Riemann surface. Since $c_1$ of this fiber is nonzero, this will preclude the possibility of a $\mathbb{Z}$-grading on the Fukaya-Seidel category to be defined later.

    Instead, if we just worked in the chart $U_{0,3}$ and blew up the $k+1$ points on the axis corresponding to $V_0$, then removed the strict transform of the coordinate lines, we will get a manifold that we denote by $M^0$. This manifold admits a Lefschetz fibration induced by $\mathbf{f}$ whose general fiber is a twice-punctured genus $\frac{k+1}{2}$ Riemann surface\footnote{The punctures correspond to the pencil's base points on $V_1, V_2$}. It is a partial compactification of $(\mathbb{C}^\times)^2$. By \cite[Section 7]{evans_symplectic_2014}, the manifold $M^0$ is in fact diffeomorphic to the following manifold defined by an algebraic equation:
    $$\{zx=P(y)\}\subset \mathbb{C}^2_{x,z}\times \mathbb{C}^\times_y$$
    The $(\mathbb{C}^\times)^2$ inside $M^0$ is given by the locus where $x\neq 0$.
    In other words, $M^0$ is diffeomorphic to the $A_k$ Milnor fiber with a cylinder removed. As such, it has $H_2(M_0)=\mathbb{Z}^{k+1}$, with $k$ of the factors coming from a chain of $-2$ spheres, and another one coming from the generator of $H_2((\mathbb{C}^\times)^2)$. We will equip $M^0$ with the non-exact symplectic form as in Equation \ref{def:symplecticform}.  

    \subsubsection{Critical points and critical values}
    Now, we proceed to determine the critical points. These are all contained inside $(\mathbb{C}^\times)_{x,y}^2$ (provided that $P$ does not admit double roots). We use the variable $t$ to denote the value of $\mathbf{f}_s$.
    \begin{prop}\label{prop:criticalptsdeformation}
    As $s\rightarrow 0$ (setting $P(y)=(1+y)^{k+1}+\epsilon$ for a small $\epsilon$), the critical values become distributed as follows:
    \begin{enumerate}[label=\Roman*.]
      \item $k-2$ of the critical values are distributed near the roots of $t^{k-2}-(\frac{k-2}{k})^{k-2}\frac{1}{s}=0$.
      \item $3$ of them will be arbitrarily close to $0$
      \item Another $k+1$ critical points are equidistributed near $-1$.
    \end{enumerate}
    \end{prop}

    \begin{proof}
    We provide the proof in the appendix \ref{propappendix:criticalptsdeformation}. The case when $\epsilon=0$ is illustrated in Figure \ref{fig:vanishingpaths}, where the $k+1$ critical values of type \uppercase\expandafter{\romannumeral 3} have all collided.
    \end{proof}

    \begin{rem}
    If we set $\mathbf{q}_i=1$ there will be a single, $k+1$-fold degenerate critical point on the base locus on $V_0$ as in Figure \ref{fig:sketch} below. This can be thought of as the 'big' eigenvalue of the quantum multiplication by $c_1$ on $QH^\bullet(\mathcal{X})$ when $\mathcal{X}$ is monotone. During the process of resolving the base locus, the point where the two branches meet at $V_0$ is replaced by a tower of $k$ holomorphic spheres all living in this special fiber, so in the end, the fibration has the same number of vanishing cycles.
    \end{rem}

      \begin{figure}[!h]
        \begin{tikzpicture}[scale=0.5]
        \node[inner sep = 0] (v2) at (1,0) {};
        \node[inner sep = 0] (v1) at (-1,-1) {};
        \node[inner sep = 0] (v3) at (-1,7) {};
        \draw[very thick,red!70!black]  (v1) edge node[midway, below] {$V_1$}  (v2);
        \draw[very thick,red!70!black]   (v2) edge node[midway, right] {$V_2$} (v3);
        \draw[very thick,red!70!black]   (v1) edge node[midway, left=8pt, below] {$V_0$} (v3);

      \begin{scope}[yshift=60, rotate=90]
      \draw[thick] (-0.5,0) to [out=30, in=140] (0.2,0);
      \draw[thick] (-0.5+0.1,0.05) to [out=-30, in=-140] (0.2-0.1,0+0.05);
      \end{scope}

      \begin{scope}[yshift=10, rotate=90]
      \draw[thick] (-0.5,0) to [out=30, in=140] (0.2,0);
      \draw[thick] (-0.5+0.1,0.05) to [out=-30, in=-140] (0.2-0.1,0+0.05);
      \end{scope}

      \begin{scope}[yshift=35,rotate=90]
      \draw[thick] (-0.5,0) to [out=30, in=140] (0.2,0);
      \draw[thick] (-0.5+0.1,0.05) to [out=-30, in=-140] (0.2-0.1,0+0.05);
      \end{scope}

      \draw[gray] (-0.65,4.5) to [out=0, in=-30] (-0.65,3.72);
      \draw[gray,dotted] (-0.65,4.5) to [out=230, in=130] (-0.65,3.72);

      \draw[very thick] (-1,3) .. controls (-1,3.9) and (-0.9096,4.4743) .. (-0.65,4.5) .. controls (-0.4,4.5) and (-0.25,4.15) .. (0,3.5) .. controls (0.2,2.7) and (0.35,2.15) .. (0.4,1.6) .. controls (0.4906,0.5474) and (0.4506,-0.2583) .. (0,-0.5) .. controls (-0.4326,-0.6367) and(-0.7136,0.1066) .. (-0.9,1) .. controls (-1,1.6) and (-1,2)  .. (-1,3);
      \draw[very thick] (-0.75,3) ellipse (0.25cm and 0.8cm);

      \node[black,circle,fill, inner sep=2] at (-1,3) {};
      \end{tikzpicture}
      \caption{\label{fig:sketch} Sketch of the critical point at the base locus when all $\mathbf{q}_i=1$. The base locus can be resolved after $k$ consecutive blowups at the intersection of the branch of $\Sigma_{-1}$ with intersection multiplicity $k$ with $V_0$. }
      \end{figure}
    
    We will from now on focus on this deformed LG model with $s$ arbitrarily small and real. An animation of how the the critical values move as $s$ goes from $1$ to $s\approx 0$ can be found \href{https://bogdan-simeonov.github.io/visualizations/}{here}.

    \subsection{A collection of vanishing cycles}
    In this section, we determine the vanishing cycles topologically by using a bifibration method, looking at the branch points of some fixed reference fiber by projecting to the $y$-axis and observing how they collide when we reach a critical value. Initially, we will choose a reference fiber on the real line and follow vanishing paths similar to those of Figure \ref{fig:vanishingpaths}. The diagram in Figure \ref{fig:vanishingpaths} depicts the case when all $\mathbf{q}_i=1$, whereas (as in Proposition \ref{prop:criticalptsdeformation}) we slightly deform the potential, in which case the orange critical value at $-1$ splits into $k+1$ critical values nearby. We define the vanishing paths for these $k+1$ points in a simple way: they start at the reference fiber, go slightly above the blue critical values and then go to the $k+1$ points in a cyclic manner.

    This gives us a collection of Lagrangian vanishing cycles which we label in clockwise order:\[\color{purple}\tilde{L}_{k-1}, \tilde{L}_{k-2}, \dots, \tilde{L}_2, \color{blue} P_{-1},P_0,P_1, \color{orange} B_1,\dots, B_{k+1}\] \noindent We use the suggestive notation of $B_i$ for the Lagrangian vanishing cycles in accordance with their putative mirror partner sheaves $\mathcal{O}_{B_i}$.
      \begin{rem}
        The vanishing cycles $\tilde{L}_i$ are not the mirrors to the sheaves $e_i$ on the B-side. Rather, to match them, one needs to apply a \textit{left dual} mutation which we will describe later on in this section.
      \end{rem}

      Throughout, we will pick a reference fiber whose value $\mathbf{f}=t_*$ is such that $|t_*|$ is smaller than the absolute values of the $k-2$ critical values of type \uppercase\expandafter{\romannumeral 1} as in Proposition \ref{prop:criticalptsdeformation}.

      \subsubsection{The general fiber as a branched double cover}

      Let $\Sigma^{pre}_t$ be a fiber $\mathbf{f}|_{(\mathbb{C}^\times)^2}^{-1}(t)$. Under the $y$ projection, $\Sigma^{pre}_t$ is (outside a finite set) a double cover of $\mathbb{C}^\times$ with $k+1$ branch points. However, there are another $k+1$ points, the roots of $P(y)$ (these are \textit{not} branch points), where the preimage under $y$ consists of only one point, rather than the expected two. Once these are filled in (by partially compactifying $(\mathbb{C}^\times)^2$ to $M^0$), we get an honest branched double cover of $\mathbb{C}^\times$, which is what we now consider: the general fiber in the Lefschetz fibration of $M^0$ will be denoted $\Sigma^0_t$. The $k+1$ punctures that are filled correspond to $k+1$ sections of the Lefschetz fibration arising from the $k+1$ points on the base locus on $V_0$ which are blown up. 

      Consider now the branch point equation for $\Sigma_t^{pre}$ on the open torus: it is given by $$\mathbf{f}_s=0,\qquad \partial_x\mathbf{f}_s=0 \underbrace{\implies}_{\text{ \ref{propappendix:criticalptsdeformation}}} P(y)-\frac{1}{4s}y(t-y)^2=0$$
     For $|t|\gg0$ and with $|t|$ less than the absolute values of the critical values of type \uppercase\expandafter{\romannumeral 3} (and with $s$ sufficiently small), this has a root near $0$ as well as two roots near $t$ which we call the \textit{twin branch points}\footnote{One can reason as follows: consider the roots of the equation $4sP-y(t-y)^2=0$, once compactified to $\mathbb{P}^1$. When $s=0$, there is a root at $0$, a double root at $y=t$ and a $k-2$-fold root at $\infty$. Then, perturbing $s$ by a small amount deforms the positions of the roots continuously.}. Moreover, there are another $k-2$ equidistributed branch points $y$ satisfying $|y^{k-2}- \frac{1}{4s}|^2\approx 0$. We collect this into a proposition:

      \begin{prop}\label{prop:branchpoints}
      For $s$ sufficiently small and $\frac{1}{s}\gg |t| \gg 0$ the branch points of the Riemann surface $\Sigma^0_t$ under the $y$-projection to $\mathbb{C}^\times$ are distributed as follows: 
      \begin{enumerate}[label=\Roman*.]
        \item $k-2$ are distributed close to the roots of $y^{k-2}=\frac{1}{4s}$.
        \item One is close to $0$
        \item Two are close to $t$ 
      \end{enumerate} 
      \end{prop}

      Below is a schematic picture illustrating the branch points of a reference fiber $\Sigma^0_t$ for such a $t$. 

      \begin{figure}[H]
        \hspace{1cm}
\begin{tikzpicture}[scale=0.7]
        \definecolor{DarkGreen}{RGB}{1,50,32}

        \node [circle, fill, DarkGreen, inner sep = 1.5pt, label =above:$2$] (v6) at (2.6564,1.3942) {};
        \node [circle, fill, DarkGreen, inner sep = 1.5pt, label =above:$3$] (v8) at (1.7042,2.469) {};
        \node [circle, fill, DarkGreen, inner sep = 1.5pt, label =above:$4$] (v10) at (0.3616,2.9781) {};
        \node [circle, fill, DarkGreen, inner sep = 1.5pt, label =above:$5$] (v12) at (-1.0638,2.805) {};
        \node [circle, fill, DarkGreen, inner sep = 1.5pt, label =above:$6$] (v14) at (-2.2455,1.9894) {};
        \node [circle, fill, DarkGreen, inner sep = 1.5pt, label =above:$7$] (v16) at (-2.9128,0.7179) {};
        \node [circle, fill, DarkGreen, inner sep = 1.5pt, label =above:$8$] (v18) at (-2.9128,-0.7179) {};
        \node [circle, fill, DarkGreen, inner sep = 1.5pt, label =above:$9$] (v20) at (-2.2455,-1.9894) {};
        \node [circle, fill, DarkGreen, inner sep = 1.5pt, label =above:$10$] (v22) at (-1.0638,-2.805) {};
        \node [circle, fill, DarkGreen, inner sep = 1.5pt, label =above:$11$] (v24) at (0.3616,-2.9781) {};
        \node [circle, fill, DarkGreen, inner sep = 1.5pt, label =above:$12$] (v26) at (1.7042,-2.469) {};
        \node [circle, fill, DarkGreen, inner sep = 1.5pt, label =above:$13$] (v28) at (2.6564,-1.3942) {};
        \node [circle, fill, DarkGreen, inner sep = 1.5pt, label =above:$14$] (v2) at (3,0) {};

        \node [shape=rectangle, fill, DarkGreen, inner sep = 1.5pt, label =above:$0$] (v3) at (1.5,0)  {};
        \node [shape=rectangle, fill, DarkGreen, inner sep = 1.5pt, label =above:$1$] (v1) at (2.33,0) {};
        
        \node [circle, draw, red, inner sep = 1.5pt] (v4) at (6,0) {};
        \node [circle, draw, red, inner sep = 1.5pt] (v5) at (0.5,0) {};
        \node [circle, fill, DarkGreen, inner sep = 1.5pt] (v7) at (0.8,0) {}; 
        \draw[thick, opacity=0.5,->, blue!50!red] (2.2328,-1.1564) -- (2.5132,-1.3232);      
        \draw[thick, opacity=0.5,->, blue!50!red] (2.5,0) -- (2.8,0);	
        \draw[thick, orange, ->, opacity=0.5] (2.2969,-0.0913) .. controls (2.0551,-0.4367) and (1.754,-0.5798) .. (1.4559,-0.6096);
        \draw[thick, orange, opacity=0.5] (1.6027,-0.0912) .. controls (1.7132,-0.1075) and (1.8956,-0.1627) .. (2.0194,-0.2867);
        \draw[thick, orange, ->, opacity=0.5] (2.1303,-0.373) .. controls (2.228,-0.4847) and (2.2978,-0.6591) .. (2.1665,-0.9592);
        \node [shape=rectangle, fill, DarkGreen, inner sep = 1.5pt, opacity=0.5] (v3) at (1.3505,-0.6034) {};
        \node [shape=rectangle, fill, DarkGreen, inner sep = 1.5pt, opacity=0.5] (v3) at (2.0889,-1.0645) {};
      \end{tikzpicture}  
        \caption{\label{fig:branch points labelled} Branch points for $\Sigma^0_t$ when $k=15$ and $t$ is on the real line, with added numbered labels which will come in play later, with the 'twin' branch points labeled 0 and 1. The red branch points correspond to the two punctures of $\Sigma_t^0$ at $y=0$ and $y=\infty$. Note that the one at $y=0$ corresponds to the base point on $V_1$ and the one at $\infty$ to the base point on $V_2$}
      \end{figure}
      
      \subsubsection{The vanishing cycles for the critical values of type \uppercase\expandafter{\romannumeral 1}}
      Starting from the reference $t_*$, we consider the effect on the branch points by following the vanishing paths in Figure \ref{fig:vanishingpaths} that go to the $k-2$ critical values of type \uppercase\expandafter{\romannumeral 1}. These are isotopic to a combination of a rotation of $t$, followed by an outwards radial scaling. In Figure \ref{fig:branch points labelled} above, the movement associated to the radial outwards scaling of $t$ is illustrated by the purple arrows pointing towards the branch points labelled $14$ (and $13$). The movement associated to rotating $t$ is depicted by the orange transposition arrows. Both of these movements are animated \href{https://bogdan-simeonov.github.io/visualizations/}{here} and proved in detail in the Appendix \ref{propappendix:radialoutwards}. 

      We will now describe what topological classes the vanishing cycles represent in $H_1(\Sigma^0_*)$. 
      \begin{defn}
      Let $l$ denote the class in $H_1(\Sigma^0_*)$ which is represented by a curve which projects under the branched double cover to an arc connecting the twin branch points. Let $l_i,2\leq i \leq k-1$ denote the class in $H_1(\Sigma^0_*)$ which is represented by a curve which projects under the branched double cover to an arc which starts at the branch point $1$, goes down and around and ends up at the branch point labelled $i$.
      We choose the orientations so that $\langle l_i, l_j\rangle =-1, i<j$ and $\langle l,l_i\rangle=-1$.
      \end{defn} 
      The projections under the branched double cover of representatives of these classes can be visualized as follows:
      \begin{figure}[H]
        \hspace{1cm}
\begin{tikzpicture}[scale=0.7]
        \definecolor{DarkGreen}{RGB}{1,50,32}
          \node [circle, fill, DarkGreen, inner sep = 1.5pt] (v6) at (2.6564,1.3942) {};
          \node [circle, fill, DarkGreen, inner sep = 1.5pt] (v8) at (1.7042,2.469) {};
          \node [circle, fill, DarkGreen, inner sep = 1.5pt] (v10) at (0.3616,2.9781) {};
          \node [circle, fill, DarkGreen, inner sep = 1.5pt] (v12) at (-1.0638,2.805) {};
          \node [circle, fill, DarkGreen, inner sep = 1.5pt] (v14) at (-2.2455,1.9894) {};
          \node [circle, fill, DarkGreen, inner sep = 1.5pt] (v16) at (-2.9128,0.7179) {};
          \node [circle, fill, DarkGreen, inner sep = 1.5pt] (v18) at (-2.9128,-0.7179) {};
          \node [circle, fill, DarkGreen, inner sep = 1.5pt] (v20) at (-2.2455,-1.9894) {};
          \node [circle, fill, DarkGreen, inner sep = 1.5pt] (v22) at (-1.0638,-2.805) {};
          \node [circle, fill, DarkGreen, inner sep = 1.5pt] (v24) at (0.3616,-2.9781) {};
          \node [circle, fill, DarkGreen, inner sep = 1.5pt] (v26) at (1.7042,-2.469) {};
          \node [circle, fill, DarkGreen, inner sep = 1.5pt] (v28) at (2.6564,-1.3942) {};
          \node [circle, fill, DarkGreen, inner sep = 1.5pt] (v2) at (3,0) {};

          \node [rectangle, fill, DarkGreen, inner sep = 1.5pt] (v3) at (1.5,0)  {};
          \node [rectangle, fill, DarkGreen, inner sep = 1.5pt] (v1) at (2.33,0) {};

          \node [circle, draw, red, inner sep = 1.5pt] (v4) at (6,0) {};
          \node [circle, draw, red, inner sep = 1.5pt] (v5) at (0.5,0) {};
          \node [circle, fill, DarkGreen, inner sep = 1.5pt] (v7) at (0.8,0) {};
        
          \draw  (v3) edge node[midway,above=3pt] {$l$} (v1);

          \draw[purple] (v1) edge node[midway, above=1pt]{$l_{14}$} (v2);  
                  
          \draw[violet] (v1) edge node[midway, above=3pt]{$l_{9}$} (v20);
          \end{tikzpicture}   
      \caption{\label{fig:branch points topology} A collection of arcs in $\mathbb{C}^\times_y$ which lift to the classes $l, l_9, l_{14}$ under the branched double cover (depicted is $k=15$).}
      \end{figure}

      \begin{prop}\label{prop:Ljhomology}
      The vanishing cycles associated to the $k-2$ critical values of type \uppercase\expandafter{\romannumeral 1}  with reference fiber $\Sigma^0_*$ and vanishing paths as in Figure \ref{fig:vanishingpaths} are given in homology by $$[\tilde{L}_{j}]=(-1)^j((k-1-j)l+l_j)$$ 
      \end{prop}

      \begin{proof}
      This is clear by our description of how the branch points move under rotation and under radial scaling. We first rotate, transposing the two twin branch points $k-1-j$ times. Then, the twin point closer to the branch point of type \uppercase\expandafter{\romannumeral 1} collides with it. Unwinding this process, the projection of the vanishing cycle to the $y$-axis will look like in Figure \ref{fig:branch points}, which is given in homology by the class $(k-1-j)l+l_j$.
      \end{proof}

      \begin{figure}[ht]
        \hspace{1cm}
\begin{tikzpicture}[scale=1]
        \definecolor{DarkGreen}{RGB}{1,50,32}

        \node [rectangle, fill, DarkGreen, inner sep = 1pt] (v3) at (1,0) {};
        \node [circle, fill, DarkGreen, inner sep = 1pt] (v1) at (1.8,0) {};
        \node [circle, fill, DarkGreen, inner sep = 1pt] (v7) at (0.3,0) {};
        \node [circle, fill, DarkGreen, inner sep = 1pt] (v4) at (0.5562,1.7119) {};
        \node [circle, fill, DarkGreen, inner sep = 1pt] (v5) at (-1.4562,1.058) {};
        \node [circle, fill, DarkGreen, inner sep = 1pt] at (-1.4562,-1.058) {};
        \node [circle, fill, DarkGreen, inner sep = 1pt] at (0.5562,-1.7119) {};
        \node [rectangle, fill, DarkGreen, inner sep = 1pt] (v2) at (1.4,0) {};

        \draw [red, thick] (v2) edge (v1);
        \draw [DarkGreen,opacity=0.5, thick](0:1.006) .. controls (9:1.1034) and (16:1.4362) .. (2:1.5051) .. controls (-15:1.718) and (-50:1.7871) .. (-72:1.8142);
        \draw [blue,opacity=0.5, thick](-1:1.4008) .. controls (-13:1.3637) and (-26:0.7207) .. (11:0.7312) .. controls (31:0.8988) and (21:1.4278) .. (4:1.4476) .. controls (-18:1.6542) and (-77:1.4515) .. (-144:1.8049);
        \draw[purple, thick] (1.0012,0.0246) .. controls (0.9985,0.2364) and (1.2497,0.2985) .. (1.3202,0.2969) .. controls (1.398,0.2984) and (1.5558,0.2378) .. (1.5789,-0.0031) .. controls (1.5849,-0.2465) and (1.2852,-0.3712) .. (1.0917,-0.3584) .. controls (0.896,-0.3485) and (0.5154,-0.1872) .. (0.5842,0.1516) .. controls (0.6594,0.5366) and (1.0347,0.5756) .. (1.1752,0.5586) .. controls (1.2898,0.5583) and (1.6531,0.416) .. (1.6738,0.0023) .. controls (1.6936,-0.554) and (1.4751,-1.0898) .. (0.2885,-0.9452) .. controls (-0.6114,-0.7418) and (-1.2925,-0.015) .. (-1.4525,1.05);

        \node [rectangle, fill, DarkGreen, inner sep = 1pt] (v3) at (1,0) {};
        \node [  circle, fill, DarkGreen, inner sep = 1pt] (v1) at (1.8,0) {};
        \node [circle, fill, DarkGreen, inner sep = 1pt] (v7) at (0.3,0) {};
        \node [circle, fill, DarkGreen, inner sep = 1pt] (v4) at (0.5562,1.7119) {};
        \node [circle, fill, DarkGreen, inner sep = 1pt] (v5) at (-1.4562,1.058) {};
        \node [circle, fill, DarkGreen, inner sep = 1pt] at (-1.4562,-1.058) {};
        \node [circle, fill, DarkGreen, inner sep = 1pt] at (0.5562,-1.7119) {};
        \node [rectangle, fill, DarkGreen, inner sep = 1pt] (v2) at (1.4,0) {};

        \node [circle, draw, red, inner sep = 1.5pt]  at (3,0) {};
          \node [circle, draw, red, inner sep = 1.5pt]  at (0,0) {};

        \end{tikzpicture}
      \caption{\label{fig:branch points} The projections under the branched double covering of the topological vanishing cycles associated to $L_{k-1}, L_{k-2}, L_{k-3}, L_{k-4}$ for $k=7$. 
      }
      \end{figure}

      For a collection of vanishing paths defining vanishing thimbles $D_i$, the Seifert pairing (see \cite[Section 3.1.3]{gugiatti_mirrors_2025}) is defined by $$\langle D_i, D_j\rangle_{Seifert}=\begin{cases}
      \langle [\partial D_i], [\partial D_j]\rangle, i<j\\
      1, \, i=j\\
      0, \, i>j
      \end{cases}$$

      Currently, the Gram matrix of the Seifert pairing for our collection of Lagrangian vanishing cycles is of the form \[\begin{bmatrix} 1 & -2 & 3 & -4 & \cdots \\ 0 & 1 & -2 & 3 & \cdots \\0 & 0 & 1 & -2 & \cdots \\0 & 0 & 0 & 1 & \ddots\end{bmatrix}\] which differs from the Gram matrix of the sheaves $\langle e_2, \dots, e_{k-1}\rangle$ under the Euler pairing. To match them, we perform a series of mutations known as passing to the \textit{left dual} (as in \cite[Definition 2.14]{auroux_mirror_2004}).

      Explicitly, this process can be described as follows: we left mutate $\tilde{L}_{k-2}$ through $\tilde{L}_{k-1}$, denoting the result by $\mathbb{L}^{(1)}\tilde{L}_{k-2}$. Then, we left mutate $\tilde{L}_{k-3}$ through $ \mathbb{L}^{(1)}\tilde{L}_{k-2}, \tilde{L}_{k-1}$, denoting this by $\mathbb{L}^{(2)}\tilde{L}_{k-3}$. We do this for all Lagrangians $\tilde{L}_i$. At the end of this process, we have a new sequence $$\mathbb{L}^{(k-3)}\tilde{L}_{2},\mathbb{L}^{(k-4)}\tilde{L}_{3}, \dots, \mathbb{L}^{(1)}\tilde{L}_{k-2},\tilde{L}_{k-1}$$

      \begin{prop}\label{prop:leftdual}The left dual to the collection $\langle \tilde{L}_{k-1},\dots, \tilde{L}_2\rangle$, which we denote by $\langle L_{2},\dots, L_{k-1}\rangle$, is described in homology (up to a choice of orientations) by \begin{gather*}
        [L_{k-1}]=l_{k-1},\qquad
        [L_{k-2}]=l_{k-2}-2l_{k-1}+l,\qquad
        [L_{i}]=l_{i}-2l_{i+1}+l_{i+2}, i<k-2\\
      \end{gather*}
        As such, its Gram matrix matches the Gram matrix of the exceptional subcollection $\langle e_2, \dots, e_{k-1}\rangle$ on the B-side.
      \end{prop}
      \begin{proof}
      We prove this claim by induction. First, $[\mathbb{L}_{\tilde{L}_{k-1}}\tilde{L}_{k-2}]=[ \tilde{L}_{k-2}]-\langle  \tilde{L}_{k-1}, \tilde{L}_{k-2}\rangle [ \tilde{L}_{k-1}]=[ \tilde{L}_{k-2}]+2[ \tilde{L}_{k-1}]=(-1)(l_{k-2}-2l_{k-1}+l)$ as claimed.

      The rest of the proof is a pleasant exercise in combinatorics. First of all, left mutation on the level of homology is given by $\mathbb{L}_a b= b-\langle a,b\rangle a$, so the consecutive left mutation through a sequence of elements $a_1,\dots, a_m$ is given by $$\mathbb{L}_{a_1,\dots, a_m} b = b+\sum_{i_1<i_2<\dots <i_s} (-1)^s\langle a_{i_1}, a_{i_2}\rangle \dots \langle a_{i_s},b\rangle a_{i_1}$$

      Suppose now as an inductive hypothesis that $[ \mathbb{L}^{(i)}\tilde{L}_{k-1-i}]=(-1)^{k-1-i}(l_{{k-1-i}}-2l_{k-i}+l_{k-i+1})$ for all $i\leq m$. Thus, for $ i>1$ we have $\langle \mathbb{L}^{(i)}\tilde{L}_{k-1-i},  \tilde{L}_{k-1-(m+1)}\rangle=0$. Moreover, $\langle  \tilde{L}_{k-1},  \tilde{L}_{k-1-(m+1)}\rangle = (-1)^{m+1}(m+2), \langle  \mathbb{L}^{(1)}\tilde{L}_{k-2},  \tilde{L}_{k-1-(m+1)}\rangle = (-1)^{m+1}(m+3)$.
      
      We now want to mutate $\tilde{L}_{k-1-(m+1)}$ past everything on its left: $$\mathbb{L}^{(m)}\tilde{L}_{k-1-m}, \mathbb{L}^{(m-1)}\tilde{L}_{k-m},\dots, \mathbb{L}^{(1)}\tilde{L}_{k-2}, \tilde{L}_{k-1}, \tilde{L}_{k-1-(m+1)}, \dots$$

      By the above remark, the term of $\mathbb{L}^{(i)}\tilde{L}_{k-1-i}$ in $\mathbb{L}^{(m+1)}\tilde{L}_{k-1-(m+1)}$ will be given by the following sum:$$\sum_{i>i_1>\dots i_s>0} (-1)^s \langle \mathbb{L}^{(i)}\tilde{L}_{k-1-i}, \mathbb{L}^{(i_1)}\tilde{L}_{k-1-i_1}\rangle \dots \langle \mathbb{L}^{(i_s)}\tilde{L}_{k-1-i_s},  \tilde{L}_{k-1-(m+1)} \rangle$$
      Therefore, we must consider all possible ways to get from $\mathbb{L}^{(i)}\tilde{L}_{k-1-i}$ to $\tilde{L}_{k-1-(m+1)}$. For a composable path of maps to contribute a nonzero term to the sum above, it must traverse either $\tilde{L}_{k-1}$ or $\mathbb{L}^{(1)}\tilde{L}_{k-2}$ penultimately.
    
      That is, we need to consider all ways to compose morphisms starting from $\mathbb{L}^{(i)}\tilde{L}_{k-1-i}$ and ending at either $\mathbb{L}^{(1)}\tilde{L}_{k-2}$ or $\tilde{L}_{k-1}$. The coefficient of $\mathbb{L}^{(i)}\tilde{L}_{k-1-i}$ appearing in $\mathbb{L}^{(m+1)} \tilde{L}_{k-1-(m+1)}$ will then be the signed, weighted sum of all these composable paths. By the inductive hypothesis, the $\mathbb{L}^{(i)}\tilde{L}_{k-1-i}$ satisfy \begin{gather*}
      \langle \mathbb{L}^{(i)}\tilde{L}_{k-1-i}, \mathbb{L}^{(i-1)}\tilde{L}_{k-1-(i-1)}\rangle =2, \qquad\langle \mathbb{L}^{(i)}\tilde{L}_{k-1-i}, \mathbb{L}^{(i-2)}\tilde{L}_{k-1-(i-2)}\rangle = 1, \\
      \langle \mathbb{L}^{(i)}\tilde{L}_{k-1-i}, \mathbb{L}^{(j)}\tilde{L}_{k-1-j}\rangle=0, \,i-j>2
      \end{gather*} Hence, starting from $\mathbb{L}^{(i)}\tilde{L}_{k-1-i}$ we are allowed either a jump to $\mathbb{L}^{(i-1)}\tilde{L}_{k-i}$ weighted by $2$ or a jump to $\mathbb{L}^{(i-2)}\tilde{L}_{k-i+1}$ weighted by $1$. Consider the following quiver, weighing morphisms from $i$ to $i-1$ by $2$ and from $i$ to $i-2$ by $1$:
      
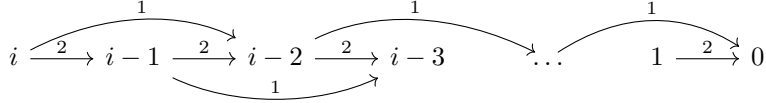
\[\begin{tikzcd}
      i & {i-1} & {i-2} & {i-3} & \dots & 1 & 0
      \arrow[from=1-1, to=1-2, "2"]
      \arrow[curve={height=-18pt}, from=1-1, to=1-3, "1"]
      \arrow[from=1-2, to=1-3, "2"]
      \arrow[curve={height=18pt}, from=1-2, to=1-4, "1"]
      \arrow[from=1-3, to=1-4, "2"]
      \arrow[curve={height=-18pt}, from=1-3, to=1-5, "1"]
      \arrow[curve={height=-18pt}, from=1-5, to=1-7, "1"]
      \arrow[from=1-6, to=1-7, "2"]
      \end{tikzcd}\]

      We denote the weighted sum of all possible paths from $i$ to $0$ by $s_i$. It is clear that these $s_i$ satisfy the recursion $$s_i=2s_{i-1}-s_{i-2},s_1=2, s_2=3$$hence this signed sum is $s_i={i+1}$.

      The relationship between this quiver and the total coefficient of $\mathbb{L}^{(i)}\tilde{L}_{k-1-i}$ in $\mathbb{L}^{(m+1)}\tilde{L}_{k-1-(m+1)}$ is the following:
      \begin{itemize}
        \item firstly, we need to consider all composable morphisms from $\mathbb{L}^{(i)}\tilde{L}_{k-1-i}$ to $\tilde{L}_{k-1}$, which are then concatenated by $\langle \tilde{L}_{k-1}, \tilde{L}_{k-1-(m+1)}\rangle = (-1)^{m+1}(m+2)$. These contribute $s_i(-1)^{m+1}(m+2)$
        \item secondly, we need to consider all composable morphisms from $\tilde{L}_{k-1-i}$ to $\mathbb{L}^{(1)}\tilde{L}_{k-2}$ which at the end are concatenated with $\langle \mathbb{L}^{(1)}\tilde{L}_{k-2}, \tilde{L}_{k-1-(m+1)}\rangle = (-1)^{m+1}(m+3)$. These contribute $s_{i-1}(-1)^{m+1}(m+3)$ 
      \end{itemize}
      Adding these up, we get a contribution of $$(-1)^i(-1)^m((m+2)(i+1)-(m+3)i)=(-1)^{m+i}(m+2-i)$$

      Therefore, \[
      [\,\mathbb{L}^{(m+1)}\tilde{L}_{k-1-(m+1)}]
      = [ \tilde{L}_{k-1-(m+1)}]+\sum_{i=0}^m(-1)^{m+i}(m+2-i)[ \mathbb{L}^{(i)}\tilde{L}_{k-1-i}]=\]
      \begin{multline*}
      =(-1)^{m+1}\bigg((m+1)l+l_{k-1-(m+1)}- \\
      \sum_{i=2}^m (m+2-i)(l_{k-1-i}-2l_{k-i}+l_{k-i+1}) - \\ 
      (m+1)(l_{k-2}-2l_{k-1}+l)-(m+2)l_{k-1}\bigg)=
      \end{multline*}
      \[=(-1)^{m+1}(l_{k-1-(m+1)}-2l_{k-1-m}+l_{k-1-(m-1)})\]
      The final equality is because the $l$ terms cancel, and moreover all the terms $l_i,i<m-1$ will telescope and cancel as well.

      \end{proof}

      \begin{rem}\label{rem:localKoszul}
        The reader should think of the relation $[L_{k-1-i}]=l_{k-1-i}-2l_{k-i}+l_{k-i+1}$ as related to the local Koszul resolution $$[\mathcal{O}\otimes \rho_{k-i+1}\rightarrow \mathcal{O}^{\oplus 2}\otimes \rho_{k-i}\rightarrow \mathcal{O}\otimes \rho_{k-1-i}]\simeq \mathcal{O}_0\otimes \rho_{k-1-i}$$
      \end{rem}
      The classes $l_i-l_{i+1}$ are represented by arcs connecting adjacent branch points. This means that $L_i$ is represented by a curve that projects to an arc which joins together branch points labelled by $i$ and $i+2$ as in Figure \ref{fig:projectionsofleftduals}. Thinking of the branch points as handle attachments on the Riemann surface $\Sigma^0_*$, the result is the same one described in Construction \ref{constr:Lefschetz}, namely the vanishing cycle $L_{k-1-i}$ is built out of the cores of the handles $H^-_{k-1-i}$, two copies of $H^-_{k-i}$ and the core of $H^-_{k+1-i}$.
      \begin{figure}[H]
      \minipage{0.48 \textwidth}
      \centering
        \begin{tikzpicture}[scale=0.15]
      \definecolor{DarkGreen}{RGB}{1,50,32}
      \def\n{3}

      \foreach \k in {0,...,\numexpr\n-1} {
          \node[circle,DarkGreen,fill,inner sep=1pt] (e'\k)
            at ({360*\k/\n}:8) {};
        }
  
        \node [circle, draw, red, inner sep = 1.5pt] (zero) at (0,0) {};
        
        \node[circle,fill,DarkGreen,inner sep = 1pt] (m) at (2.5,0) {};
        \node[rectangle,fill,DarkGreen,inner sep = 2pt, label=above:\tiny $0$] (m') at (4.7,0) {};
        \node[rectangle,fill,DarkGreen,inner sep = 2pt, label=above:\tiny $1$] (M') at (5.7,0) {};

        \draw[thick,red]
         (m') 
          to[out=-70, in=-90,looseness=1.4] ($(e'0)+  (-0.8,0) $)
          to[out=90, in=180]  ($(e'0) +(0.2,0.5)$)
          to[out=0, in=-20]  node[midway,right] {$L_3$}  (e'2)
         ;
        \draw[DarkGreen,thick, bend right=90, looseness=4,opacity=1] (e'1) to node[midway,left] {$L_2$} (e'0);
        \draw[blue!50,thick] (e'0) to node[midway,above] {\tiny $L_4$} (M');

        \foreach \k  [evaluate=\k as \kk using int(\n+1-\k)]  in {0,...,\numexpr\n-1} {
          \node[circle,DarkGreen,fill,inner sep=1pt,
                label=above:\tiny $\kk$] ()
            at ({360*(-\k)/\n}:8) {};
        }
      \end{tikzpicture}
          \endminipage
          \hfill
          \minipage{0.48\textwidth}%
          \centering
          \usetikzlibrary{backgrounds}
      \pgfdeclarelayer{base}
      \pgfsetlayers{base,background,main}

        \begin{tikzpicture}[scale=0.7,
          gap/.style={draw, very thick},
          band/.style={fill=white, draw=black, line width=0.6pt},
          line cap=round
        ]

        \def\R{1}
        \def\s {1/32*360}
        \def \epsilon{1/1.9}
        \colorlet{blue50}{blue!50}

        \begin{pgfonlayer}{base}
        \fill[gray!80] (0,0) circle (\R);
        \end{pgfonlayer}
        
        \draw[very thick] (-1,0) ellipse (3cm and 2cm);
        \fill[gray!40, even odd rule] 
        (-1,0) ellipse (3cm and 2cm)
          (0,0) circle (\R-2.5/100);

        \fill[gray!80] (-3.5,0) circle (0.1);
        \draw[black, thick] (-3.5,0) circle (0.1);

        \begin{scope}[rotate=-90]
        \foreach \i/\a/\b in {1/ \s/3*\s, 2/5*\s/7*\s, 3/9*\s/11*\s, 4/13*\s/15*\s} {
          \path (\a:\R) coordinate (rg\i a);
          \path (\b:\R) coordinate (rg\i b);
          \path ($(rg\i a)!0.5!(rg\i b)$) coordinate (rg\i);
        }

        \foreach \i/\a/\b in {1/- \s/-3*\s, 2/-5*\s/-7*\s, 3/-9*\s/-11*\s, 4/-13*\s/-15*\s} {
          \path (\a:\R) coordinate (lg\i a);
          \path (\b:\R) coordinate (lg\i b);
          \path ($(lg\i a)!0.5!(lg\i b)$) coordinate (lg\i);
        }

        \begin{pgfonlayer}{background}
        \path[band, fill=gray!40, draw=none] (rg1a) -- (lg4a) -- (lg4b) -- (rg1b) -- cycle;
        \draw[very thick] (rg1a) -- (lg4a);
        \draw[very thick] (rg1b) -- (lg4b);
        
        \draw[thick, blue50] (rg1)--(lg4);
        \draw[thick, red, opacity=1,line cap=round] 
          ($(lg4a)!0.66!(lg4b)$)--($(rg1a)!0.66!(rg1b)$);
        \draw[thick, red, opacity=1,line cap=round] 
          ($(lg4a)!0.33!(lg4b)$)--($(rg1a)!0.33!(rg1b)$);

        \path[band, fill=gray!40, draw=none] (rg2a) -- (lg3a) -- (lg3b) -- (rg2b) -- cycle;
        \draw[very thick] (rg2a) -- (lg3a);
        \draw[very thick] (lg3b) -- (rg2b);
        
        \draw[thick, red, opacity=1,line cap=round] (lg3)--(rg2);

        \path[band, fill=gray!40, draw=none] (rg3a) -- (lg2a) -- (lg2b) -- (rg3b) -- cycle;
        \draw[very thick] (rg3a) -- (lg2a);
        \draw[very thick] (lg2b) -- (rg3b);
      \end{pgfonlayer}


          \path[band,fill=gray!40,draw=none] (rg4a) -- (lg1a) -- (lg1b) -- (rg4b) -- cycle;
            \draw[very thick] (rg4a) -- (lg1a);
          \draw[very thick] (lg1b) -- (rg4b);

        \foreach \k in {-1,3,7,11,15,19,23,27,31} {
          \draw[very thick] (\k*\s:\R) arc[start angle=\k*\s, end angle=(\k+2)*\s, radius=\R];
        }

        \draw[thick,blue50]  ($(lg1a)!0.5!(lg1b)$)--($(rg4a)!0.5!(rg4b)$);
        \draw[blue50, thick,shorten <=0, shorten >=0] (rg4) to[out=-180,in=-180] (lg4);
        \draw[blue50, thick,shorten <=0, shorten >=0] (rg1) to[out=0,in=0] (lg1);
        \draw[red, thick,line cap=round] ($(lg4a)!0.33!(lg4b)$) to[out=-3000,in=-120] (lg3);
        \draw[red, thick,shorten <=0, shorten >=0] (rg2) to[out=50,in=50] ($(rg1a)!0.66!(rg1b)$);
        \end{scope}

        \path[name path=A,xshift=-70,yshift=70,rotate=90] (-3.215,-0.0583) .. controls (-2.7968,-0.3531) and (-2.2716,-0.3738) .. (-1.85,-0.0574);

        \path[name path=B,xshift=-70,yshift=70,rotate=90] (-3.215,-0.0583) .. controls (-2.6,0.3) and (-2.3,0.3) .. (-1.85,-0.0574);
        \node (v1) at (-0.3342,0.9518) {};
        \node (v2) at (0.2425,-0.873) {};

        \tikzfillbetween[of=A and B, name=filledArea]{white};
        \draw[very thick,xshift=-70,yshift=70,rotate=90] (-3.31,-0.1207) .. controls (-2.6,0.3) and (-2.3,0.3) .. (-1.7341,-0.1288);\draw[very thick, xshift=-70,yshift=70,rotate=90] (-3.2,-0.0583) .. controls (-2.7968,-0.3531) and (-2.2716,-0.3738) .. (-1.8835,-0.0674);

        \draw[red, thick] (-0.3111,0.8999) .. controls (-0.5138,1.5287) and (-2.1427,0.8128) .. (-2.1702,0.0737);
        \draw[dashed,red, thick,opacity=0.5] (-2.1625,0.0277) .. controls (-1.7367,-0.7741) and (-1.1819,-1.319) .. (0.061,-1.8635);
        \draw[red, thick] (0.1022,-1.8621) .. controls (0.5352,-1.7393) and (0.4828,-1.4808) .. (0.3102,-0.9132);
        \node[fill,circle,inner sep = 1pt] at (-0.3544,0.9788) {};
        \node[fill,circle,inner sep = 1pt] at (0.3346,-1.0088) {};
        \end{tikzpicture}
          \endminipage
          \caption{\label{fig:projectionsofleftduals} Left: the projections of the $L_i$ in our reference fiber $\Sigma^0_*$ when $k=5$. Right: the surface $\Sigma^0$ described as a torus with three boundary components, to which three handles have been attached, with $L_4$ in blue and $L_3$ in red. }
      \end{figure}

    \subsubsection{The vanishing cycles associated to critical values of types \uppercase\expandafter{\romannumeral 2} and \uppercase\expandafter{\romannumeral 3}}
    We now consider the vanishing cycles associated to the $k+1$ critical values near $-1$ and the three critical values near $0$.

    When $t$ is near the critical points around $-1$, there are still $k-2$ large branch points, one close to $0$ and the twin branch points are close to $-1$. The only branch points that can collide here are the twin branch points (this is most easily seen by adjusting $\epsilon \rightarrow 0$ at which point the branch point equation has a double root at $y=-1$). In other words, all these $k+1$ vanishing cycles are homologous and have class $l$ in homology.

    Finally, when $y$ is close to the origin, the branch point equation is approximated by the cubic $$4s(1+\epsilon)=y(y-t)^2$$
    One can use computer algebra to see how the roots of this cubic vary with $t$. 
    
    The vanishing cycles of type \uppercase\expandafter{\romannumeral 2} and  \uppercase\expandafter{\romannumeral 3} are unaffected by the $k-2$ branch points of type \uppercase\expandafter{\romannumeral 1} (as in \ref{prop:branchpoints}), and hence they sit inside a torus inside $\Sigma^0$ before the $k-2$ handles are attached (like the special torus of Section 1). Defining $a$ and $b$ to be the meridian and longitude classes as depicted in Figure \ref{fig:vanishingpathsresolution}, we conclude:

    \begin{prop}
    The vanishing cycles of type \uppercase\expandafter{\romannumeral 2} and  \uppercase\expandafter{\romannumeral 3} are given in homology by:
    \begin{itemize}
      \item $[ B_1]=[ B_2]=\dots=[ B_{k+1}]=l$
      \item $[ P_{-1}]=b-l-2a$
      \item $[ P_0]=b$
      \item $[ P_1]=b+l+2a$
    \end{itemize}
    \end{prop}
    \begin{figure}[H]
      \noindent
    \minipage{0.48 \textwidth}
    \centering
    \begin{tikzpicture}[scale=1.4]
    \definecolor{DarkGreen}{RGB}{1,50,32}
    \node [rectangle, fill, DarkGreen, inner sep = 1.5pt] (v3) at (1,0) {};
    \node [rectangle, fill, DarkGreen, inner sep = 1.5pt] (v1) at (1.8,0) {};

    \node [circle, draw, red, inner sep = 1.5pt] (v5) at (0,0) {};
    \node [circle, fill, DarkGreen, inner sep = 1.5pt] (v7) at (0.3,0) {};

    \draw [dashed] (v5) edge (v7);
    \draw [dashed] (v3) edge node[midway,above=10pt] {$l$} (v1);

    \draw  (1.4,0) ellipse (0.8 and 0.2);

    \draw[blue,thick,opacity=0.5]  (0.15,0) ellipse (0.3 and 0.2);
    \node[blue,thick,opacity=0.5,above=10pt] at (0.15,0) {$a$} {};

    \draw[red,thick,opacity=0.5][] (1.4,0) arc[start angle=0, end angle=180, x radius=0.6, y radius=0.2];
    \draw[red,thick,,opacity=0.5][dotted] (1.4,0) arc[start angle=0, end angle=-180, x radius=0.6, y radius=0.2];
    \node[red,thick,opacity=0.5,above=10pt] at (0.75,0) {$b$} {};

    \end{tikzpicture} 
    \endminipage
    \hfill
    \minipage{0.48\textwidth}%
    \centering
    \hspace*{-2cm}
    \begin{tikzpicture}[scale=1.4]
      \definecolor{DarkGreen}{RGB}{1,50,32}

      \node [rectangle, fill, DarkGreen, inner sep = 1.5pt] (v3) at (1,0) {};
      \node [rectangle, fill, DarkGreen, inner sep = 1.5pt] (v1) at (1.8,0) {};

      \node [circle, draw, red, inner sep = 1.5pt] (v5) at (0,0) {};
      \node [circle, fill, DarkGreen, inner sep = 1.5pt] (v7) at (0.3,0) {};

      \draw[thick,black]  (v3) edge node[midway,above] {$P_0$} (v7);

      \draw [thick,DarkGreen,opacity=0.75](1.793,-0.0081) .. controls (1.2444,1.0103-0.1) and (0.2847,-0.7406) .. (-0.6274,-0.1364) .. controls (-1.1068,0.218) and (-0.4606,0.7217) .. (0.3028,-0.0015);

      \draw [thick,red,opacity=0.75](1.793,0.0081) .. controls (1.2444,-1.0103-0.1) and (0.2847,0.7406) .. (-0.6274,0.1364) .. controls (-1.1068,-0.218) and (-0.4606,-0.7217) .. (0.3028,0.0015);

      \node[red] at (1.3962,-0.4646) {$P_1$} {};

      \node[DarkGreen] at (1.3962,0.4646) {$P_{-1}$} {};

      \node [rectangle, fill, DarkGreen, inner sep = 1.5pt] (v3) at (1,0) {};
      \node [rectangle, fill, DarkGreen, inner sep = 1.5pt] (v1) at (1.8,0) {};

      \node [rectangle, fill, DarkGreen, inner sep = 1.5pt] (v7) at (0.3,0) {};
      \end{tikzpicture} 
    \endminipage
    \caption{\label{fig:vanishingpathsresolution} The twin branch points, together with the missing branch point at $0$ and the one close to the origin are depicted on the left, together with the projections of representatives of the classes $a,b,l$. Here, the dotted lines describe branch cuts, with $a$ and $l$ contained inside one of the sheets and $b$ going in between the two sheets of the branched double cover. Arcs describing the projections of the topological vanishing cycles are described on the right.} 
    \end{figure}

    We can mutate $P_{-1}$ past $P_0$, after which it becomes $[\tilde{P}]=2b+2a+l$. In the mirror duality, we will identify $P_0, \tilde{P}, P_1$ with the sheaves $\Phi\mathcal{O},\Phi \mathcal{T}(-H),\Phi \mathcal{O}(H)$.

    \begin{prop}\label{prop:abstractandexplicit}
      The Lefschetz fibration $\mathbf{f}_s:M^0\rightarrow \mathbb{C}$, when $M^0$ is equipped with the exact symplectic structure, coincides with the abstract Lefschetz fibration from \ref{constr:Lefschetz} associated to the pair $(\mathcal{X}, \mathcal{D})$ obtained from taking $\mathbb{P}^2$ with its toric boundary, blowing up iteratively $k+1$ times a point on $z=0$ and contracting the strict transform of $z=0$. 
    \end{prop}
    \begin{proof}
    This follows from our description of the vanishing cycles in the previous section, as well as Remark \ref{rem:localKoszul}. 
    \end{proof}
    \subsubsection{A set of $2$-cycles associated to the vanishing cycles}

    In the computation of the Fukaya-Seidel category, it will be important to weigh the holomorphic disk contributions by their symplectic area. In order to do that, it will be helpful to relate $H_2(M^0)$ to the vanishing cycles.

    \begin{defn}\label{def:2cycles}
    Let $\Delta_{i,j}$ be a $2$-chain in $\Sigma^0_*$ which witnesses the equality $[B_i]=[B_j]\in H_1(\Sigma^0_*)$. We cap this off with vanishing thimbles, producing a $2$-cycle$$S_{i,j}:=\Delta_{i,j}-D_{B_i}+D_{B_j}$$
    In other words, this is a matching sphere between two critical values corresponding to the vanishing cycles $B_i$ and $B_j$.
    Similarly, let $C$ be a $2$-chain witnessing the equality $[P_0]+[P_{1}]=[\tilde{P}]$. This has an associated $2$ -cycle $$\overline{C}:=C-D_{P_0}-D_{P_1}+D_{\tilde{P}}$$ 
    \end{defn}

    \begin{rem}
      This class is in fact homologous to the image of the generator of $H_2((\mathbb{C}^\times)^2)$ (see \cite[Lemma 4.9]{auroux_mirror_2004}) which will lead to a non-commutative deformation of the del Pezzo surface $\mathcal{X}$. We will not concern ourselves with this in this paper.
    \end{rem}

\section{Homological mirror symmetry for $X_{k+1}$ and a mirror map}\label{sec:Lagvc}
In this section, we compute the Fukaya-Seidel category (as in Definition \ref{def:FukayaCategory}) of vanishing thimbles associated to the vanishing paths from the previous section. We show an equivalence (in families) between the derived category of $\mathcal{X}$ and this symplectic category, by constructing an explicit mirror map.

On the A-side, we take the Lefschetz fibration $\mathbf{f}_s:M^0\rightarrow \mathbb{C}$ and equip it with the symplectic form as in Definition \ref{def:symplecticform}, with the non-exactness supported near the chain of spheres $S_{1,2},\dots, S_{k,k+1}$ of Definition \ref{def:2cycles}. To ensure the Fukaya-Seidel category is well-defined, we follow \cite[Section 3]{auroux_mirror_2004}: while the total space is non-exact, each fiber is an exact symplectic manifold. Moreover, $\pi_2(\Sigma^0)=\pi_2(\Sigma^0, L)=0$ for each Lagrangian vanishing cycle $L$, which prevents bubbling phenomena. The fiber is a punctured curve, so $c_1(\Sigma^0_*)=0$ hence there is a $\mathbb{Z}$-grading. Finally, to ensure that symplectic parallel transport and hence the Lagrangian vanishing cycles are well-defined, it is enough to show that $(M^0,\omega)$ is complete and moreover $|\nabla \mathbf{f}_s|$ is bounded from below outside of a compact subset (this is a Palais-Smale type condition, see \cite[Section 6]{seidel_suspending_2009}). We verify this in the appendix \ref{prop:PalaisSmale}. The Lagrangian vanishing cycles come equipped with the unique non-trivial spin structure.
 \subsubsection{B-fields and sign conventions}
  To consider not just a real-valued deformation of the symplectic structure, but more generally a complex one, we will attach a B-field to $M^0$. This is a closed $2$-form on $M^0$. We extend the objects of the Fukaya-Seidel category to be Lagrangian vanishing cycles $L$, coupled with a connection on the trivial line bundle on $L$ whose curvature is $-2\pi iB$. Each holomorphic disk with Lagrangian boundary conditions appearing in the $A_\infty$ operations will appear with the following weight: $$(-1)^{v(u)}\mathrm{exp}(2\pi i\int_{D^2} u^*(B+i\omega)) \mathrm{hol} (\partial u)$$
  Furthermore, the curvature condition implies that $$\mathrm{hol}_{\nabla}(L)=\mathrm{exp}(-2\pi i \int_{\mathbf{D}_L} B)$$where $\mathbf{D}_L$ is the thimble associated to a Lagrangian vanishing cycle $L$.
  The sign $v(u)$ can be computed by using the same sign convention as in \cite{auroux_mirror_2006} by picking a marked point on each Lagrangian vanishing cycle which is distinct from the intersection points. For the $L_i$, the point we pick is a lift of the midpoint as in Figure \ref{fig:McKayBranchedProjection} in such a way so that all these lifts live on the same sheet of the covering.

  The B-fields that we consider will be sums of Thom forms supported near the same $-2$ spheres as in the definition of the symplectic forms $\omega$.

\subsubsection{The Lagrangian vanishing cycles}
While we have determined the vanishing cycles topologically in \ref{sec:Lagvc}, we need to take some extra care in describing them symplectically, due to the non-exactness of the symplectic form. Begin with $(M^0, \omega^{ex})$, with $\omega^{ex}=dx\wedge d\overline{x}+d \log y\wedge d\log \overline{y}+dz \wedge d\overline{z}$. In this situation, $\omega^{ex}$ is anti-invariant under complex conjugation, so by the same argument as in \cite[Section 4.2]{auroux_mirror_2004}, the Lagrangian vanishing cycles are not only smoothly but also Hamiltonian isotopic to the curves which project to arcs connecting two branch points, as described in the previous Section \ref{sec:LGmodel}. In particular, when working with $\omega^{ex}$, all of the $B_i$ are Hamiltonian isotopic to each other, and have two intersections points of degree $0$ and $1$. We want to get rid of this phenomenon, so as to achieve a strong exceptional collection. To this end, we deform slightly the symplectic form to $\omega^{\underline{\epsilon}}$ as in Definition \ref{def:symplecticform}, the effect of which is to introduce symplectic area between $B_i$ and $B_j$ and make them Hamiltonian displaceable from each other, but otherwise keeps the rest of the Lagrangian vanishing cycles unaffected (since this deformation is only supported near a collection of matching $-2$ spheres for the Lagrangian vanishing cycles $B_i$). 

\subsection{Grading and morphism spaces}
  The grading data for $\Sigma^0$, as well as the Lagrangians and the intersection points between them,  depends on the choice of a line field, in other words a section of $\mathbb{P}T_{\Sigma^0}$. We choose the horizontal line field in the planar description of the surface $\Sigma^0$, as in Section \ref{sec:LCSL}.

  \begin{defn}\label{def:prelimgradingdata} We define preliminary grading data $\alpha^{pre}$ as the collection of lifts of phase functions:
 \begin{itemize}
  \item $\alpha^{pre}_{L_i}:L_i\rightarrow \mathbb{R}$ such that $\alpha_{L_i}\equiv 2$ at the start and end of the $i$-th handle $H^-_{-i}$, $\alpha_{L_i}\equiv 1$ at the start and end of the $i+1$-st handle $H^-_{-(i+1)}$, and  $\alpha_{L_i}\equiv 0$ at the start and end of the $i+2$-th handle $H^-_{-(i+2)}$.
  \item $\alpha^{pre}_{\mathcal{V}}: \mathcal{V}_i\rightarrow \mathbb{R}$ (for $\mathcal{V}\in\{P_0,\tilde{P},P_1, B_1, \dots, B_{k+1}\}$) to be the unique lift of the phase function which is $0$ on the portion of the Lagrangian $\mathcal{V}_i$ which is on $H^-_{0}$.

  \end{itemize}
  Recall that the grading of the intersection point between two Lagrangians $c\in CF(\Lambda_1 ,\Lambda_2)$, as shown by Seidel in \cite[Example 11.20]{seidel_fukaya_2008}, is given by \begin{equation}\label{eq:Seidelgrading}
\lfloor \alpha_{\Lambda_2}(c)-\alpha_{\Lambda_1}(c)\rfloor+1\end{equation}
\begin{defn} We name the generators of the morphism spaces:
   \begin{gather*}
    CF(L_i,L_{i+1})=\mathbb{C}p_{1,i}\oplus \mathbb{C}p_{2,i}, \quad CF(L_i,L_{i+2})=\mathbb{C}g_{i} \\
    CF(P_0, \tilde{P})=\mathbb{C}x_0 \oplus \mathbb{C}y_0 \oplus \mathbb{C}z_0,\, CF(\tilde{P}, P_1)=\mathbb{C}x_1\oplus \mathbb{C}y_1 \oplus \mathbb{C}z_1,\, CF(P_0, P_1)=\mathbb{C}x \oplus \mathbb{C}y \oplus \mathbb{C}z\\
    CF(P_0, B_i)=\mathbb{C}a_i,\qquad CF(\tilde{P},B_i)=\mathbb{C}b_i\oplus \mathbb{C}b'_i,\qquad  CF(P_1, B_i)=\mathbb{C}c_i\\
    CF(L_{k-1}, P_0)=0, \qquad CF(L_{k-2}, P_0)=\mathbb{C} \epsilon'\\
      CF(L_{k-1}, \tilde{P})=\mathbb{C} \tilde{\epsilon}, \qquad CF(L_{k-2}, \tilde{P})=\mathbb{C} \tilde{\epsilon}'\oplus \mathbb{C} \tilde{x}^E \oplus \mathbb{C} \tilde{y}^E\\
      CF(L_{k-1}, P_1)= \mathbb{C} \epsilon ,\qquad CF(L_{k-2}, P_1)=\mathbb{C} x_1^E\oplus \mathbb{C} y_1^E\\
      CF(L_{k-1}, B_i)=\mathbb{C} \epsilon_i, \qquad CF(L_{k-2}, B_i)=\mathbb{C} \epsilon'_i
  \end{gather*}
  \end{defn}
\begin{figure}[H]
      \input{./TikzFiles/gradingpicture.tikz}
         \caption{\label{fig:gradingpicture} Left: the Lagrangians $L_j, L_{j+1},L_{j+2}$ and two shaded holomorphic disks, with only the handle $H^-_{-(j+2)}$ drawn.  Right: the Lagrangians $L_{k-2}, L_{k-1}$ with only the handles $H^-_0, H^-_{-1}$ drawn.}
      \end{figure}

  \end{defn}
      \begin{figure}[H]
      \minipage{0.48 \textwidth}
      \centering
          \usetikzlibrary{backgrounds}
      \pgfdeclarelayer{base}
                \pgfsetlayers{base,background,main}
\begin{tikzpicture}[transform shape,
        gap/.style={draw, very thick},
        band/.style={fill=white, draw=black, line width=0.6pt},
        line cap=round
      ]
\begin{pgfonlayer}{background}
        \draw[thick,dashed] (-1,12/10) to (1,-12/10);
        \draw[thick, dashed] (-1,16/10) to (1,-8/10) ;
        
      \path[band, draw=none] (-1,4/10) -- (-1,8/10) -- (1,-0/10) -- (1,-4/10)  -- cycle;
        \draw[thick,dashed] (-1,4/10) to (1,-4/10);
        \draw[thick, dashed] (-1,8/10) to (1,-0/10) ;

      \path[band, draw=none] (-1,-4/10) -- (-1,0) -- (1,8/10) -- (1,4/10)  -- cycle;
        \draw[thick,dashed] (-1,-4/10) to (1,4/10);
        \draw[thick, dashed] (-1,0) to (1,8/10) ;

      \end{pgfonlayer}
      
      \path[band, draw=none] (-1,-2+8/10) -- (1,-2+32/10) -- (1,-2+36/10) -- (-1,-2+12/10)  -- cycle;
        \draw[very thick] (-1,-2+12/10) to (1,-2+36/10) ;
        \draw[very thick] (-1,-2+8/10) to (1,-2+32/10);

      \draw[very thick] (-2,-1.6) node (v2) {}  to (-1,-1.6);
      \draw[very thick] (2,-1.6)  to (1,-1.6);
      \draw[very thick] (-2,2.8) node (v1) {} to (-1, 2.8) to (-1,2.4)  to (1,2.4) to (1,2.8) to (2,2.8);
      \draw[very thick] (1,2)  to (-1,2) ;


      \draw[very thick] (-1,-2+4/10)  to (-1,-2+8/10);
      \draw[very thick] (-1,-2+12/10)  to (-1,-2+16/10);
      \draw[very thick] (-1,-2+20/10)  to (-1,-2+24/10);
      \draw[very thick] (-1,-2+28/10)  to (-1,-2+32/10);
      \draw[very thick] (-1,-2+36/10)  to (-1,-2+40/10);
      \begin{scope}[xshift=2cm]
      \draw[very thick] (-1,-2+4/10)  to (-1,-2+8/10);
      \draw[very thick] (-1,-2+12/10)  to (-1,-2+16/10);
      \draw[very thick] (-1,-2+20/10)  to (-1,-2+24/10);
      \draw[very thick] (-1,-2+28/10)  to (-1,-2+32/10);
      \draw[very thick] (-1,-2+36/10)  to (-1,-2+40/10);
      \end{scope}

\draw[ green!50!black] (-1,-2+10/10) to (1,-2+34/10) ;
\draw[green!50!black] (-1,2+2/10) to (1,2+2/10) ;
\draw[green!50!black] (1,2+2/10) to[in = 50, out =0] (1,-2+34/10);
\draw[green!50!black] (-1,2+2/10) to[in = 270, out =180] (-1-2/10,2+8/10);
\draw[green!50!black] (-1,-2+10/10) to[in = 90, out =240] (-1-2/10,-1.57);

\begin{scope}[xshift=-3, yshift=-2]
\draw[ purple!50!black] (-1,-2+10.5/10) to (1,-2+34.5/10) ;
\draw[purple!50!black] (-1,2+2/10) to (1,2+2/10) ;
\draw[purple!50!black] (1,2+2/10) to[in = 50, out =0, looseness=1.3] (1,-2+34.5/10);
\draw[purple!50!black] (-1,2+2/10) to[in = 270, out =180] (-1-6/10,2+8.5/10);
\draw[purple!50!black] (-1,-2+10.5/10) to[in = 90, out =230] (-1-3/10,-1.54);
\draw[purple!50!black] (-1-3/10,2.85) to[in = 180, out =270, looseness=0.7] (1.4,2.35);
\draw[purple!50!black] (1.4,2.35) to[in = 90, out=0, looseness=0.4] (1.6,-1.54);
\draw[purple!50!black] (1.6, 2.85) to[in = 180, out = 270] (2.1, 2.3);
\draw[purple!50!black] (-1.9,2.3) to[in = 90, out=0, looseness=0.5] (-1.6,-1.54);
\end{scope}

      \draw[very thick, postaction={decorate},
          decoration={markings, mark=at position .5 with {\arrow{>>}}}] (-2,-1.6) to (-2,2.8);
      \draw[very thick,postaction={decorate},
          decoration={markings, mark=at position .5 with {\arrow{>>}}}] (2,-1.6) to (2,2.8);
      \draw[very thick,postaction={decorate},
          decoration={markings, mark=at position .5 with {\arrow{>}}}] (-2,-1.6) node (v2) {}  to (-1,-1.6);
      \draw[very thick,postaction={decorate},
          decoration={markings, mark=at position .5 with {\arrow{<}}}] (2,-1.6)  to (1.2,-1.6);
      \draw[very thick,postaction={decorate},
          decoration={markings, mark=at position .5 with {\arrow{>}}}] (-2,2.8) node (v2) {}  to (-1,2.8);
      \draw[very thick,postaction={decorate},
          decoration={markings, mark=at position .5 with {\arrow{<}}}] (2,2.8)  to (1.2,2.8);

\node[fill, black, circle, inner sep =1] at (-1.1928,2.4365) {};
\draw (-1.7,0) arc (-20:360:0.1);
\node[green!50!black] at (1.25,1.5) {$B_i$};
\node[purple!50!black] at (-1.3,1.5) {$P_1$};
\end{tikzpicture}
          \endminipage
          \hfill
          \minipage{0.48\textwidth}%
          \centering

          \usetikzlibrary{backgrounds}
      \pgfdeclarelayer{base}
                \pgfsetlayers{base,background,main}
\begin{tikzpicture}[transform shape,
        gap/.style={draw, very thick},
        band/.style={fill=white, draw=black, line width=0.6pt},
        line cap=round
      ]
\begin{pgfonlayer}{background}
        \draw[thick,dashed] (-1,12/10) to (1,-12/10);
        \draw[thick, dashed] (-1,16/10) to (1,-8/10) ;
        
      \path[band, draw=none] (-1,4/10) -- (-1,8/10) -- (1,-0/10) -- (1,-4/10)  -- cycle;
        \draw[thick,dashed] (-1,4/10) to (1,-4/10);
        \draw[thick, dashed] (-1,8/10) to (1,-0/10) ;

      \path[band, draw=none] (-1,-4/10) -- (-1,0) -- (1,8/10) -- (1,4/10)  -- cycle;
        \draw[thick,dashed] (-1,-4/10) to (1,4/10);
        \draw[thick, dashed] (-1,0) to (1,8/10) ;

      \end{pgfonlayer}
      
      \path[band, draw=none] (-1,-2+8/10) -- (1,-2+32/10) -- (1,-2+36/10) -- (-1,-2+12/10)  -- cycle;
        \draw[very thick] (-1,-2+12/10) to (1,-2+36/10) ;
        \draw[very thick] (-1,-2+8/10) to (1,-2+32/10);

      \draw[very thick] (-2,-1.6) node (v2) {}  to (-1,-1.6);
      \draw[very thick] (2,-1.6)  to (1,-1.6);
      \draw[very thick] (-2,2.8) node (v1) {} to (-1, 2.8) to (-1,2.4)  to (1,2.4) to (1,2.8) to (2,2.8);
      \draw[very thick] (1,2)  to (-1,2) ;


      \draw[very thick] (-1,-2+4/10)  to (-1,-2+8/10);
      \draw[very thick] (-1,-2+12/10)  to (-1,-2+16/10);
      \draw[very thick] (-1,-2+20/10)  to (-1,-2+24/10);
      \draw[very thick] (-1,-2+28/10)  to (-1,-2+32/10);
      \draw[very thick] (-1,-2+36/10)  to (-1,-2+40/10);
      \begin{scope}[xshift=2cm]
      \draw[very thick] (-1,-2+4/10)  to (-1,-2+8/10);
      \draw[very thick] (-1,-2+12/10)  to (-1,-2+16/10);
      \draw[very thick] (-1,-2+20/10)  to (-1,-2+24/10);
      \draw[very thick] (-1,-2+28/10)  to (-1,-2+32/10);
      \draw[very thick] (-1,-2+36/10)  to (-1,-2+40/10);
      \end{scope}


\draw[ purple!50!black] (-1,-2+10.5/10) to (1,-2+34.5/10) ;
\draw[purple!50!black] (-2,2+1/10) to (1,2+1/10) ;
\draw[purple!50!black] (1,2+1/10) to[in = 50, out =0, looseness=1.3] (1,-2+34.5/10);
\draw[purple!50!black] (-1,-2+10.5/10) to[in = 90, out =230] (-1-3/10,-1.57);
\draw[purple!50!black] (-1-3/10,2.82) to[in = 180, out =270, looseness=0.85] (1.4,2.35);
\draw[purple!50!black] (1.4,2.35) to[in = 90, out=0, looseness=0.4] (1.6,-1.6);
\draw[purple!50!black] (1.6, 2.8) to[in = 180, out = 270] (2, 2.3);
\draw[purple!50!black] (-2, 2.2) to (1.3, 2.2) to[in=90, out=0, looseness =0.4] (1.4, -1.6);
\draw[purple!50!black] (1.4, 2.8) to[in=180, out=270, looseness=0.8] (2,2+1/10);

\draw[gray] (-2,2.05) to (2,2.05);

      \draw[very thick, postaction={decorate},
          decoration={markings, mark=at position .5 with {\arrow{>>}}}] (-2,-1.6) to (-2,2.8);
      \draw[very thick,postaction={decorate},
          decoration={markings, mark=at position .5 with {\arrow{>>}}}] (2,-1.6) to (2,2.8);
      \draw[very thick,postaction={decorate},
          decoration={markings, mark=at position .5 with {\arrow{>}}}] (-2,-1.6) node (v2) {}  to (-1,-1.6);
      \draw[very thick,postaction={decorate},
          decoration={markings, mark=at position .5 with {\arrow{<}}}] (2,-1.6)  to (1.3,-1.6);
      \draw[very thick,postaction={decorate},
          decoration={markings, mark=at position .5 with {\arrow{>}}}] (-2,2.8) node (v2) {}  to (-1,2.8);
      \draw[very thick,postaction={decorate},
          decoration={markings, mark=at position .5 with {\arrow{<}}}] (2,2.8)  to (1.3,2.8);

\node[circle, fill, inner sep =1] at (1.1484,2.0503) {};
\node[circle, fill, inner sep =1] at (1.4773,2.0406) {};
\node[circle, fill, inner sep =1] at (1.6385,2.0439) {};
\draw (-1.7,0) arc (-20:360:0.1);
\node[purple!50!black] at (-1.3,-1) {$\tilde{P}$};
\node[gray] at (-1.3,1.7) {$P_0$};
\end{tikzpicture}
          \endminipage
          \caption{ Left: description of the Lagrangians $[P_1]=b+2a+l, [B_i]=l$. Right: depiction of the Lagrangians $[P_0]=b, [\tilde{P}]=2b+2a+l$. Note that we have $k-2$ extra handle bridges (in this case $k=5$) which are completely independent of these Lagrangians - the Lagrangians mirror to $\Phi D^b Y$ only use the special handles $H^-_0$ and $H^-_{-1}$}\label{fig:resolutionplanar}
      \end{figure}
We choose the labelling so that (after removing branch cuts) $p_1, \tilde{x}^E,x_1^E,x, y_0, y_1$ are all on one sheet of the double cover description of $\Sigma^0$ and $p_2, \tilde{y}^E, y_1^E,y, x_0, x_1$ are all on the other.
\begin{prop} With the grading data $\alpha^{pre}$, the intersection points between $P_0, \tilde{P}, P_1, B_i$ all have degree zero. Moreover, $$|p_{1,i}|^{\alpha^{pre}}=|p_{2,i}|^{\alpha^{pre}}=1, |g_i|^{\alpha^{pre}}=2$$
The intersection points between $L_{k-1}$ and $P_0, \tilde{P},P_1, B_i$ all have degree $1$:$$|\tilde{\epsilon}|^{\alpha^{pre}}=|\epsilon|^{\alpha^{pre}}=|\epsilon_i|^{\alpha^{pre}}=1 $$ and similarly the intersections between $L_{k-2}$ and $P_0, \tilde{P},P_1, B_i$ have degree $2$:
$$|\epsilon'|^{\alpha^{pre}}=|\tilde{\epsilon}'|^{\alpha^{pre}}= |\tilde{x}^E|^{\alpha^{pre}}=|\tilde{y}^E|^{\alpha^{pre}}=|x_1^E|^{\alpha^{pre}}=|y_1^E|^{\alpha^{pre}}=|\epsilon_i'|^{\alpha^{pre}}=2$$

\end{prop}

\begin{proof}   
Combined with the data of the grading \ref{def:prelimgradingdata} and Figure \ref{fig:resolutionplanar}, the result follows by using Seidel's formula \ref{eq:Seidelgrading}. 
\end{proof}

Just as on the B-side, we can shift the phase functions of the Lagrangians $L_i$ by $[k-i]$. We call this modified grading data $\alpha$, with which the collection becomes strong: \label{modifiedgrading}
  \begin{cor}
    With the modified grading data $\alpha$, the morphism spaces for the Lagrangian vanishing cycles $L_i, P_0, \tilde{P},P_1, B_i$ are concentrated degree $0$. Therefore, the only nontrivial $A_\infty$ operation is $\mu^2$.
  \end{cor}

  \subsection{The mirror to the algebra of the McKay quiver}
    We start by describing the algebra of the Lagrangians $L_2, \dots, L_{k-1}$.

    \begin{prop}\label{prop:localalgebras}
      The $A_\infty$ algebra of the Lagrangians $L_2, \dots, L_{k-2}$ is equivalent to the dg algebra of the orbifold skyscrapers $e_2, \dots, e_{k-1}$.
    \end{prop}
    \begin{proof}
    Recall the Floer cycles involved in the computation:$$CF^0(L_i, L_{i+1})=\mathbb{C}p_{1,i}\oplus \mathbb{C}p_{2,i},\qquad CF^0(L_i, L_{i+2})=\mathbb{C}g_i$$
    We claim that there are exactly two disks, contributing with opposite signs (which can be seen visualized in Figure \ref{fig:gradingpicture}). The surface $\Sigma^0$ is given by a branched double cover of $\mathbb{C}^\times$ with $k$ branch points. The Lagrangian vanishing cycles project to arcs connecting two branch points, as in \ref{fig:McKayBranchedProjection}. 
      
      \begin{figure}[H]
      \begin{tikzpicture}[scale=2]

      \definecolor{DarkGreen}{RGB}{1,50,32}
      \usetikzlibrary{intersections}

      \coordinate (v1) at (-1.5186,3.1534) {};
      \coordinate (v3) at (-0.7788,3.4122) {};
      \coordinate (v2) at (0,3.5) {};
      \coordinate (v4) at (0.7788,3.4122) {};
      \coordinate (v5) at (1.5186,3.1534) {};

      \path[name path=arc1] (v1) to[bend left=80] (v2);
      \path[name path=arc2] (v2) to[bend left=80] (v5);
      \path[name path=arc3] (v3) to[bend left=80] (v4);

      \path[name intersections={of=arc1 and arc3, by=i1}];
      \path[name intersections={of=arc2 and arc3, by=i2}];

      \fill[gray!90!DarkGreen]
        (i1) 
          to[bend left=30] (v2)
          to[bend left=40] (i2)
          to[bend right=20] (i1)
        -- cycle;
      \fill[gray!90!DarkGreen]
        (i2) 
          to[bend right=30] ($(v2)$)
          to[bend right=40] (i1)
          to[bend left=20] (i2)
        -- cycle;
        
      \draw[DarkGreen,bend left=80, thick] (v1) to (v2);
      \draw[DarkGreen!50!red,bend left=80, thick] (v3) to (v4);
      \draw[DarkGreen!50!blue,bend left=80,thick] (v2) to (v5);

      \node [DarkGreen,circle, fill, inner sep = 1pt] (v1) at (v1) {};
      \node [DarkGreen,circle, fill, inner sep = 1pt,label=below:$g_i$] (v2) at (v2) {};
      \node [DarkGreen,circle, fill, inner sep = 1pt] (v3) at (v3) {};
      \node [DarkGreen,circle, fill, inner sep = 1pt] (v4) at (v4) {};
      \node [DarkGreen,circle, fill, inner sep = 1pt] (v5) at (v5) {};

      \node [circle, label={[font=\tiny, label distance = -3pt]above:$p_{1,j},p_{2,j}$}] (i1) at (i1) {};
      \node [circle, label={[font=\tiny,label distance = -3pt]above:$p_{1,j+1},p_{2,j+1}$}] (i2) at (i2) {};

      \end{tikzpicture}
      \caption{ \label{fig:McKayBranchedProjection} Projection of the Lagrangians $L_j, L_{j+1}, L_{j+2}$ under the branched double covering.}
    \end{figure}
    Since a disk in $\Sigma^0$ projects to a disk in $\mathbb{C}^\times$ with the projected boundary conditions, we see that it must necessarily be a lift of the shaded triangle in Figure \ref{fig:McKayBranchedProjection}. There are exactly two lifts of this triangle. Moreover, by picking a reference marker point which is a lift of the mid point in the projection of $L_j$ (resp. $L_{j+1}, L_{j+2}$), we see that the two disks contributing to $\mu^2(p_{1,j},p_{1,j+1})$ and $\mu^2(p_{2,j},p_{j+1})$ have opposite signs. There are constants associated to the symplectic area that need to be accounted for. The two disks produce two compositions $\mu^2(p_{1,i},p_{1,i+1})=\vartheta_{1,i} g_i, \mu^2(p_{2,i},p_{2,i+1})=\vartheta_{2,i}g_i$. After a rescaling  $p_{1,i}^{resc}:=-\frac{\vartheta_{2,i}}{\vartheta_{1,i}}p_{1,i}$, we can arrange that $$\mu^2(p_{1,i}, p_{1,i+1})=-\mu^2(p_{2,i}, p_{2,i+1})$$ The resulting dg algebra is clearly the same as the McKay quiver algebra (with both being formal).
  \end{proof}
  \subsection{The mirror to the derived category of $Y$}
  We now move on to compute the algebra of the rest of the Lagrangians, namely $P_0,\tilde{P},P_1$ and $B_1,\dots, B_{k+1}$. Here, $\tilde{P}$ is the result of right mutating $P_{-1}$ through $P_0$ and it satisfies $[\tilde{P}]=2b+2a+l$. In computing the compositions involving $B_i$, the non-exactness of $\omega$ will play a part, and hence it will be convenient to define some Novikov parameters.

  \begin{defn}
  We define Novikov parameters $q_{i,j}:=\exp(2\pi i [B+i\omega](S_{i,j}))$. We also define $q_C:=\exp(2\pi i [B+i\omega](\overline{C}))$. 
  \end{defn}
  We will mostly ignore the non-commutative direction and set $q_C=1$.

  \begin{prop}
  The vanishing cycles $P_0,\tilde{P},P_1$ and $B_1,\dots, B_{k+1}$ are all contained within the special torus $T^0$ which is a thrice punctured torus. In this torus, $P_0,\tilde{P},P_1$ are identified with the standard mirrors to $\mathcal{O}, \mathcal{T}(-H),\mathcal{O}(H)$ as in \cite[Proposition 3.2]{seidel_more_2001} and \cite{auroux_mirror_2006}, whereas $B_i$ are parallel, disjoint meridians which are all topologically isotopic. All holomorphic triangles between the intersection points of these vanishing cycles stay within the torus and satisfy the relations: 

  For the compositions $CF(\tilde{P},P_1)\otimes CF(P_0,\tilde{P})\rightarrow CF(P_0,P_1)$, there exists constants such that \begin{gather*}
    \mu^2(y_1,x_0)=\alpha_{x,y}z,\qquad \mu^2(x_1,y_0)=\alpha_{y,x}z\\
    \mu^2(z_1,x_0)=\alpha_{x,z}y,\qquad \mu^2(x_1,z_0)=\alpha_{z,x}y\\
    \mu^2(z_1,y_0)=\alpha_{y,z}x,\qquad \mu^2(y_1,z_0)=\alpha_{z,y}x\\
    \mu^2(x_1,x_0)=\mu^2(y_1,y_0)=\mu^2(z_1,z_0)=0
  \end{gather*}
    \begin{gather*}
      \mu^2(c_i,x)=\beta_{x,i}a_i, \qquad
      \mu^2(c_i,y)=\beta_{y,i}a_i, \qquad
      \mu^2(c_i,z)=0
    \end{gather*}
    \begin{gather*}
      \mu^2( b_i,x_0)= \mu^2(b_i,y_0)=\mu^2( b'_i,z_0)=0, \qquad \mu^2(b_i,z_0)=\gamma_{z,i}a_i, \qquad \mu^2( b'_i,x_0)=\gamma_{x,i} a_i, \qquad \mu^2( b'_i,y_0)=\gamma_{y,i} a_i
    \end{gather*}
    \begin{gather*}
      \mu^2(c_i,x_1)=\eta_{x,i}b_i, \qquad \mu^2(c_i,y_1)=\eta_{y,i}b_i,\qquad
      \mu^2(c_i,z_1)=\eta_{z,i} b'_i
    \end{gather*}

    These satisfy: \begin{gather*}
      \displaystyle\frac{\quad \frac{\beta_{x,i} }{\beta_{y,i}} \quad}{ \quad \frac{\beta_{x,j}}{\beta_{y,j}}\quad}=\displaystyle\frac{\quad \frac{\eta_{y,i} }{\eta_{x,i}} \quad}{ \quad \frac{\eta_{y,j}}{\eta_{x,j}}\quad}=\displaystyle\frac{\quad \frac{\gamma_{y,i} }{\gamma_{x,i}} \quad}{ \quad \frac{\gamma_{y,j}}{\gamma_{x,j}}\quad}=q_{i,j}, \qquad \frac{\alpha_{x,y}\alpha_{y,z}\alpha_{z,x}}{\alpha_{y,x}\alpha_{z,y}\alpha_{x,z}}=-1
    \end{gather*}

  \end{prop}

  \begin{proof}
    To show that all triangles are contained in the torus, one applies \cite[Lemma 7.5]{seidel_fukaya_2008}. The proof of the relations is essentially the same as in \cite[Section 4.3]{auroux_mirror_2006}, except that the torus there is compact, whereas our torus has three punctures hence corresponds to taking the toric boundary on $\mathbb{P}^2$ rather than a smooth elliptic curve. Moreover, our coordinate $z$ on $\mathbb{P}^2$ corresponds to their coordinate $y$. Setting $q_F=0, q_C=1$, one recovers the relations for the constants $\alpha$ and $\beta$ from \cite[Section 4.3]{auroux_mirror_2006}. For the remaining constants $\gamma$ and $\eta$, we can enumerate (topologically) the holomorphic triangles for example by using the branched double cover model: 
    \begin{figure}[H]
      \minipage{0.48 \textwidth}
      \centering
          \begin{tikzpicture}[scale=2]
      \definecolor{DarkGreen}{RGB}{1,50,32}

      \node [rectangle, fill, DarkGreen, inner sep = 2pt,label=above:$a_i \,b_i$] (v3) at (1,0) {};
      \node [rectangle, fill, DarkGreen, inner sep = 2pt,label=above:$b'_i$] (v1) at (1.8,0) {};

      \node [circle, red, inner sep = 1.5pt] (v5) at (0,0)  {};
      \node [circle, fill, DarkGreen, inner sep = 1.5pt,label=below:$z$] (v7) at (0.3,0) {};
        
        
      \draw[thick,DarkGreen!50!orange] (v3) edge node[midway,above=0pt] {$B_i$} (v1);
      \draw[thick,DarkGreen!50!blue]  (v3) edge node[midway,above] {$P_0$} (v7);

      \node [circle, fill, DarkGreen, inner sep = 1.5pt] (v3) at (1,0)  {};
      \node [circle, fill, DarkGreen, inner sep = 1.5pt] (v1) at (1.8,0) {};
      \node [circle, red, draw, inner sep = 1.5pt] (v5) at (0,0) {};
      \node [circle, fill, DarkGreen, inner sep = 1.5pt] (v7) at (0.3,0) {};
     \node at (0.47,-0.1) {\tiny {$\stackrel{x_0}{y_0}$}} {};
     \node at (1,-0.1) {\tiny {$z_0$}} {};
      
      \fill[DarkGreen, opacity=0.2]  plot[smooth cycle, tension=.7] coordinates {(1.7896,-0.0243) (1.738,-0.1432) (1.6774,-0.2374) (1.5854,-0.3294) (1.48,-0.4012) (1.3655,-0.4573) (1.177,-0.4977) (0.9863,-0.4798) (0.8405,-0.4371) (0.7283,-0.3608) (0.652,-0.2801) (0.6071,-0.1073) (0.5734,-0.0355) (0.5533,-0.0108) (0.5779,-0.0018) (0.9234,-0.0015) (1,0) (1.2,0) (1.5,0) (1.6,0) (1.8,0) (1.7798,-0.0429) (1.7903,-0.0232)};
     \draw[thick,DarkGreen] (1.8,0) node (v4) {} .. controls (1.563,-0.7178) and (0.6431,-0.5523) .. (0.6219,-0.1603) .. controls (0.5269,0.1774) and (0.3497,0.1091) .. (-0.0031,-0.0947) .. controls (-0.286,-0.3118) and (-0.3975,0.2053) .. (-0.0662,0.2646) .. controls (0.1841,0.3141) and (0.7383,0.2876) .. (1,0) node (v2) {};

\node[DarkGreen] at (0.9,-0.7) {$\tilde{P}$};
\end{tikzpicture}
          \endminipage
          \hfill
          \minipage{0.48\textwidth}%
          \centering
          \begin{tikzpicture}[scale=2]
          \definecolor{DarkGreen}{RGB}{1,50,32}

          \node [rectangle, fill, DarkGreen, inner sep = 1pt,label=above:$z_1$] (v1) at (1.8,0) {};
          \node [circle, DarkGreen, inner sep = 1.5pt,label=below:$ c_i$] at (1.6,0) {};
          \node [rectangle, fill, DarkGreen, inner sep = 1pt,label=below:$ \,\,b'_i$]  at (1.8,0) {};
          \node [circle, draw, red, inner sep = 1.5pt] (v5) at (0,0)  {};
          \node [circle, fill, DarkGreen, inner sep = 1.5pt,label=below:] (v7) at (0.3,0) {};
          
          \fill[DarkGreen, opacity=0.2] plot[smooth cycle, tension=.7] coordinates {(1.2,0) (1.3,0) (1.5,0) (1.7,0) (1.8,0) (1.8,0) (1.6929,-0.1483) (1.6212,-0.2255) (1.5075,-0.2921) (1.3854,-0.3114) (1.2578,-0.2895) (1.1474,-0.2405) (1.0072,-0.1568) (0.8683,-0.0585) (0.7549,0.0278) (0.6757,0.0883) (0.5752,0.1672) (0.4688,0.2436) (0.4447,0.2549) (0.4896,0.2537) (0.5702,0.2319) (0.6566,0.2054) (0.7383,0.1777) (0.8258,0.1317) (0.8961,0.0902) (0.9571,0.0373) (0.9683,0.0309) (1,0) (1.1,0)};
          
          \draw [thick,DarkGreen!50!red,opacity=0.75](1.793,0.0081) .. controls (1.2444,-1.0103-0.1) and (0.3,1.1) .. (-0.5,0.3) .. controls (-0.8,0) and (-0.2,-0.7) .. (0.3028,0.0015);
            \draw[thick,DarkGreen] (1.8,0) .. controls (1.563,-0.7178) and (0.6431,-0.5523) .. (0.6219,-0.1603) .. controls (0.5269,0.1774) and (0.3497,0.1091) .. (-0.0031,-0.0947) .. controls (-0.286,-0.3118) and (-0.3975,0.2053) .. (-0.0662,0.2646) .. controls (0.1841,0.3141) and (0.7383,0.2876) .. (1,0) node (v2) {};
          
          \node[DarkGreen!50!red] at (1.3807,-0.36) {\tiny $P_1$} {};
          \node[DarkGreen] at (1,-0.55) {\tiny $\tilde{P}$} {};
          \node at (0.6,0.3) {\tiny $x_1,y_1$} {};
          \node [circle, fill, DarkGreen, inner sep =1.5pt,label=above:] at (0.3,0) {};
          \node [rectangle, fill, DarkGreen, inner sep = 1pt,label=above:$z_1$] (v1) at (1.8,0) {};
          \node [rectangle, fill, DarkGreen, inner sep = 1pt,label=above:$b_i$] (v3) at (1,0) {};
           \draw[thick,DarkGreen!50!orange] (v3) edge node[midway,above=0pt] {$B_i$} (v1);
          \end{tikzpicture}

          \endminipage
          \label{fig:PtildBi}\caption{Left: projections of triangles contributing to compositions $P_0\rightarrow \tilde{P}\rightarrow B_i$. Right: the same, but for $\tilde{P}\rightarrow P_1 \rightarrow B_i$.}
      \end{figure}
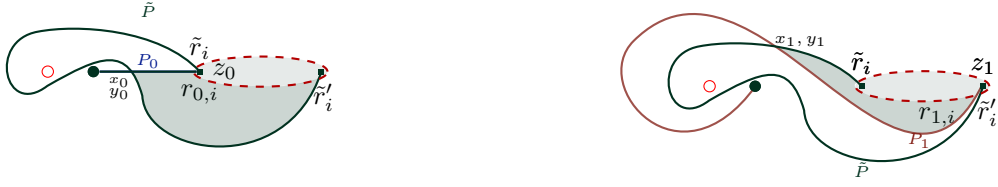
      Each of the shaded triangles has two lifts. Moreover, there is a disk which, in the topological model above, is the constant disk at $z_0$. The cross-ratios are all the same, by using the associativity of $\mu^2$:\begin{gather*}
      \mu^2(c_i,\mu^2(y_1,z_0))=\alpha_{z,y}\beta_{x,i}a_i=\mu^2(\mu^2(c_i,y_1),z_0)=\eta_{y,i}\gamma_{z,i}a_i\implies \alpha_{z,y}\beta_{x,i}=\eta_{y,i}\gamma_{z,i}
    \end{gather*}
    Similarly, \begin{gather*}
      \alpha_{y,z}\beta_{x,i}=\eta_{z,i}\gamma_{y,i}, \quad
      \alpha_{z,x}\beta_{y,i}=\eta_{x,i}\gamma_{z,i},\quad
      \alpha_{x,z}\beta_{y,i}=\eta_{z,i}\gamma_{x,i}
    \end{gather*}
    so that $$\frac{\alpha_{z,y}}{\alpha_{z,x}}\frac{\beta_{x,i}}{\beta_{y,i}}=\frac{\eta_{y,i}}{\eta_{x,i}}, \qquad
    \frac{\alpha_{y,z}}{\alpha_{x,z}}\frac{\beta_{x,i}}{\beta_{y,i}}=\frac{\gamma_{y,i}}{\gamma_{x,i}}$$
  \end{proof}

  One should interpret the compositions $\tilde{P}\rightarrow P_1\rightarrow B_i$ in the same way as on the B-side: the kernel of $\mu^2(c_i,-)$ is given by $$\mu^2(c_i, \eta_{y,i}x_1 -\eta_{x,i}y_1)=0$$ which can be used to define $k+1$ points $[\eta_{y,i}:-\eta_{x,i}:0]\in \mathbb{P}^2$. These constants are nonzero, so we can identify the $k+1$ points using the cross-ratio property as $\{-\frac{\eta_{y,1}}{\eta_{x,1}},-\frac{\eta_{y,1}}{\eta_{x,1}} q_{1,2},\dots, -\frac{\eta_{y,1}}{\eta_{x,1}}q_{1,k+1} \}\subset \mathbb{C}^\times$. 
  \subsection{The gluing morphisms and their composition law}

  We have shown that the algebra of the Lagrangians $\langle L_2,\dots, L_{k-1}\rangle$ matches the algebra of the orbifold skyscrapers $\langle e_2[2-k], \dots, e_{k-1}[-1]\rangle$. Moreover, we have shown that the algebra of $\langle P_0, \tilde{P},P_1,B_1,\dots, B_{k+1}\rangle$, depending on the cohomology class of the symplectic form $\omega$, corresponds to the algebra of $D^b(Y)=\langle D^b(\mathbb{P}^2), \mathcal{O}_{B_1},\dots, \mathcal{O}_{B_{k+1}}\rangle$ for a complex structure on $Y$ determined by the composition law in the Fukaya category. The last piece that remains is verifying that these two subcategories are glued in the same way. 
  
      \begin{prop}
      There are nonzero constants such that
        \begin{gather*}
          \mu^2(x_1,\tilde{\epsilon})=\mu^2(y_1,\tilde{\epsilon})=\mu^2(b_i,\tilde{\epsilon})=0\\
         \mu^2(z_1,\tilde{\epsilon})=\epsilon, \qquad \mu^2(b'_i,\tilde{\epsilon})=\alpha_{\tilde{\epsilon}, b'_i}\epsilon_i, \qquad \mu^2(c_i,\epsilon)=\alpha_{\epsilon, c_i}\epsilon_i
        \end{gather*}
       \begin{gather*}
          \mu^2(\tilde{\epsilon},p_{1,k-2})=\alpha_{p_1,\tilde{\epsilon}} \tilde{x}^E,\qquad \mu^2(\tilde{\epsilon},p_{2,k-2})=\alpha_{p_2,\tilde{\epsilon}} \tilde{y}^E \\
          \mu^2(\epsilon,p_{1,k-2})=\alpha_{p_1,\epsilon} x_1^E,\qquad \mu^2(\epsilon,p_{2,k-2})=\alpha_{p_2,\epsilon} y_1^E\\
          \mu^2(\epsilon_i,p_{1,k-2})=\lambda_{1,i} \epsilon'_i, \qquad \mu^2(\epsilon_i,p_{2,k-2})=\lambda_{2,i} \epsilon'_i
        \end{gather*}
        These constants satisfy $$\displaystyle\frac{\quad \frac{\lambda_{1,i} }{\lambda_{2,i}} \quad}{ \quad \frac{\lambda_{1,j}}{\lambda_{2,j}}\quad}=q_{i,j}$$
      \end{prop}
      \begin{proof}
        The proof follows again by examining the configuration of the projections of the Lagrangians in the branched double cover. The argument for the cross-ratios is identical to the one for the constants $\beta$. The disks contributing to $\mu^2(\tilde{\epsilon}, -)$ are shown below on the left, whereas the disk contributing to $\mu^2(c_i, \epsilon)$ corresponds to the constant one seen on the right\footnote{ The disk is constant when $\omega$ is exact, but need not be in the non-exact setting.  }.
      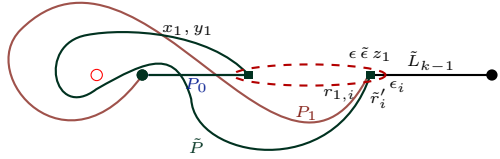
\begin{figure}[H]
              \minipage{0.48 \textwidth}
      \centering
          \begin{tikzpicture}[scale=2]
      \definecolor{DarkGreen}{RGB}{1,50,32}

      \node [rectangle, fill, DarkGreen, inner sep = 1.5pt,label=above:] (v3) at (1,0) {};
      \node [circle, draw, red, inner sep = 1.5pt] (v5) at (0,0)  {};
      \node [circle, fill, black, inner sep = 1.5pt] (v6) at (2.4,-1)  {};
      \node [circle, fill, DarkGreen, inner sep = 1.5pt,label=below:] (v7) at (0.3,0) {};   
      
      \draw[thick] (2.3886,-1.0209) .. controls (3.3935,-0.2637) and (2.3922,0.8284) .. (2.3956,0.1027) .. controls (2.3502,-0.4312) and (1.8756,-0.9266) .. (1.0034,-0.016);
	\node at (2.9852,0.0992) {$L_{k-2}$};
      \draw [thick,DarkGreen!50!red,opacity=0.75](1.793,0.0081) .. controls (1.2444,-1.0103-0.1) and (0.3,1.1) .. (-0.5,0.3) .. controls (-0.8,0) and (-0.2,-0.7) .. (0.3028,0.0015);
        \draw[thick,DarkGreen] (1.8,0) .. controls (1.563,-0.7178) and (0.6431,-0.5523) .. (0.6219,-0.1603) .. controls (0.5269,0.1774) and (0.3497,0.1091) .. (-0.0031,-0.0947) .. controls (-0.286,-0.3118) and (-0.3975,0.2053) .. (-0.0662,0.2646) .. controls (0.1841,0.3141) and (0.7383,0.2876) .. (1,0) node (v2) {};
        \draw[thick] (1.8,0) edge node[midway,above=-2.5pt] {\tiny $L_{k-1}$} (2.6,0);
        \node[circle, fill, inner sep = 1.5pt] at (2.6,0) {};
      
      \node[DarkGreen!50!red] at (1,-0.25) {\tiny $P_1$} {};
      \node [circle, fill, DarkGreen, inner sep =1.5pt,label=above:] at (0.3,0) {};
      \node [rectangle, fill, DarkGreen, inner sep = 1.5pt,label=above:] at (1,0) {};
      \node [rectangle, fill, DarkGreen, inner sep = 1.5pt,label=above:$\epsilon\, \tilde{\epsilon}$] at (1.8,0) {};
    
      \node at (2.53,-0.1) {\tiny$p_1,p_2$};
      \node at (1.45,-0.55) {\tiny$\tilde{x}^E, \tilde{y}^E$};
      \node at (1.45,-0.17) {\tiny$x_1^E, y_1^E$};
      
      \fill[DarkGreen, opacity=0.2]  plot[smooth cycle, tension=.7] coordinates {(2.3745,-0.012) (1.9637,-0.0035) (1.8272,0.0028) (1.7125,-0.1091) (1.6507,-0.2003) (1.5065,-0.2856) (1.4036,-0.3121) (1.3624,-0.315) (1.4036,-0.3562) (1.4624,-0.3915) (1.5389,-0.4327) (1.6683,-0.4945) (1.786,-0.518) (1.892,-0.518) (2.045,-0.4945) (2.1685,-0.4121) (2.2568,-0.3091) (2.3098,-0.2061) (2.3392,-0.1473) (2.3774,-0.0178)};
      \fill[DarkGreen, opacity=0.2]  plot[smooth cycle, tension=.7] coordinates {(2.3714,0.0017) (2.0535,0.0037) (1.83,-0.0041) (1.7384,-0.1233) (1.6897,-0.2109) (1.5767,-0.3286) (1.5329,-0.3624) (1.481,-0.4034) (1.5027,-0.4148) (1.5102,-0.4317) (1.5389,-0.4327) (1.6683,-0.4945) (1.786,-0.518) (1.892,-0.518) (2.045,-0.4945) (2.1685,-0.4121) (2.2568,-0.3091) (2.3098,-0.2061) (2.3392,-0.1473) (2.3774,-0.0178)};
      
      \end{tikzpicture}
          \endminipage
          \hfill
          \minipage{0.48\textwidth}%
          \centering
          \begin{tikzpicture}[scale=2]
      \definecolor{DarkGreen}{RGB}{1,50,32}
       \draw[thick,DarkGreen!50!orange] (1,0) edge node[midway,above=0pt] {$B_i$} (1.8,0);
      \node [rectangle, fill, DarkGreen, inner sep = 1.5pt,label=above:] (v3) at (1,0) {};
      \node [rectangle, fill, DarkGreen, inner sep = 1.5pt,label=below:\tiny$c_i\,\,\,\,\,\,\,\,\,\,\,\,\,\,\,\,\,\,$] (v1) at (1.8,0) {};
       \node [rectangle, fill, DarkGreen, inner sep = 1.5pt,label=below:\tiny $\,\,\,\,\,b_i'$]  at (1.8,0) {};
       \node [rectangle, DarkGreen, inner sep = 1.5pt,label=above:\tiny $\epsilon\,\tilde{\epsilon}\,z_1$]  at (1.8,0) {};
       \node [rectangle, DarkGreen, inner sep = 1.5pt,label=below:\tiny $\,\,\,\,\,\,\,\,\,\,\,\,\,\,\,\,\epsilon_i$]  at (1.8,0.05) {};
      \node [circle, draw, red, inner sep = 1.5pt] (v5) at (0,0)  {};
      \node [circle, fill, DarkGreen, inner sep = 1.5pt,label=below:] (v7) at (0.3,0) {};

      \draw [thick,DarkGreen!50!red,opacity=0.75](1.793,0.0081) .. controls (1.2444,-1.0103-0.1) and (0.3,1.1) .. (-0.5,0.3) .. controls (-0.8,0) and (-0.2,-0.7) .. (0.3028,0.0015);
        \draw[thick,DarkGreen] (1.8,0) .. controls (1.563,-0.7178) and (0.6431,-0.5523) .. (0.6219,-0.1603) .. controls (0.5269,0.1774) and (0.3497,0.1091) .. (-0.0031,-0.0947) .. controls (-0.286,-0.3118) and (-0.3975,0.2053) .. (-0.0662,0.2646) .. controls (0.1841,0.3141) and (0.7383,0.2876) .. (1,0) node (v2) {};
        \draw[thick] (1.8,0) edge node[midway,above=-2.5pt] {\tiny $L_{k-1}$} (2.6,0);
        \node[circle, fill, inner sep = 1.5pt] at (2.6,0) {};
      
      \draw[thick,DarkGreen!50!blue]  (v3) edge node[midway,below=-2.5pt] {\tiny $P_0$} (v7);
      \draw[thick,DarkGreen]  (v3) edge node[midway,below=20pt] {\tiny $\tilde{P}$} (v7);
      \node[DarkGreen!50!red] at (1.3807,-0.25) {\tiny $P_1$} {};
      \node at (0.6,0.3) {\tiny $x_1,y_1$} {};
      \node [circle, fill, DarkGreen, inner sep =1.5pt,label=above:] at (0.3,0) {};
      \node [rectangle, fill, DarkGreen, inner sep = 1.5pt,label=above:] at (1,0) {};
      \node [rectangle, fill, DarkGreen, inner sep = 1.5pt,label=above:] at (1.8,0) {};
    
    
      \end{tikzpicture}

          \endminipage \caption{Projections of the Lagrangian vanishing cycles $L_{k-1}, P_0, \tilde{P},P_1,B_i$ under the branched double cover.}
      \end{figure}
      \end{proof}

      It remains to compute the compositions of type $L_{k-2}\rightarrow M\rightarrow N$ where $M,N \in \{ P_0, \tilde{P}, P_1, B_i\}$. A lot of this has been implicitly done already: most of the morphisms in $CF(L_{k-2},M)$ can be factorized into a morphism in $CF(L_{k-2}, L_{k-1})$ followed by a morphism $CF(L_{k-1}, M)$. For these types of morphism, determining the composition reduces, by associativity, to the previous computations. The only morphisms that do not factorize are $\epsilon'\in CF(L_{k-2}, P_0)$ and $\tilde{\epsilon}'\in CF(L_{k-2}, \tilde{P})$.

      \begin{prop}
        There are nonzero constants such that \begin{gather*}
          \mu^2(x_1,\tilde{\epsilon}')=\alpha_{\epsilon, x_1} y_1^E, \qquad \mu^2(y_1,\tilde{\epsilon}')=\alpha_{\epsilon, y_1} x_1^E\\ 
          \mu^2(b_i,\tilde{\epsilon}')=\alpha_{\epsilon, b_i}\epsilon'_i, \qquad \mu^2(b'_i,\tilde{\epsilon}')=0, \qquad \mu^2(a_i,\epsilon')=\alpha_{\epsilon, a_i}\epsilon'_i\\
          \mu^2(x,\epsilon')=\alpha_{\epsilon,x} x_1^E, \qquad \mu^2(y,\epsilon')=\alpha_{\epsilon,y} y_1^E, \qquad \mu^2(z,\epsilon')=0\\
          \mu^2(x_0,\epsilon')=\alpha_{\epsilon, x_0} \tilde{y}^E, \qquad \mu^2(y_0,\epsilon')=\alpha_{\epsilon,y_0} \tilde{x}^E, \qquad \mu^2(z_0,\epsilon')=\tilde{\epsilon}' 
        \end{gather*}
      \end{prop}

      \begin{proof}
        The enumeration of the disks follows the same strategy as the previous propositions. We note that the intersection point $\tilde{\epsilon}'$ happens at the left twin branch point.
      \end{proof}

      \begin{rem}
        The fact that $\mu^2(\epsilon', x_0)$ is a multiple of $\tilde{y}^E$ and not of $\tilde{x}^E$ (and similarly for $y_0, x_1, y_1$) is due to the fact that they are on different sheets of the branched covering. On the B-side, this is related to the fact that, as in \ref{prop:quiverwithrelsB}, $\iota^*x_0=(0, \iota^*y), \iota^*x_1=(\iota^*y,0)^T$.
      \end{rem}

      \begin{cor}\label{cor:algebraofLagrangians}
        The endomorphism algebra of the exceptional collection of Lagrangian vanishing cycles $$\langle L_2, \dots, L_{k-1}, P_0, \tilde{P}, P_1, B_1, \dots, B_{k+1}\rangle$$ is generated by the elements $x_0,y_0,z_0,x_1,y_1,z_1, c_i, 1\leq i \leq k+1$ as well as $p_{1,j},p_{2,j}, 2\leq j \leq k-2$ and $\epsilon', \tilde{\epsilon}$. They satisfy the relations
        \begin{gather*}
        \mu^2(p_{1,i+1}, p_{1,i})=\frac{\vartheta_{1,i}}{\vartheta_{2,i}}\mu^2(p_{2,i+1}, p_{2,i}),\quad \mu^2(p_{1,i+1}, p_{2,i})=\mu^2(p_{2,i+1},p_{1,i})=0,\quad \mu^2(\epsilon ', p_{1,k-3})=\mu^2(\epsilon ', p_{2,k-3})=0\\
        \mu^2(y_1, x_0) = \frac{\alpha_{x,y}}{\alpha_{y,x}}\mu^2(x_1, y_0),\quad 
        \mu^2(z_1, x_0) = \frac{\alpha_{z,x}}{\alpha_{x,z}}\mu^2(x_1, z_0),\quad 
        \mu^2(z_1, y_0) = \frac{\alpha_{z,y}}{\alpha_{y,z}} \mu^2(y_1, z_0)\\
        \mu^2(x_1 , x_0) = \mu^2(y_1, y_0) = \mu^2(z_1, z_0) =  \mu^2(x_1, \tilde{\epsilon})=\mu^2(y_1, \tilde{\epsilon})= \mu^2(c_i,\eta_{y,i}x_1-\eta_{x,i}y_1)=0\\
       \mu^2(x_0, \epsilon')=\frac{\alpha_{\epsilon, x_0}}{\alpha_{p_2,\tilde{\epsilon}}} \mu^2(\tilde{\epsilon}, p_{2,k-2}),\qquad  \mu^2(y_0, \epsilon')=\frac{\alpha_{\epsilon, y_0}}{\alpha_{p_1, \tilde{\epsilon}}} \mu^2(\tilde{\epsilon}, p_{1,k-2})\\
      \displaystyle\frac{\quad \frac{\eta_{y,i} }{\eta_{x,i}} \quad}{ \quad \frac{\eta_{y,j}}{\eta_{x,j}}\quad}=q_{i,j}, \quad  \frac{\alpha_{x,y}\alpha_{y,z}\alpha_{z,x}}{\alpha_{y,x}\alpha_{z,y}\alpha_{x,z}}=-1 
      \end{gather*}
      \end{cor}

      We are now ready to complete the proof of Theorem \ref{thm:mainthm}.
      \begin{proof}[Proof of Theorem \ref{thm:mainthm}]
        The equivalence of categories follows from the previous Corollary \ref{cor:algebraofLagrangians}, as well as the description of the algebra on the B-side as in Proposition \ref{prop:quiverwithrelsB} and a suitable rescaling of the generators. An explicit choice of rescaling is the following: \begin{gather*}
        p_{1,i}^{resc}:= -\frac{\vartheta_{2,i}}{\vartheta_{1,i}}p_{1,i},\quad 2\leq i \leq k-3, \qquad
        p_{1,k-2}^{resc}:=\frac{\alpha_{p_1,\tilde{\epsilon}}}{\alpha_{\epsilon, y_0}}\frac{\alpha_{\epsilon, x_0}}{\alpha_{p_2, \tilde{\epsilon}}}p_{1,k-2}\\
        \tilde{\epsilon}^{resc}:=\frac{\alpha_{p_2,\tilde{\epsilon}}}{\alpha_{\epsilon,x_0}}\tilde{\epsilon},\quad
        x_1^{resc}:=\frac{\eta_{x,1}}{\eta_{y,1}}x_1,\quad y_1^{resc}:=-\frac{\alpha_{y,x}}{\alpha_{x,y}}\frac{\eta_{y,1}}{\eta_{x,1}}y_1,\quad z_1^{resc}:=\frac{\alpha_{y,z}}{\alpha_{z,y}}\frac{\alpha_{x,y}}{\alpha_{y,x}}\frac{\eta_{x,1}}{\eta_{y,1}}z_1
        \end{gather*} with the other generators remaining unchanged.     \end{proof}     
\newpage
\section*{Appendix}
\setcounter{section}{5}
\subsection{Special representations and modular arithmetic}
          Consider the special $I$-series $I(r,l)$ (which is defined in \cite[Definition 2.5]{gugiatti_full_2023}).
      \begin{prop}\label{propappendix:overlap}
      For any non-special $d$ with $i_t > d > i_{t-1}$, the associated $j_t$ satisfies $k+j_t<r$.
      \end{prop}
      \begin{proof}
      Since $j_{n+1}<j_n<\dots < j_0$ is increasing, it is enough to show this for the smallest non-special $d$ which has the biggest associated $j_t$.

      The smallest non-special $d$ is the least integer such that $i_n=1,i_{n-1}=2,\dots, i_{n-d+1}=d-1$ are all special, but $d$ is not. This implies that $b_n=\dots = b_{n-d+2}=2$ and moreover that $j_{n}=k, j_{n-1}=2k-r, \dots, j_{n-(d-1)}=dk-(d-1)r$. Recall the recursion $$j_{n-d+2}= b_{n-(d-1)}j_{n-(d-1)}-j_{n-d}$$

      Now, if $j_{n-(d-1)}\geq r-k$ then we would have a contradiction: it would imply that $2j_{n-d+1}-j_{n-d+2}=(d+1)k-dr\geq 0$ so that $b_{n-d+1}=2$ and $j_{n-d}=(d+1)k-dr$ and $i_{n-d}=d$ is special. Hence, $j_{n-(d-1)}< r-k$ and thus since the J series is decreasing, also $j_{n-d}<r-k$.

      \end{proof}
\subsection{Lemmata on the critical points and the dynamics of the branch points}

      \begin{prop*}[Proof of Proposition \ref{prop:criticalptsdeformation}]\label{propappendix:criticalptsdeformation}
     As $s\rightarrow 0$ (setting $P(y)=(1+y)^{k+1}+\epsilon$ for a small $\epsilon$), the critical values become distributed as follows:
    \begin{enumerate}[label=\Roman*.]
      \item $k-2$ of the critical values are distributed near the roots of $t^{k-2}-(\frac{k-2}{k})^{k-2}\frac{1}{s}=0$.
      \item $3$ of them will be arbitrarily close to $0$
      \item Another $k+1$ critical points are equidistributed near $-1$.
    \end{enumerate}
    Moreover, all of the critical points are non-degenerate and when $s$ is real, only two of the critical points are in $\mathbb{R}_+^2$, one of which is of type \uppercase\expandafter{\romannumeral 2} and the other of type \uppercase\expandafter{\romannumeral 3}.
      \end{prop*}
      \begin{proof}
        We write $$\mathbf{f}=\frac{P(y)}{xy}+sx+g(y)$$We consider the equations on the open torus. In fact, by the non-degeneracy of $P$ (for generic $\mathbf{q}_i$), there will not be any critical points on the toric boundary of $V$ so that suffices.

        We have the system $$\mathbf{f}=t, \partial_x\mathbf{f}=0, \partial_y \mathbf{f}=0$$with $t$ denoting the critical value, the first two of which imply that $$2P=xy(t-g), P=sx^2y$$ and hence that $2sx=t-g$. In particular, $$P-\frac{1}{4s} y (g-t)^2=0$$The equation $\partial_y \mathbf{f}$ tells us that $$g'(y)xy^2-P+yP'=0 \implies P=yP'+xy^2g'(y)$$
        We now set $g=y$. We see from $P-yP'=xy^2$ together with $x^2=\frac{P}{sy}$ that the critical points satisfy the polynomial in $y$: \begin{equation}\label{eq:critpointpoly}
        (P-yP')^2=\frac{1}{s}Py^3
        \end{equation}In particular, restricting to small $y$ satisfying $\frac{1}{s}\gg|y^{2k+2}|$, the roots of this are close to the roots of $Py^3$, hence there are three roots close to $y=0$ and $k+1$ roots  close to $-\mathbf{q}_1, \dots, -\mathbf{q}_{k+1}$. We will now assume that the $\mathbf{q}_i$ are close to $1$ so that $P=(1+y)^{k+1}+\epsilon$ for some small $\epsilon$. The other $k-2$ roots are equidistributed with $|y^{k-2}- \frac{1}{s}|\approx 0$.

        It remains to show that the associated critical values $t$ have a similar distribution to the solutions of Equation \ref{eq:critpointpoly}. First, $$t=2sx+y=2s\frac{P-yP'}{y^2}+y$$ Thus, when $y$ is close to $-\mathbf{q}_i$ then also $t$ is close to  $-\mathbf{q}_i$. 
        
        Secondly, notice that since $P=(1+y)^{k+1}+\epsilon$, $$xy^2=P-yP'=P\frac{1-ky}{1+y}+\epsilon (k+1)\frac{y}{y+1}=sx^2y(-k+\frac{k+1}{y+1})+\epsilon (k+1)(1-\frac{1}{y+1})$$ hence after dividing by $xy$ we get 
        \begin{equation}\label{eq:xfromy}
        y=sx(-k+\frac{k+1}{y+1})+\frac{1}{xy}(k+1)\epsilon(1-\frac{1}{y+1})
        \end{equation}Thus, when $|y^{k-2}- \frac{1}{s}|\approx 0$ and so $|y|$ is very big, then $|y+ksx|\approx 0$ and thus $ t$ is close to  $\frac{k-2}{k}y$. 
        
        Finally, when $y$ is close to $0$ then since $\frac{1}{xy}=\frac{sx}{P}$, from Equation \ref{eq:xfromy} we conclude that $|y-sx|\approx 0$ thus $t=2sx+y$ is very close to $3y$. 

        For non-degeneracy, we appeal to Kouchnirenko's formula \cite{kouchnirenko_polyedres_1976} which in our example states that the sums of the multiplicities of the critical points of $\mathbf{f}_s$ on $(\mathbb{C}^\times)^2 $ is equal to twice the volume of the Newton polytope of $\mathbf{f}_s$: $$\sum \mu_p = 2!\mathrm{vol}(\mathrm{Newt}(\mathbf{f}_s))$$
        By Pick's formula, $2\mathrm{vol}(\mathrm{Newt}(\mathbf{f}))=2i+b-2=(k+1)+( k+2+1)-2=2k+2$. Since there are $2k+2$ critical points, they must all have multiplicity $1$ i.e. they are non-degenerate. Note that we need the Laurent polynomial $\mathbf{f}_s$ to be Newton non-degenerate to apply Kouchnirenko's formula, which in turn requires $P(y)$ to have distinct roots. Otherwise, we might have a degenerate critical point at infinity. 

        Finally, we find which of the critical points are real. The polynomial in equation \ref{eq:critpointpoly} takes negative values for $y=0$ and also for $y\gg 0$. (strictly speaking, we are only considering $y\in \mathbb{C}^\times$ but for the purpose of understanding the behaviour of this polynomial, we can take $y=0$.) Moreover, it is positive for $y$ in a bounded interval $I\subset \mathbb{R}_+$, provided $\frac{1}{s}>\sup_I \frac{(P(y)-P'(y))^2}{P(y)}$. There are thus at least two real roots, but also by Descartes' rule of signs, there can be at most two, so we conclude there are exactly two. In fact, one of these is very close to $0$ and the other is very big, since they must be outside of the compact interval $I$. 
      \end{proof}
      \begin{prop}\label{propappendix:radialoutwards}
      Suppose $\frac{1}{s}\gg |t| \gg 0$ and $t$ is moving in a straight line radially outwards to a critical value at large radius. The result of this on the level of the branch points is the following: the twin branch points (labelled $0$ and $1$ in Figure \ref{fig:branch points labelled}) are moving radially outward towards one of the equidistributed branch points labelled $2,3,\dots, k-1$, without changing their relative position. This continues until the twin branch point with the larger absolute value collides with it.
      \end{prop}

      \begin{proof}
        We summarize the proof in three steps:
        \begin{itemize}
          \item We approximate and replace our polynomial by a simpler one in the form $y_0^k-(y_0-t_0)^2$
          \item We notice the new polynomial has a $\mathbb{Z}_{k-2}$-equivariant property so that we can restrict to only studying the case when $t$ is real.
          \item We use a Sturm sequence argument to show that when $t$ is real and increases towards the critical value, the two twin branch points will live on the real line (or arbitrarily close to it), moving towards to a far away branch point which is also arbitrarily close to being real. The smaller of the twin branch points remains less than the bigger one, whereas the bigger of the twins hits the far away branch point as we move $t$ radially to the critical point.
        \end{itemize}

        Consider again the branch point equation: $(1+y)^{k+1}+\epsilon=P=\frac{1}{4s}y(y-t)^2$. 
        Divide everything by $r^{k+1}$ where $r^{k-2}=\frac{1}{4s}$, resulting in:$$(\frac{y}{r}+\frac{1}{r})^{k+1}+\frac{\epsilon}{r^{k+1}}=\frac{y}{r}(\frac{y}{r}-\frac{t}{r})^2$$

        Denote $y_0:=\frac{y}{r}, t_0:= \frac{t}{r}$. The roots of this polynomial are arbitrarily closely approximated by the roots of $$y_0^{k+1}=y_0(y_0-t_0)^2$$since we can take $s$ as small as we want. Hence, we will have a root very close to $0$, and $k-2$ other roots close to the roots of \begin{equation}\label{eq:infinitesimalapprox}
          y_0^k=(y_0-t_0)^2
        \end{equation}

        We notice the following property of this polynomial: if $(y_0,t_0)$ is a solution, then $(\zeta_{k-2}y_0, \zeta_{k-2}t_0)$ is also a solution, where $\zeta_{k-2}^{k-2}=1$. We will show that for a given range of real $t_0$, the twin branch points (or their infinitesimal approximations) are collinear. By a Sturm sequence argument we will show that \ref{eq:infinitesimalapprox} has three real roots, two of them corresponding to the 'twin' branch points and one corresponding to one of the equidistributed branch points. We will show that, as $t_0$ increases, the bigger twin branch point will hit the far away real one. By the $\zeta_{k-2}$ property, the real case suffices, as it will imply the case when $t_0$ is in any of the rays $\zeta_{k-2}\mathbb{R}_+$.

        So, as our first step, fix $t_0$ real and let $h=y_0^k-(y_0-t_0)^2$. We will show this has three real roots. Firstly, $h'(y_0)=ky_0^{k-1}-2(y_0-t_0)>0$ for $ y_0\in (0,t_0]$ so it is strictly increasing. Moreover, $h(0)<0, h(t_0)>0$ so there is exactly one real root here and it happens just before $t_0$. This is the stable root - the one that does not collide with another branch point as we increase $t_0$.

        Now let's write the Sturm sequence for $h$: $$h_0=h, h_1=h'=ky_0^{k-1}-2y_0+2t_0, h_2=\frac{k-2}{k}y_0^2-2t_0\frac{k-1}{k}y_0+t_0^2$$
        We continue the Sturm sequence by finding the remainder $h_3=Ay_0+B$ in the division $$h_1=qh_2-h_3$$

        Since the roots of $h_2$ are $t_0, \frac{k}{k-2}t_0$, then $$A=-\frac{h_1(t_0\frac{k}{k-2})-h_1(t_0)}{t_0\frac{2}{k-2}}=-\frac{k-2}{k}\bigg(t_0^{k-2}k((\frac{k}{k-2})^{k-1}-1)-2\frac{k}{k-2}+2\bigg)$$which can also be written as $$-\bigg(t_0^{k-2}k(\frac{\lambda^{k-1}-1}{\lambda-1})-2\bigg), \lambda = \frac{k}{k-2}$$
        and $$B=-\frac{t_0\frac{k}{k-2}h_1(t_0)-t_0h_1(t_0\frac{k}{k-2})}{\frac{2}{k-2}t_0}=-2t_0+\frac{k^2}{2}((\frac{k}{k-2})^{k-2}-1)(t_0)^{k-1}$$
        We have that $$h_0(0)=-t_0^2<0,h_1(0)=2t_0>0,h_2(0)=t_0^2>0$$Moreover, $$h_3(0)=B$$which is negative for $(1/t_0)^{k-2}>(\frac{k}{k-2}^{k-2}-1)\frac{k^2}{4}$ and positive if the opposite holds. The double point occurs when there is equality: we denote this by $t_{double}$, at which $h_3(0)=0$. In particular, for $$\frac{1}{t_0}>\frac{k}{k-2}(\frac{k^2}{4})^{1/k-2}$$ the inequality $h_3(0)<0$ holds. It moreover implies that $A>0, B<0$.

        By Descartes rule of signs, the number of real roots of $y_0^k-(y_0-t_0)^2$ is either $3$ or $1$. For $t_0\in (0, t_{double})$ we can observe the Sturm sequence: $$\begin{gathered}
        h_0(0)<0, h_1(0)>0, h_2(0)>0, h_3(0)<0,h_4=?
        \end{gathered}$$
        By Sturm's theorem, the number of real roots is equal to the number of sign changes in the above sequence, which is at least 2. We conclude it is equal to 3 for $t_0\in (0, t_{double})$. Thus, we see that there are always three real roots with multiplicity in this range, and the double root occurs when the two bigger real roots (namely, the bigger of the two twin branch points and the far away real branch point) join together, as claimed.
      \end{proof}

      \begin{prop}
      Starting at a value of $t$ such that $\frac{1}{s}\gg |t| \gg 0$, and then rotating it along the path $t_\theta = te^{-i\theta}, \theta\in [0,\frac{2\pi}{k-2}]$, the two twin branch points will interchange position, with their phase difference changing by approximately $\pi+\frac{2\pi}{k-2}$. Hence, the monodromy around such a circle will interchange the twin branch points $k$ times.
      \end{prop}
      \begin{proof}
      Consider again the infinitesimal approximation $h=y_0^k-(y_0-t_0)^2$. We showed that when $t_0$ is on the ray $\mathbb{R}_{\geq 0} \zeta_{k-2}$, there are exactly three roots of this polynomial on this ray: the twin branch points and one farther away. As such, we can consider the ratios $y_0/t_0$ where $y_0$ is one of these three roots - this is a real quantity.

      Centering at $y=t_0$, there are Puiseux expansions for the roots closest to $t_0$ (the 'twin branch points'), namely $$y_0=t_0\pm t_0^{k/2}+\frac{k}{2}t_0^{k-1}+\dots$$ For $t_0<1$ real, the ratios are $1\pm t_0^\frac{k-2}{2}+O(t_0^{k-1})$ - one slightly above 1 and one slightly below. As we rotate $t_0$ by $-\frac{2\pi}{k-2}$, the twin branch points will flip because of the square root producing the $\pm$ sign ambiguity: the one that had $y_0/t_0>1$ will now have $y_0/t_0<1$ and vice versa.

      The outcome is that for each $1/(k-2)$ rotation, the two twin branch points switch, in fact in a specified manner: the initially bigger branch point goes over the smaller one.
      \end{proof}

      \subsection{The Palais-Smale condition}

      \begin{prop}{\label{prop:PalaisSmale}}
        The symplectic manifold $(M^0, \omega)$ is complete and the gradient of $\mathbf{f}_s$ is bounded from below outside of a compact subset, for distinct values of $\mathbf{q}_i$ and $s>0$ real.
      \end{prop}
      \begin{proof}
        When the $\mathbf{q}_i$ are distinct, we can realize $M^0$ as a smooth hypersurface $\{zx=P(y)\}\subset \mathbb{C}^2\times \mathbb{C}^\times$. The fact that $M^0$ is complete is due to the fact that $(M^0, \omega^{ex})$ is complete, since it is just a hypersurface in $\mathbb{C}^2\times \mathbb{C}^\times$ with a complete metric, and $\omega$ coincides with $\omega^{ex}$ outside of a compact subset.  We verify the Palais-Smale condition by brute force: we need to show that $|\nabla \mathbf{f}_s|^2$ is bounded from below outside of a compact subset, so we will again use the metric induced by $\omega^{ex}$. We have that $$\nabla^{M^0} \mathbf{f}_s = \nabla \mathbf{f}_s - \frac{\langle \nabla \mathbf{f}_s,\nabla g \rangle}{|\nabla g|^2} \nabla g$$ where $\nabla$ denotes the gradient in the ambient $\mathbb{C}^2\times \mathbb{C}^\times$ and $g={xz-P(y)}$ is the function defining $M^0$. As such, $$|\nabla^{M^0}\mathbf{f}_s|^2=|\nabla \mathbf{f}_s|^2-\frac{\langle \nabla \mathbf{f}_s, \nabla g \rangle}{|\nabla g|^2}$$
        which is given explicitly as the non-negative quantity \begin{equation}\label{eq:explicitgradient}|s|^2+\frac{1}{|y|^2}+|y-\frac{z}{y}|^2-\frac{|s\overline{z}+\frac{\overline{x}}{y}  - |y|^2 (1-\frac{z}{y^2})\overline{P'(y)}|^2}{|x|^2+|z|^2+|yP'(y)|^2}\geq 0$$In particular, for $s=0$ we have that $$\frac{1}{|y|^2}+|y-\frac{z}{y}|^2-\frac{|\frac{\overline{x}}{y}  - |y|^2 (1-\frac{z}{y^2})\overline{P'(y)}|^2}{|x|^2+|z|^2+|yP'(y)|^2}\geq 0\end{equation}

        We want to show the gradient \ref{eq:explicitgradient} (for $s\neq 0$) is bounded from below outside of a compact set, so we may take $x,z\gg0$, in other words we look at $M^0\setminus (M^0\cap \mathbb{D}\times \mathbb{D} \times A)$ where $A$ is a compact annulus. Since $xz=P(y)$, this implies that also $y\gg 0$. We compute: \begin{multline}
          |s|^2+\frac{1}{|y|^2}+|y-\frac{z}{y}|^2-\frac{|s\overline{z}+\frac{\overline{x}}{y}  - |y|^2 (1-\frac{z}{y^2})\overline{P'(y)}|^2}{|x|^2+|z|^2+|yP'(y)|^2}\geq \\ |s|^2+\big (\frac{1}{|y|^2}+|y-\frac{z}{y}|^2-\frac{|\frac{\overline{x}}{y}  - |y|^2 (1-\frac{z}{y^2})\overline{P'(y)}|^2}{|x|^2+|z|^2+|yP'(y)|^2}\big)+\big(\frac{|\frac{\overline{x}}{y}  - |y|^2 (1-\frac{z}{y^2})\overline{P'(y)}|^2}{|x|^2+|z|^2+|yP'(y)|^2}-\frac{|s\overline{z}+\frac{\overline{x}}{y}  - |y|^2 (1-\frac{z}{y^2})\overline{P'(y)}|^2}{|x|^2+|z|^2+|yP'(y)|^2}\big)\geq \\
          |s|^2 +\underbrace{\big(\frac{|\frac{\overline{x}}{y}  - |y|^2 (1-\frac{z}{y^2})\overline{P'(y)}|^2 - |s\overline{z}+\frac{\overline{x}}{y}  - |y|^2 (1-\frac{z}{y^2})\overline{P'(y)}|^2}{|x|^2+|z|^2+|yP'(y)|^2}\big)}_{\text{ bounded in norm by} \frac{1}{2}|s|^2}\geq \frac{1}{2}|s|^2
        \end{multline}
        The last inequality is due to the fact that, for $|x|,|y|,|z|\gg 0$, we have by the reverse triangle inequality: $$|\frac{|\frac{\overline{x}}{y}  - |y|^2 (1-\frac{z}{y^2})\overline{P'(y)}|^2 - |s\overline{z}+\frac{\overline{x}}{y}  - |y|^2 (1-\frac{z}{y^2})\overline{P'(y)}|^2}{|x|^2+|z|^2+|yP'(y)|^2}|^2\leq |s|^2 \frac{|z|^2}{|x|^2+|z|^2+|yP'(y)|^2}$$ and finally $$\frac{1}{|\frac{x}{z}|^2+1+|\frac{yP'(y)}{z}|^2}\leq \frac{1}{2} \iff |\frac{yP'(y)}{z}|^2+|\frac{x}{z}|^2\geq 1$$Since $xz=P$, then $|\frac{yP'(y)}{z}|^2=|x|^2|\frac{yP'(y)}{P(y)}|^2$ which for $|x|,|y|\gg 0$ is asymptotically given by $|x|^2\gg 0$ because $yP'(y)$ and $P(y)$ are both monic polynomials of degree $k+1$. We can conclude that the gradient is bounded from below,
      \end{proof}



\newpage
\sloppy
\printbibliography
\end{document}